\documentclass{amsart}
\usepackage{graphicx}
\usepackage{amsfonts}
\usepackage{amsthm} 
\usepackage{amsmath}
\usepackage{amssymb}
\usepackage{amscd}
\usepackage{mathrsfs}
\usepackage{epstopdf}
\usepackage{tikz-cd}
\usepackage{xcolor}
\usepackage{hyperref}
\usepackage{tikz}
\usepackage{tikz-cd} 
\usetikzlibrary{cd}
\usepackage{framed}
\usepackage[T1]{fontenc}
\usepackage{bm}
\usepackage{verbatim}

\newtheorem{theorem}{Theorem}[section]
\newtheorem{letterthm}{Theorem}

\newtheorem{lemma}[theorem]{Lemma}
\newtheorem{proposition}[theorem]{Proposition}
\newtheorem{definition}[theorem]{Definition}
\newtheorem{cor}[theorem]{Corollary}

\newtheorem{remark}[theorem]{Remark}

\DeclareFontFamily{U}{wncy}{}
    \DeclareFontShape{U}{wncy}{m}{n}{<->wncyr10}{}
    \DeclareSymbolFont{mcy}{U}{wncy}{m}{n}
    \DeclareMathSymbol{\Sh}{\mathord}{mcy}{"58} 

\title[Automorphic symbols and {\it L}-values of GL(2)]{Automorphic symbols and automorphic {\it L}-values of GL(2) over imaginary quadratic fields}
\author{Jaesung Kwon}
\email{js.kwon@unist.ac.kr}
\address{Ulsan National Institute of Science and Technology, 50, UNIST-gil, Eonyang-eup, Ulju-gun, Ulsan, Republic of Korea, 44919}
                         
\date{\today}

\begin{document}

\begin{abstract}
We study non-vanishing modulo primes of critical values of $L$-functions
over imaginary quadratic fields twisted by Coates--Wiles characters.  
For a parallel weight two Hecke eigenform over an imaginary quadratic field, 
we prove that a positive proportion of Coates--Wiles characters have nonzero
integral $L$-values modulo each prime in a positive-density set.  
The argument constructs automorphic symbols in parabolic homology with integral coefficients
and expresses the critical values through their pairings with parabolic cohomology classes. 
A vertical family of additive averages of these pairings recovers Fourier coefficients 
and forces the module generated by symbols to have full rank.  
We transfer this full-rank property to non-vanishing modulo primes.
This extends the homological strategy of Kim--Sun~\cite{KimSun:2017} 
from classical modular curves to arithmetic orbifolds, 
while addressing the unit obstructions specific to the imaginary quadratic setting.
\end{abstract}

\subjclass[2010]{11F67 11F12}

\maketitle
\numberwithin{equation}{section}
\setcounter{tocdepth}{1}
\tableofcontents

\section{Introduction}
\subsection{Historical Background}
Critical values of $L$-functions are expected to encode arithmetic information.
For modular forms, Shimura~\cite{Shimura:1976} proved that, after division by suitable periods,
critical $L$-values become algebraic and satisfy a Galois equivariance property.  
Deligne~\cite{Deligne:1979} subsequently formulated 
a general conjectural framework for critical values of motives: 
at a critical point, the period-normalized value of a motivic $L$-function 
should belong to its coefficient field and satisfy the corresponding Galois equivariance property.
We refer to these two properties together as the \emph{rationality} of critical $L$-values.
We call a period-normalized critical value an \emph{algebraic $L$-value}.
For cohomological automorphic forms on $\mathrm{GL}_2$ over number fields,
Hida~\cite{Hida:1994} constructed periods from their cohomological rational structures
and proved rationality results for algebraic $L$-values.

The arithmetic meaning of such algebraic $L$-values
is made more precise by the conjectures of Bloch--Kato~\cite{BlochKato:1990}. 
They relate the leading term of a motivic $L$-function at the critical point, 
normalized by Tamagawa numbers and regulators,
to the size of the Tate--Shafarevich group of the motive. 
The Birch--Swinnerton-Dyer conjecture~\cite{BSD:1965} provides a prototypical example: 
for an elliptic curve, it relates the leading term at \(s=1\) to its regulator, 
Tamagawa numbers, torsion subgroup, and Tate--Shafarevich group.

To study algebraic $L$-values modulo a prime, algebraicity alone is not enough:
one must also control their integrality at that prime.
For a fixed prime, Vatsal~\cite{Vatsal:1999} constructed canonical periods,
determined up to units in the local coefficient ring, such that the corresponding
algebraic $L$-values are locally integral and congruences between them are well defined.

At a critical point, complex non-vanishing means that the order of vanishing is zero, 
so the leading coefficient is the critical value itself.  
Non-vanishing modulo a prime of an integral algebraic $L$-value is
stronger: it implies complex non-vanishing and, through a conjectural
leading-term formula, places additional prime-dependent restrictions
on the associated arithmetic invariants.

Complex non-vanishing in $p$-adic families of characters has been studied extensively.
Rohrlich~\cite{Rohrlich:1984} proved that, for an elliptic curve, the critical
values of its twists by characters in a cyclotomic $\mathbb{Z}_p$-tower are nonzero
for all but finitely many characters. 
The role of character twists in determining modular forms was
developed by Luo--Ramakrishnan~\cite{LuoRamakrishnan:1997}, who showed that 
a modular form is determined by the critical values of its $L$-function twisted by
cyclotomic characters. Their theorem extracts more
information from these values than complex non-vanishing alone.

Non-vanishing mod $p$ has also been studied in $\ell$-adic families of characters.
Here \(\ell\) plays the role of \(p\) in the character towers considered above, 
while \(p\) now denotes the residual characteristic.
Washington~\cite{Washington:1978} proved the non-$p$-divisibility of Dirichlet $L$-values and related it
to the ray class groups in cyclotomic $\mathbb Z_\ell$-extensions.  
In anticyclotomic families, Hida~\cite{Hida:2004} and Vatsal~\cite{Vatsal:2003} proved non-vanishing mod $p$ 
for Hecke and modular $L$-values, respectively.
More recently, Kim and Sun~\cite{KimSun:2017} studied cyclotomic twists of modular $L$-values 
through the homological behaviour of modular symbols.

\subsection{The setting and main result}
Let $F$ be an imaginary quadratic field.
Fix an odd rational prime $\ell$ which splits in $F$, and a prime $\mathfrak l$ of $F$ 
lying above $\ell$. Coates--Wiles~\cite{CoatesWiles:1977} consider a tower of extensions
consisting of the fields obtained by adjoining $\mathfrak {fl}^n$-torsion points 
of a CM elliptic curve whose CM Grössencharacter has conductor $\mathfrak f$. 
These fields form the ray class field tower of conductors $\mathfrak f\mathfrak l^n$.
In this paper, we consider characters of $\mathrm{Cl}_F(\mathfrak l^\infty)$,
which is the inverse limit of the ray class groups of $\mathfrak l$-power conductors. 

Let $f$ be a normalized Hecke eigenform of parallel weight $(2,2)$ and level $\mathfrak N$.  
For a finite-order Hecke character $\xi$, write
\[
 L(s,f\otimes\xi)
 =\sum_{\mathfrak a\subset\mathcal O_F}
   \frac{a_f(\mathfrak a)\xi(\mathfrak a)}{N(\mathfrak a)^s},
\]
where the sum is over nonzero integral ideals.
Here, we use the convention that $\xi(\mathfrak a)=0$ when $\mathfrak a$ is not coprime to the conductor. 
We study its critical value $L(1,f\otimes\xi)$.
Let $\Omega_f$ be Hida's period, $G(\xi)$ the Gauss sum, and
$\Gamma_F(1)$ the archimedean factor defined in \S\ref{int:rep:sec}.
The algebraic $L$-value considered here is
\[
 L_f(\xi)
 :=\frac{\Gamma_F(1)G(\xi)L(1,f\otimes\xi)}{\Omega_f}\in\overline{\mathbb Q}.
\]
To construct an integral $L$-value, we introduce compactly supported lifts of parabolic cohomology classes.
Such lifts are not unique, so we introduce the boundary killing operator:
Choose an element $q\in 1+\mathfrak N$ such that $q\mathcal O_F$ is a prime ideal. 
Put $\mathcal E_q=T_{q\mathcal O_F}-N(q\mathcal O_F)-1$, which is an operator
appearing in the Eisenstein-ideal of Mazur~\cite{Mazur:1977} for the classical modular curve case.
This operator kills the boundary cohomology group (Lemma~\ref{hecke:action:boundary:cohom}).
Let $\varepsilon_q$ be its eigenvalue on $f$.
The construction in \S\ref{subsec:parabolic:eigenspace} and \S\ref{subsec:integral:Lvalue} 
shows that $\varepsilon_qL_f(\xi)$ is integral.

To state the main result, let us give some notation.
Denote by $X_n$ the set of primitive $\ell$-power order characters of $\mathrm{Cl}_F(\mathfrak l^\infty)$
of conductor $\mathfrak l^n$. Assume that $\ell>3$ and it does not divide the class number of $F$.
Let $\Delta$ be the prime-to-$\ell$-part of $\mathrm{Cl}_F(\mathfrak l^\infty)$, which is finite.
Let $E_F:=\mathcal O_F^\times$ be the unit group.
For $\varepsilon$ in $E_F$ and a character $\phi\in X_n$, 
let $\phi^\varepsilon\in X_n$ be the Galois conjugate of $\phi$
obtained by letting $\varepsilon$ act on the values of $\phi$
(see \S\ref{subsec:main:result}).

We consider residual primes $\mathfrak P$ away from $\varepsilon_q\#E_F$,
so that the mod $\mathfrak P$ reduction of $L_f(\xi)$ makes sense since $\varepsilon_qL_f(\xi)$ is integral.

\begin{letterthm}[Theorem~\ref{thm:mod-p-nonvanishing}]\label{thm:main:introduction}
Assume that \(\mathfrak N=\mathfrak l^e\) with $e>0$. 
Then there exist a number field $L$, a positive-density set $\mathcal P$ of finite
primes of $L$, and an arithmetic progression $\mathfrak V\subset\mathbb Z_{>0}$
such that the following statement is true:
Let $\mathfrak p\in\mathcal P$ and $\mathfrak P$ a prime ideal of $\overline{\mathbb Z}$ lying above $\mathfrak p$. 
Then, for any $\phi\in\bigcup_{n\in\mathfrak V}X_n$, 
there exist $\varepsilon\in\mathcal O_F^\times$ and $\psi\in\widehat\Delta$ such that
\[
 L_f(\phi^\varepsilon\psi)\not\equiv0\pmod{\mathfrak P}.
\] 
\end{letterthm}
This theorem says that for every conductor exponent $n\in\mathfrak V$ and every $\phi\in X_n$,
at least one of the finitely many values
\[
\{L_f(\phi^\varepsilon\psi)\,:\,\varepsilon\in\mathcal O_F^\times,\,\psi\in\widehat\Delta\}
\]
is non-zero modulo $\mathfrak P$.
Since the number of unit conjugates and tame twists is bounded
independently of \(n\), this gives a positive proportion of
nonzero values in the family
\(\{\phi\psi:\phi\in X_n,\ \psi\in\widehat\Delta\}\),
uniformly for \(n\in\mathfrak V\).

\subsection{Relation to earlier work}

Namikawa~\cite{Namikawa:2012} proved an imaginary-quadratic analogue of the
non-vanishing theorem of Ash--Stevens~\cite{AshStevens:1986}.  His main theorem
treats Hecke eigenforms on $\mathrm{GL}(2)$ over an imaginary quadratic field of
class number one and of parallel weight $(2,2)$.  Under the Non-Eisenstein
hypothesis, he obtains infinitely many finite-order Hecke twists whose integral
$L$-values are nonzero mod $p$. His expression of these values uses Urban's cap-product~\cite{Urban:1995} 
and the period and the cohomology class attached to the given eigenform constructed by Hida~\cite{Hida:1994}.

The present work likewise considers Hecke eigenforms on $\mathrm{GL}(2)$ over
imaginary quadratic fields.  We work in parallel weight $(2,2)$ and $\mathfrak l$-power levels.
We allow arbitrary class number prime to $\ell$ and obtain a positive proportion 
in a specified vertical family of Coates--Wiles characters.  
We also use Hida's period and class, but construct directly the pairing 
between integral parabolic homology and cohomology of arithmetic orbifolds
and prove the duality needed for the main theorem.  
In particular, we use the boundary-killing operator $\mathcal E_q$ 
to define the automorphic symbols. This makes the symbols integral parabolic homology classes.

Kim and Sun~\cite{KimSun:2017} studied cyclotomic twists of classical modular
forms using ordinary homology and cohomology of the Baily--Borel--Satake compactification of modular
curves.  Vatsal's canonical periods and congruence formulae~\cite{Vatsal:1999} allow them 
to prove generation of the residual Hecke field, which is stronger than non-vanishing mod $p$.  
In the automorphic setting used here, Hida's period and rationality
theorem~\cite{Hida:1994} do not provide integral structures compatible with congruences.
Our argument therefore proves non-vanishing mod $p$ but does not establish residual Hecke-field generation.  
It is an automorphic analogue of the non-vanishing part of Kim--Sun's theorem and a $p\ne\ell$ counterpart of
the results of~\cite{Kwon:2023} on the $p$-adic $L$-functions.

Symbols also appear in the construction of $p$-adic $L$-functions for
cohomological automorphic forms on $\mathrm{GL}_2$ over general number fields
by Barrera Salazar and Williams~\cite{BarreraWilliams:2019}.  Their symbols are
compactly supported cohomology classes attached to eigenforms, 
and their evaluation on automorphic cycles then produces critical $L$-values. 
Their symbols are defined for modules over fields.
Our automorphic symbols, by contrast, lie in parabolic homology with coefficients in an integral ring.

On the Hecke $L$-function side, Gillard~\cite{Gillard:1985} and Schneps~\cite{Schneps:1987} 
proved the vanishing of the
$\mu$-invariant of the $p$-adic $L$-functions interpolating $L$-values
along the Coates--Wiles $\mathbb Z_p$-extension in the CM setting.

\subsection{Automorphic symbols and integral \emph{L}-values}
\label{subsec:intro:automorphic-symbols}
A central object of our paper is the \emph{automorphic symbol}.

The underlying arithmetic space is an orbifold
\[
Y_1^i(\mathfrak N)=\Gamma_1^i(\mathfrak N)\backslash\mathscr H_3.
\]
Here $\mathscr H_3$ is hyperbolic $3$-space, and
$\Gamma_1^i(\mathfrak N)$ is the arithmetic subgroup associated with
the $i$-th ideal class of $F$ at $U_1(\mathfrak N)$-level
(\S\ref{subsec:u0:u1}).

Let $A$ be a Dedekind domain and put
\[
H=H_!^1(Y_1^i(\mathfrak N),A),\qquad
M=H_1^!(Y_1^i(\mathfrak N),A).
\]
Write $(-)^\vee=\operatorname{Hom}_A(-,A)$. The parabolic pairing $\cap:M\times H\to A$
induces an isomorphism
\[
\mathcal D_\cap:H\xrightarrow{\ \sim\ }M^\vee,
\qquad
\eta\longmapsto\left(m\mapsto m\cap\eta\right)
\]
(Lemma~\ref{dualmap:isom}). Since $M$ is finitely generated over $A$,
the quotient $M/M_{\mathrm{tors}}$ is projective. We may therefore fix
an $A$-linear section
\[
s_M:M^{\vee\vee}\longrightarrow M
\]
of the natural map $M\to M^{\vee\vee}$.
For an additive twist parameter $u$, let
\[
\lambda_{u,i}:H\longrightarrow A,\qquad
\eta\longmapsto
\operatorname{Tr}_E\Delta_i^*\alpha_{u,i}^*\eta_{\mathrm c}.
\]
Here $\eta_{\mathrm c}=\mathcal E_q\eta'$ for any compactly supported
lift $\eta'$ of $\eta$. The map $\alpha_{u,i}^*$ realizes the additive
translation by $u$ at a suitable level, and $\Delta_i$ parametrizes
the geodesic from $0$ to $\infty$ in $\mathscr H_3$. The map $\operatorname{Tr}_E$
evaluates the resulting compactly supported class on the oriented
geodesic. These maps are defined in
\S\ref{subsection:integral-automorphic-symbols}.
The boundary-killing property of $\mathcal E_q$ makes $\lambda_{u,i}$
independent of the choice of $\eta'$.
We define
\[
\Upsilon_i(u)
:=s_M\left(\left(\mathcal D_\cap^\vee\right)^{-1}
\left(\lambda_{u,i}\right)\right)\in M=H_1^!(Y_1^i(\mathfrak N),A).
\]
Then we have
\[
\Upsilon_i(u)\cap\eta
=\operatorname{Tr}_E\Delta_i^*\alpha_{u,i}^*\eta_{\mathrm c}.
\]
Although the homology class can depend on $s_M$, its pairing with
every parabolic cohomology class is independent of this choice.

For a Hecke eigenform $f$,
our argument expresses the integral $L$-values $\varepsilon_q L_f(\phi)$ in terms of automorphic symbols.

\subsection{Idea of the proof of the main theorem}

We follow the strategy introduced by Kim and Sun~\cite{KimSun:2017}.
First, we study the analytic behavior of averages of additively twisted
$L$-values.  The resulting vertical limit formula identifies the limit of such
an average with a Fourier--Whittaker coefficient of the cusp form
(Proposition~\ref{prop:additive-average-vertical-limit}). 

We then translate this formula into a limit formula for pairings with
automorphic symbols (Proposition~\ref{prop:automorphic-symbol-pairing-limit}).
For $c\in O_{\mathfrak l}^\times$, $n_0>0$ and $n>2n_0$, let us set
\[
M_n^1(c)
:=\operatorname{Span}_A
\left\{
\Upsilon_1\left(\frac{a}{\varpi_{\mathfrak l}^n}\right):
a\in c+\mathfrak l^{n_0} O_{\mathfrak l}
\right\}\subset H_1^!(Y_1^i(\mathfrak{N}),A),
\]
where $\varpi_{\mathfrak l}$ is a uniformizer of $O_{\mathfrak l}$.
If $M_n^1(c)\otimes_A\mathbb C$ is a proper subspace of the complex parabolic homology for infinitely many $n$, 
then there is a non-zero parabolic cohomology class which annihilates when it pairs with the symbols
of form $\Upsilon_1(a/\varpi_{\mathfrak l}^n)$.
Then the limit formula, together with the invertibility of $\mathcal E_q$,
forces the non-zero cohomology class to vanish. This is a contradiction.
It implies that, for some $n^\circ$, the module $M_{n^\circ}^1(c)$
is a full-rank submodule of integral parabolic homology (Theorem~\ref{vertical:direction}). 
By using properties of automorphic symbols, there is an arithmetic progression $\mathfrak V$ 
such that 
\[
M_{n^\circ}^1(c)\subset M_n^1(c)
\]
for any $n\in\mathfrak V$
(Theorem~\ref{full:rank:vertical}).  Away from a finite set of primes, 
coefficient change therefore identifies the reduction of $M_n^1(c)$
with the full parabolic homology over the residual field
(Theorem~\ref{prop:vertical-full-rank-finite-field}).

Suppose now that all residual $L$-values in the relevant decomposition-group
and tame-character averages vanish.  The integral $L$-value formula and these
averages imply that a Fourier transform of the parabolic cohomology class pairs trivially
with every generator of $M_n^1(c)$
(Proposition~\ref{prop:first-averaged-congruence}). The preceding coefficient-change isomorphism 
and perfect parabolic duality then force this Fourier transform to vanish.

Finally, we show, away from
a fixed obstruction ideal, that at least one of these Fourier transforms is
nonzero (Proposition~\ref{prop:residual-nonvanishing-eta-f}).  This contradiction
proves the residual non-vanishing asserted in
Theorem~\ref{thm:mod-p-nonvanishing}.

\subsection{Obstructions over imaginary quadratic fields}
The obstruction comes from the unit group $E_F$.

The Atkin--Lehner operators $W_{\mathfrak d}^{\chi}$
(\S\ref{subsec:atkin-lehner-func-eq}) are used in the analytic argument.
The construction of $W_{\mathfrak d}^{\chi}$ requires extending a local factor of its central
character $\chi$ to a global finite-order Hecke character.  Such an extension
requires the local part of $\chi$ to be trivial on $E_F$, as
in~\eqref{eq:atkin-lehner-unit-condition}.  
On the other hand, our key argument (Theorem~\ref{vertical:direction}) pairs
symbols against all cusp forms contributing to
$H_!^1(Y_1^1(\mathfrak N),\mathbb C)$, so this condition must hold for the
central character of every such form.  
If the level has only one prime factor, then the unit restriction condition holds automatically
as the central character is a Hecke character.
We therefore impose $\mathfrak l$-power level.
Note that for classical modular curves the totally positive unit group is
trivial, and this obstruction does not occur in~\cite{KimSun:2017}.

To apply the full-rankness result, we take decomposition-group
averages of the algebraic $L$-values. We express the resulting
averages using two successive level-changing pullbacks through
an intermediate, lower prime-power level. This separates the
Fourier transform of the cohomology class at the lower level
from its pairings with automorphic symbols of higher level.

The two pullbacks naturally involve both $Y_1^i$ and its
$\mathrm{SL}_2$-level cover $Y^i$. The latter corresponds to
$U_1^1$-level and carries the action of $E_F$
(\S\ref{subsec:SL2:covering}). To compose the pullbacks, the
intermediate class on $Y^i$ must descend to $Y_1^i$.
We establish its $E_F$-invariance to justify this descent.
This compatibility is used in
Proposition~\ref{prop:first-averaged-congruence}.

\subsection{Outline of the paper}

In \S\ref{sec:automorphic-forms}, we introduce automorphic forms over imaginary quadratic fields.
In \S\ref{section:l-functions-of-cusp-forms}, we study their twisted $L$-functions, 
and in \S\ref{sec:vertical-limits}, we prove the vertical Fourier-coefficient limit.  
In \S\ref{section:conj:on:mod:symb}, we establish the duality between the parabolic homology and cohomology.  
In \S\ref{sec:arithmetic-orbifolds}, we describe the relevant arithmetic orbifolds, 
and in \S\ref{sec:integral-symbols}, we construct integral automorphic symbols representing integral $L$-values.
In \S\ref{sec:full-rank}, we prove complex and residual full-rankness of the symbol modules.  
Finally, in \S\ref{sec:mod-p-nonvanishing}, we apply these results 
and prove the main non-vanishing theorem for Coates--Wiles characters.

\subsection*{Acknowledgement}
This work was supported by Basic Science Research Program 
through the National Research Foundation of Korea (NRF) 
funded by the Ministry of Education (RS-2025-00563144).
The author would like to thank Dohyeong Kim and Hae-Sang Sun for helpful discussion and comments.

\section{Automorphic forms over imaginary quadratic fields}
\label{sec:automorphic-forms}

Recall that $F$ is an imaginary quadratic field.
We use the standard embeddings of $F_v^\times$ into
$\mathbb A_F^\times$ and of $\mathrm{GL}_2(F_v)$ into
$\mathrm{GL}_2(\mathbb A_F)$, with identity components away from $v$.
We also use the diagonal embeddings of $F^\times$ and
$\mathrm{GL}_2(F)$.

\subsection{Hecke characters}

\begin{definition}
A Hecke character $\xi$ of $F$ is a continuous homomorphism
\[
  \xi: F^\times \backslash \mathbb{A}_F^\times 
  \longrightarrow \mathbb{C}^\times.
\]
\end{definition}

For an idele $a=(a_v)_v \in \mathbb{A}_F^\times$ and a place $v$ of $F$, 
write $a_v$ for its $v$-component.
At each place $v$, fix an absolute value 
$|\cdot|_v:F_v^\times \to \mathbb{R}_{>0}^\times$, 
normalized so that the product formula holds. 
For example, $|z|_\infty=|z|^2$ for the complex place $\infty$ of $F$,
where $|\cdot|$ is the absolute value on the complex numbers.
This induces the global idelic norm
\[
  |\cdot|_{\mathbb A_F}: F^\times\backslash \mathbb{A}_F^\times 
  \longrightarrow \mathbb{R}_{>0}^\times,
\]
which is well-defined and surjective.  
Thus, noncanonically, one has a split exact sequence
\[
\begin{tikzcd}
1 \arrow{r} & F^\times\backslash \mathbb{A}_F^{\times,1} \arrow{r}
& F^\times\backslash \mathbb{A}_F^\times \arrow[r,"|\cdot|_{\mathbb A_F}"]
& \mathbb{R}_{>0}^\times \arrow{r} & 1,
\end{tikzcd}
\]
where $\mathbb{A}_F^{\times,1} := 
\{a\in \mathbb{A}_F^\times : |a|_{\mathbb A_F}=1\}$.
As $F^\times\backslash \mathbb{A}_F^{\times,1}$ is compact, 
a Hecke character $\xi$ factors uniquely as
\[
  \xi = \xi^{\mathrm{u}} \times |\cdot|_{\mathbb A_F}^{\,r},
\]
where $\xi^{\mathrm{u}}$ is unitary and $r=r(\xi)\in \mathbb{R}$.

The infinite part of $\mathbb{A}_F$ is given by 
$F\otimes_\mathbb{Q}\mathbb{R}=\mathbb{C}$.
The \emph{archimedean component} of a Hecke character $\xi$ is its restriction to 
$\mathbb{C}^\times \subset \mathbb{A}_F^\times$, 
denoted by $\xi_\infty:\mathbb{C}^\times \to \mathbb{C}^\times$.

\subsection{Ray class groups}
The kernel of $\xi$ contains $1+\widehat{\mathfrak{c}}$
for some integral ideal $\mathfrak{c}$ of $F$, 
where
\[
  1+\widehat{\mathfrak{c}} :=
    \prod_{v\mid \mathfrak{c}} (1+\mathfrak{c}O_v)
    \times \prod_{v\nmid \mathfrak{c}} O_v^\times
\]
and $O_v$ denotes the valuation ring of $F_v$.
Such $\mathfrak{c}$ is called a \emph{modulus} of $\xi$.
The largest such ideal with respect to inclusion is called the
\emph{conductor} of $\xi$.
A Hecke character of modulus $\mathfrak c$ is \emph{primitive} 
if its conductor is $\mathfrak c$.
\begin{definition}
The ray class group of modulus $\mathfrak c$ is
\begin{equation*}
  \mathrm{Cl}_F(\mathfrak c)\,:=\,F^\times\backslash \mathbb{A}_F^\times \;\Big/\;
  \left(\mathbb{C}^\times \times 
  \left(1+\widehat{\mathfrak{c}}\right)\right).
\end{equation*}
\end{definition}
If the archimedean component $\xi_\infty$ of a character $\xi$ of modulus $\mathfrak c$ is trivial, then
$\xi$ factors through $\mathrm{Cl}_F(\mathfrak c)$, which is a finite group.
Consequently, $\xi$ has finite image.

A Hecke character $\xi$ is of \emph{finite order}
if its image is a finite subgroup of $\mathbb{C}^\times$.
As we observed, if $\xi_\infty=1$ then $\xi$ is of finite order.
Conversely, if $\xi$ has finite order, then $\xi_\infty=1$,
since $\mathbb{C}^\times$ is connected whereas the image of
$\xi_\infty$ is finite.
Consequently, $\xi$ has finite order if and only if $\xi_\infty=1$.

For simplicity, set $O_{\mathfrak c}:=\prod_{v\mid\mathfrak c}O_v$.
Let $E_F$ be the unit group $\mathcal O_F^\times$. 
The diagonal embedding maps each $\varepsilon\in E_F$ to an element
of $O_{\mathfrak c}^\times$.
There is a fundamental exact sequence of ray class groups
\begin{equation}\label{eq:ray-class-local-global-exact-sequence}
\begin{tikzcd}
1 \arrow{r} &
E_F/E_F(\mathfrak c) \arrow{r} &
(O_\mathfrak{c}/\mathfrak{c}O_\mathfrak{c})^\times
\arrow{r} &
\mathrm{Cl}_F(\mathfrak c) \arrow{r} &
\mathrm{Cl}_F \arrow{r} & 1,
\end{tikzcd}
\end{equation}
where $E_F(\mathfrak c):=E_F\cap (1+\mathfrak c)$ and $\mathrm{Cl}_F$ is the class group of $F$.
Here, the map $(O_{\mathfrak c}/\mathfrak cO_{\mathfrak c})^\times\to\mathrm{Cl}_F(\mathfrak c)$
is induced by the embeddings 
$F_v^\times\to\mathbb A_F^\times$, $v\mid\mathfrak c$.
There is a canonical identification
\[
(O_\mathfrak{c}/\mathfrak{c}O_\mathfrak{c})^\times
\simeq
\prod_{v\mid\mathfrak{c}}(O_v/\mathfrak{c}O_v)^\times.
\]

\subsection{Extension of local characters to Hecke characters}
Let us choose a character
$
\rho:
O_{\mathfrak c}^\times\longrightarrow\mathbb C^\times
$
whose kernel contains $1+\mathfrak cO_{\mathfrak c}$.

\begin{proposition}\label{prop:hecke-character-local-extension}
Assume that $\ker(\rho)$ contains the image of $E_F$.
Then, there exists a finite-order Hecke character $\chi$ 
of modulus $\mathfrak c$ such that
\[
 \chi\vert_{O_{\mathfrak c}^\times}=\rho.
\]
\end{proposition}

\begin{proof}
Applying $\mathrm{Hom}_{\mathbb Z}(-,\mathbb{C}^\times)$
to \eqref{eq:ray-class-local-global-exact-sequence},
and using the injectivity of the divisible $\mathbb Z$-module $\mathbb C^\times$, 
we obtain a character $\chi$ satisfying the assertion.
\end{proof}

\subsection{Cuspidal automorphic forms of GL(2)}\label{cuspform}
Let us introduce some notation.
Let us denote by 
\[
\mathbf{L}_{2}
:=
\mathrm{Sym}^2(\mathbb{C}^2)
\]
the space of homogeneous polynomials of degree $2$ in the variables
\[
\mathbf{x}
=
\left(\begin{matrix} X \\ Y \end{matrix}\right).
\]
For a $\mathbb Z$-algebra $A$, let us set 
$\widehat{A}:=A\otimes_\mathbb{Z}\widehat{\mathbb{Z}}$.
Let $\mathfrak{N}$ be an integral ideal of $F$. 
Let us define an open compact subgroup of the non-archimedean part of
$\mathrm{GL}_2(\mathbb{A}_F)$:
\[
U_0(\mathfrak{N})
:=
\left\{\left(\begin{matrix} a & b \\ c & d \end{matrix}\right)
\in 
\mathrm{GL}_2(\widehat{\mathcal{O}}_F)
\;:\;
c\in\widehat{\mathfrak{N}}\right\}.
\]
Let $\chi$ be a finite order Hecke character of modulus $\mathfrak{N}$.
For each finite place $v$, let $\chi_v:=\chi|_{F_v^\times}$.
For a fractional ideal $\mathfrak b$ of $F$, put
\[
\chi_{\mathfrak b}:=\prod_{v\mid\mathfrak b}\chi_v,
\]
viewed as a character of $\prod_{v\mid\mathfrak b} F_v^\times$.

\begin{definition}\label{adel:cuspform:defn} 
A cuspidal automorphic form on $\operatorname{GL}(2)$ 
of parallel weight $(2,2)$, level $\mathfrak N$, and central character $\chi$
is a smooth function 
\[
f:\operatorname{GL}_2(\mathbb{A}_F)\rightarrow \mathbf{L}_2
\]
satisfying the following conditions for any $g\in\operatorname{GL}_2(\mathbb{A}_F)$:
\begin{itemize}
    \item[(1)] 
    For the Casimir operators $\{D_\tau\mid\tau\in I_F\}$ of 
    $\mathrm{SL}_2$,
    \[
    D_\tau f=0.
    \]
    \item[(2)]
    For all $\gamma\in\mathrm{GL}_2(F)$, 
    $z\in\mathbb A_F^\times$, $u\in U_0(\mathfrak{N})$,
    and $c\in \operatorname{SU}_2(\mathbb{C})$,
    \[
    f(\gamma z g u c,\mathbf{x})=\chi(z)\chi_{\mathfrak N}(d)f(g,c\mathbf{x}),
    \]
    where $d$ is the lower-right entry of the matrix $u$.
    \item[(3)] 
    Let $dx$ be a Haar measure on $F\backslash\mathbb{A}_F$. Then,
    \[
    \int_{F\backslash \mathbb{A}_F}f\left(\left(\begin{matrix} 1 & x \\ 0 & 1 \end{matrix}\right) g\right)dx=0.
    \]
\end{itemize}
\end{definition}
We refer to a cuspidal automorphic form simply as a cusp form.
Let us denote by $S_2(\mathfrak{N},\chi)$ 
the space of cusp forms in Definition~\ref{adel:cuspform:defn}.
\begin{remark}
The assumption that the weight is parallel and integral is necessary
for the space of cusp forms to be nonzero~\cite[p.~461]{Hida:1994}.
Under this condition, Hida's rationality results apply to the critical values considered later. 
After normalization by Hida's period and the relevant Gauss and archimedean factors, 
these values are algebraic~\cite[Theorem 8.1]{Hida:1994}.
\end{remark}

\subsection{Fourier-Whittaker expansion}\label{subsec:whittaker}
Write \(y=(y^{(\infty)},y_\infty)\) 
according to the finite and archimedean components of \(\mathbb A_F\).
Let $\mathbf{e}_F$ be the standard additive character of $F\backslash\mathbb{A}_F$ 
whose archimedean component is given by
$\mathbf{e}\left(\mathrm{Tr}_{\mathbb C/\mathbb R}(z)\right)$, where $z\in \mathbb{C}$ and $\mathbf{e}(z):=e^{2\pi iz}$.
Define the function $\mathrm{W}$ on $\mathbb C^\times$ as
\[
\mathrm{W}(y):=K_0(4\pi|y|),
\]
where $K_0$ is the modified Bessel function of the second kind of order zero.
Note that this is the archimedean component of the weight-\(2\) Whittaker function.

Let $\mathfrak d_F$ be the different ideal of $F/\mathbb Q$.
For a finite idele $a=(a_v)_{v\nmid\infty}$ and an integral ideal $\mathfrak b$ of $F$, set
\[
 a\mathfrak b
 :=\mathfrak b\prod_{v\nmid\infty}
 \mathfrak p_v^{\operatorname{ord}_v(a_v)}.
\]
For a function $h:\mathrm{GL}_2(\mathbb A_F)\longrightarrow\mathbf L_2$, 
write the $XY$-component of $h$ as $\widetilde{h}$.
Let $f\in S_2(\mathfrak{N},\chi)$ be a cusp form of weight $2$. 
By Hida \cite[Theorem 6.1]{Hida:1994}, there is a complex-valued function 
$a_f$ on the set of fractional ideals of $F$ such that 
$a_f$ vanishes outside the set of integral ideals and
\begin{equation}\label{Fourier:Whittaker:expansion}
\widetilde{f}(g)=|y|_{\mathbb A_F}\sum_{\alpha\in F^\times}a_f\left(\alpha y^{(\infty)}\mathfrak{d}_F\right)
\mathrm{W}\left(\alpha y_\infty\right)\mathbf{e}_F(\alpha x),
\end{equation}
for every
$g=\left(\begin{smallmatrix}y&x\\0&1\end{smallmatrix}\right)
\in\mathrm{GL}_2(\mathbb A_F)$.

\subsection{Gauss sums and twisted cusp forms}\label{twisted:cuspforms}
For each finite place $v$ of $F$, choose a uniformizer $\varpi_v$ of $F_v$.
We regard \(\varpi_v\) as a finite idele 
by setting all its other components equal to \(1\). 
For a fractional ideal $\mathfrak b$ of $F$, put
$\varpi_\mathfrak{b}
:=
\prod_{v\mid\mathfrak{b}}\varpi_v^{\mathrm{ord}_v(\mathfrak{b})}$.

Let $\mathfrak c$ be an integral ideal of $F$. 
Set
\[
(\mathfrak{c}^{-1}/\mathcal O_F)^\times
:=
\{[\alpha]\in\mathfrak{c}^{-1}/\mathcal O_F\;:\;
\mathrm{Ann}_{\mathcal O_F}([\alpha])=\mathfrak{c}\}.
\] 
Choose a complete set of representatives of
$(\mathfrak{c}^{-1}/\mathcal O_F)^\times$.
Let $\mathcal{R}_\mathfrak{c}$ be the image of this set 
under the diagonal map $F\to\prod_{v\mid\mathfrak c}F_v$. 
For $u\in \mathcal{R}_\mathfrak{c}$, the condition
$\mathrm{Ann}_{\mathcal O_F}([r_u])=\mathfrak{c}$ is equivalent to
$\mathrm{ord}_v(u_v)=-\mathrm{ord}_v(\mathfrak{c})$ 
for every $v\mid\mathfrak{c}$, where $r_u\in\mathfrak c^{-1}$ maps to $u$.
Consequently, multiplication by $\varpi_\mathfrak{c}$ 
induces a bijection of finite sets
\[
\mathcal{R}_\mathfrak{c}
\longrightarrow
(O_\mathfrak{c}/\mathfrak{c}O_\mathfrak{c})^\times,
\qquad
u\longmapsto \varpi_\mathfrak{c}u\pmod{\mathfrak{c}O_\mathfrak{c}}.
\]
Thus, after choosing the uniformizers and the representatives, 
we may identify $\mathcal{R}_\mathfrak{c}$ with 
$(O_\mathfrak{c}/\mathfrak{c}O_\mathfrak{c})^\times$.
This is not canonical as it depends on the choice of uniformizers.
We regard \(u\in\mathcal R_{\mathfrak c}\) as an additive adele by setting \(u_v=0\) for \(v\nmid\mathfrak c\).

Let $\xi$ be a finite-order Hecke character of modulus $\mathfrak c$. Define the \emph{Gauss sum}
\[
G(\xi):=\frac{1}{\xi(d_F)}\sum_{u\in\mathcal{R}_\mathfrak{c}}
\xi(\varpi_\mathfrak{c} u)\mathbf{e}_F\left(\frac{u}{d_F}\right),
\]
where $d_F$ is a finite idele of $F$ corresponding to $\mathfrak{d}_F$
and $\xi(\varpi_\mathfrak{c} u)$ is defined via the map
$(O_\mathfrak{c}/\mathfrak{c}O_\mathfrak{c})^\times
\to\mathrm{Cl}_F(\mathfrak{c})$.
Note that this definition does not depend on the choices of 
$d_F$ and $\mathcal{R}_\mathfrak{c}$.
For $a\in O_{\mathfrak c}^\times$, we have the following inversion formula:
\begin{equation}\label{eq:gauss:inversion}
\frac{1}{\xi(d_F)}
\sum_{u\in\mathcal R_{\mathfrak c}}
\xi(\varpi_\mathfrak{c} u)\mathbf{e}_F\left(\frac{au}{d_F}\right)
=G(\xi)\overline{\xi}(a).
\end{equation}
If \(\xi\) is primitive, then \eqref{eq:gauss:inversion} gives
\begin{equation}\label{eq:gauss:product}
G(\xi)G(\overline\xi)=\xi(-1)N(\mathfrak c).
\end{equation}
Put $t(u):=\left(\begin{smallmatrix} 1 & u \\ 0 & 1 \end{smallmatrix}\right)
\in\mathrm{GL}_2(\mathbb{A}_F)$ for $u\in\mathcal{R}_\mathfrak{c}$.
Assume that $\xi$ is primitive, and define the \emph{twisted cusp form} by
\[
f\otimes\xi(g):=\frac{\xi(\det g)}{G(\xi)}
\sum_{u\in\mathcal{R}_\mathfrak{c}}\xi(\varpi_\mathfrak{c} u)f(gt(u)),
\quad g\in\mathrm{GL}_2(\mathbb{A}_F).
\]
This definition does not depend on the choices of 
$\mathcal{R}_\mathfrak{c}$.
Note that our normalization for $f\otimes\xi$ differs from that of
Hida~\cite[Section 6]{Hida:1994}.
Then,
\[
f\otimes\xi\in S_{2}(\mathfrak{N}\cap\mathfrak{c}^2,\chi\xi^2),
\]
and $a_{f\otimes\xi}(\mathfrak{a})=a_f(\mathfrak{a})\xi(\mathfrak{a})$
for any integral ideal $\mathfrak{a}$ of $F$
by Hida \cite[Section 6]{Hida:1994}.

\section{{\it L}-functions of cusp forms}\label{section:l-functions-of-cusp-forms}
Let us set $\mathbb{C}^\times_1:=\{y\in \mathbb{C}^\times:|y|=1\}$ and
\[
E:=\mathbb{C}^\times_1 \backslash \mathbb{C}^\times.
\]
Equip $E$ with the quotient topology. 
The Whittaker function $\mathrm W$ introduced in \S\ref{subsec:whittaker},
as well as the function $y\mapsto |y|$, descends to $E$. 
Via the identification $E\xrightarrow{\;\simeq\;}\mathbb R_{>0}^\times$,
we normalize the Haar measure on $E$ by
\[
d^\times y=\frac{dr}{r},\quad r=|y|.
\]
For $\alpha\in F^\times$ and $s\in\mathbb{C}$ with $\mathrm{Re}(s)>0$, 
the Whittaker function satisfies
\begin{equation}\label{Whittaker:function:integral}
\int_E
\mathrm{W}(\alpha y)|y|^{2s}d^\times y 
=
\int_{0}^\infty K_0(4\pi|\alpha|r)r^{2s} \frac{dr}{r}
=
\frac{1}{N(\alpha \mathcal O_F)^s}\frac{\Gamma(s)^2}{4(2\pi)^{2s}}.
\end{equation}

\subsection{{\it L}-functions of cusp forms and their integral representations}\label{int:rep:sec} 
Let $f\in S_2(\mathfrak N,\chi)$, and let $a_f$ be its
Fourier--Whittaker coefficient function defined by
\eqref{Fourier:Whittaker:expansion}. 
Let $a$ be a finite idele of $F$ and $u$ a finite adele of $F$.

\begin{definition}
The additively twisted partial $L$-function is
\[
L_a(s,f;u):=\frac{1}{\# E_F}\sum_{\alpha\in F^\times}
\frac{a_f(\alpha a\mathfrak{d}_F)\mathbf e_F(\alpha au)}{N( \alpha a\mathfrak{d}_F)^s}.
\]
\end{definition}
By the bound of Blomer--Brumley \cite{BlomerBrumley:2011}, 
this series converges absolutely for $\mathrm{Re}(s)>\frac{3}{2}+\theta$,
where $\theta=\frac{7}{64}$.
If $u$ is integral, then $\mathbf e_F(\alpha au)=1$ whenever $a_f(\alpha a\mathfrak{d}_F)\neq 0$,
hence $L_a(s,f;u)$ does not depend on $u$. 
In this case, we write \(L_a(s,f)\) for \(L_a(s,f;u)\).

For a function $h:\mathrm{GL}_2(\mathbb A_F)\to\mathbf{L}_2$,
let \(h\vert_u\) denote the right translate of \(h\) by \(t(u)\), defined by
\[
h\vert_u(g,\mathbf x):=h(gt(u),\mathbf x).
\]
$h\vert_u$ is also a $\mathbf{L}_2$-valued function defined on $\mathrm{GL}_2(\mathbb A_F)$.
Set
\[
\widetilde{h_a}(y):=\widetilde h
\left(
\left(\begin{matrix} a & 0 \\ 0 & 1 \end{matrix}\right)
\left(\begin{matrix} y & 0 \\ 0 & 1 \end{matrix}\right)
\right),
\]
where $\widetilde h$ is defined in \S\ref{subsec:whittaker}.
By~\eqref{Fourier:Whittaker:expansion}, we have
\begin{equation}\label{modularform:part:expansion}
\widetilde{(f\vert_u)_a}(y)
=
|ay|_{\mathbb A_F}\sum_{\alpha\in F^\times}a_f(\alpha a\mathfrak{d}_F)
\mathbf e_F(\alpha au)\mathrm{W}(\alpha y).
\end{equation} 
This function descends to a function on $E$.
Define the gamma factor $\Gamma_F(s)$ by
\[
\Gamma_F(s):=
\frac{D_F^s\# E_F}{4(2\pi)^{2s}}\Gamma(s)^2,
\]
where $D_F:=N(\mathfrak d_F)$.
Then $L_a(s,f;u)$ has the following integral representation:

\begin{proposition}\label{integral:representation:lfunction} 
For $s\in\mathbb{C}$ with $\mathrm{Re}(s)>\frac{3}{2}+\theta$,
\[
|a|_{\mathbb A_F}^{s-1}\int_{E}\widetilde{(f\vert_u)_a}(y)|y|^{2s-2}d^\times y
=
\Gamma_F(s)L_a(s,f;u).
\]
\end{proposition}

\begin{proof}
Substituting \eqref{modularform:part:expansion} and integrating term by term, we obtain
\begin{align*}
|a|_{\mathbb A_F}^{s-1}\int_{E} \widetilde{(f\vert_u)_a}(y) |y|^{2s-2}d^\times y
&=
|a|_{\mathbb A_F}^s\int_{E}\sum_{\alpha\in F^\times}
a_f(\alpha a\mathfrak{d}_F)\mathbf e_F(\alpha au)| y|^{2s} 
\mathrm{W}(\alpha y)d^\times y 
\\
&=
\frac{\Gamma_F(s)}{\#E_F}\sum_{\alpha\in F^\times}
\frac{a_f(\alpha a\mathfrak{d}_F)\mathbf e_F(\alpha au)}{N(\alpha a\mathfrak{d}_F)^s}
\end{align*}
from \eqref{Whittaker:function:integral}.
Here we have used
\[
N(\alpha a\mathfrak d_F)
=
N(\alpha\mathcal O_F)|a|_{\mathbb A_F}^{-1}D_F.
\]
This proves the proposition.
\end{proof}

\subsection{Atkin--Lehner operators}
\label{subsec:atkin-lehner-func-eq}
We first define the idelic Atkin--Lehner operators.
Let $\mathfrak N$ be an integral ideal and $\mathfrak d$ a Hall divisor of it, that is,
\[
 \mathfrak d\mid\mathfrak N,
 \qquad \mathfrak d+\mathfrak N\mathfrak d^{-1}=\mathcal O_F.
\]
Using the fixed uniformizers $\varpi_v$ and the notation $\varpi_{\mathfrak b}$ 
introduced in \S\ref{twisted:cuspforms},
define $B_{\mathfrak d}\in\mathrm{GL}_2(\mathbb A_F)$ by its local components:
\[
 (B_{\mathfrak d})_v=
 \begin{cases}
 \left(\begin{matrix}
  0&-1\\
  \varpi_{v}^{\operatorname{ord}_v(\mathfrak{d})}&0
 \end{matrix}\right), 
 & v\mid\mathfrak d, \\
 \left(\begin{smallmatrix}
  1&0\\
  0&1
 \end{smallmatrix}\right), 
 & \text{otherwise}.
 \end{cases}
\]
Thus
$\det(B_{\mathfrak d})=\varpi_{\mathfrak d}.$
\begin{lemma}\label{lem:normalizer:equation}
We have
\[
 B_{\mathfrak d}^{-1}U_0(\mathfrak N)B_{\mathfrak d}
 =U_0(\mathfrak N).
\]  
\end{lemma}

\begin{proof}
If $v\nmid\mathfrak d$, then
$(B_{\mathfrak d})_v$ is the identity matrix, so there is nothing to prove.
Suppose that $v\mid\mathfrak d$ and
put $e=\operatorname{ord}_v(\mathfrak{d})$.
Note that $\operatorname{ord}_v(\mathfrak{N})=e$.  
Write
\[
 B=(B_{\mathfrak d})_v
 =\begin{pmatrix}0&-1\\\varpi_v^e&0\end{pmatrix},
 \qquad
 k=\begin{pmatrix}a&b\\c&d\end{pmatrix}
 \in U_0(\mathfrak N)_v.
\]
Direct multiplication gives
\begin{equation}\label{eq:normalizer:calculation}
 B^{-1}kB
 =\begin{pmatrix}
 d&-\varpi_v^{-e}c\\
 -\varpi_v^e b&a
 \end{pmatrix}.
\end{equation}
This matrix belongs to $U_0(\mathfrak N)_v$, and hence
$B^{-1}U_0(\mathfrak N)_vB\subseteq U_0(\mathfrak N)_v$.
Since $B=-\varpi_v^e B^{-1}$,
the same calculation gives the reverse inclusion.
\end{proof}

For a Hecke character $\chi$, consider the following local restriction condition
\begin{equation}\label{eq:atkin-lehner-unit-condition}
 \chi_{\mathfrak d}(\varepsilon)=1
 \quad\text{for every}\quad\varepsilon\in E_F.
\end{equation}
Here, we use the notation $\chi_{\mathfrak b}$ introduced in \S\ref{cuspform}.
Let $\mathscr X_{\mathfrak d}$ be the set of finite-order Hecke characters $\chi$ of modulus $\mathfrak N$ 
that satisfy \eqref{eq:atkin-lehner-unit-condition}. 
Put $\mathfrak N_0=\mathfrak N\mathfrak d^{-1}$. For $\chi\in\mathscr X_{\mathfrak d}$, consider the character
\[
\rho_\chi:=\chi_{\mathfrak d}\vert_{O_{\mathfrak d}^{\times}}\times \mathbf{1}_{O_{\mathfrak N_0}^{\times}}
:O_{\mathfrak N}^{\times}\longrightarrow\mathbb C^\times.
\] 
By Proposition~\ref{prop:hecke-character-local-extension},
choose a finite-order Hecke character $\chi_1$ of modulus $\mathfrak N$
extending $\rho_\chi$.

\begin{definition}
For $\chi\in\mathscr X_{\mathfrak d}$, define the Atkin--Lehner operator
$W_{\mathfrak d}^\chi$
on functions $h:\mathrm{GL}_2(\mathbb A_F)\to\mathbf{L}_2$ by
\[
 (W_{\mathfrak d}^\chi h) (g)
 :=-\overline\chi_1(\det g)h(gB_{\mathfrak d}).
\]
\end{definition}

\begin{remark}\label{rem:atkin-lehner-unit-obstruction}
The construction requires the local restriction condition~\eqref{eq:atkin-lehner-unit-condition}.
Over a general number field, such an extension exists precisely 
when $\rho_\chi$ is trivial on the group of totally positive units $E_F^+$.
For $F=\mathbb Q$, $E_{\mathbb Q}^+=\{1\}$, so the condition is automatic.  
This explains why the classical Atkin--Lehner construction requires no analogous condition. 
\end{remark}

Put $\chi':=\overline{\chi_1}\chi_2$, where $\chi_2:=\chi\overline{\chi_1}$.
For each $\chi\in \mathscr X_{\mathfrak d}$, fix an auxiliary character $\chi_1$ as above, 
chosen to depend only on the restriction of $\chi$ to $O_{\mathfrak d}^{\times}$.
In particular, $(\overline{\chi'})_1=\chi_1$.

\begin{lemma}\label{lem:exact-divisor-atkin-lehner-properties}
The Atkin--Lehner operator
\[
 W_{\mathfrak d}^\chi:S_2(\mathfrak N,\chi)
 \longrightarrow S_2(\mathfrak N,\chi')
\]
is well-defined for $\chi\in\mathscr X_\mathfrak d$.
\end{lemma}

\begin{proof}
As $\chi_1$ is a Hecke character, the defining formula is left
$\mathrm{GL}_2(F)$-invariant. Right translation by $B_{\mathfrak d}$ 
and multiplication by $\overline{\chi_1}\circ\det$ preserve smoothness, 
the equivariance conditions of the archimedean place, and cuspidality.
Let $z\in\mathbb A_F^\times$,
$k=\left(\begin{smallmatrix}a&b\\c&d\end{smallmatrix}\right)\in U_0(\mathfrak N)$.
Write $B=B_{\mathfrak d}$ and $k'=B^{-1}kB$.  
Lemma~\ref{lem:normalizer:equation} and 
\eqref{eq:normalizer:calculation} give
\begin{align*}
 f(zgkB)=\chi(z)f(gBk')
 &=\chi(z)\chi_{\mathfrak d}(a)\chi_{\mathfrak N_0}(d)f(gB) \\
 &=\chi(z)\chi_1(\det k)\chi_{\mathfrak N}'(d)f(gB).
\end{align*}
Here, $a,d\in O_{\mathfrak N}^\times$.
Consequently,
\begin{align*}
 W_{\mathfrak d}^\chi f(zgk)
 &=-\overline{\chi_1}^2(z)
 \overline{\chi_1}(\det g)\overline{\chi_1}(\det k)f(zgkB) \\
 &=\chi'(z)\chi'_\mathfrak{N}(d)W_{\mathfrak d}f(g).
\end{align*}
This proves the lemma.
\end{proof}

Let us consider the finite sum of cusp form spaces:
\[
 S_{\mathfrak d}
 :=\sum_{\chi\in\mathscr X_{\mathfrak d}}
 S_2(\mathfrak N,\chi)
 \subset
 \left\{h:\mathrm{GL}_2(\mathbb A_F)\to\mathbf L_2\right\}.
\]
This sum is direct. Indeed, suppose that $\sum_\chi f_\chi=0$, $f_\chi\in S_2(\mathfrak N,\chi)$.
Then,
\[
 \sum_\chi \chi(z)f_\chi(g)=0
 \qquad
 (g\in\mathrm{GL}_2(\mathbb A_F),\ z\in\mathbb A_F^\times).
\]
For each fixed $g$, the linear independence of distinct characters 
as functions on the ray class group $\mathrm{Cl}_F(\mathfrak{N})$, 
applied to each coordinate in $\mathbf L_2$, implies that $f_\chi(g)=0$.
Thus every $f_\chi$ is zero.
Thus the operators $W_{\mathfrak d}^\chi$ assemble into an operator
\[
 W_{\mathfrak d}:S_{\mathfrak d}
 \longrightarrow S_{\mathfrak d},
 \qquad
 W_{\mathfrak d}\left(\sum_{\chi\in\mathscr X_{\mathfrak d}} f_\chi\right)
 :=\sum_\chi W_{\mathfrak d}^\chi f_\chi
\]
by Lemma~\ref{lem:exact-divisor-atkin-lehner-properties}
and the fact that $\chi'\in \mathscr X_{\mathfrak d}$ 
for $\chi\in\mathscr X_{\mathfrak d}$.

For a function $h:\mathrm{GL}_2(\mathbb A_F)\to\mathbf L_2$ 
and a finite order Hecke character $\chi$ of conductor $\mathfrak N$,
define the operator $W_{\mathfrak N}$ by
\[
 W_{\mathfrak N}h(g)
 =-\overline\chi(\det g)h(gB_{\mathfrak N}).
\]
A direct computation shows that
\[
 W_{\mathfrak N}:S_2(\mathfrak N,\chi)
 \longrightarrow S_2(\mathfrak N,\overline\chi)
\]
is a well-defined involution.  This is called the \emph{Fricke involution}; 
for further details, see \cite[Proposition~2.9]{Kwon:2025}.

\begin{lemma}\label{lem:composite-atkin-lehner-involution}
The operator $W_{\mathfrak N}W_{\mathfrak d}$ is an automorphism of the space $S_{\mathfrak d}$. 
More precisely, if $f\in S_2(\mathfrak N,\chi)$, then
\[
 W_{\mathfrak N}W_{\mathfrak d}f
 \in S_2(\mathfrak N,\overline{\chi'})
 \quad\text{and}\quad
 (W_{\mathfrak N}W_{\mathfrak d})^2f= \overline{\chi_1}(\varpi_{\mathfrak N_0}) f.
\]
\end{lemma}

\begin{proof}
Choose $f\in S_2(\mathfrak N,\chi)$ and put
\[
 C:=B_{\mathfrak N}B_{\mathfrak d}^{-1},
 \qquad
 \zeta_\chi:=\overline{\chi_1}(\varpi_{\mathfrak N})
 \chi(\varpi_{\mathfrak d}).
\]
The equality 
$(B_{\mathfrak d})_v=-(\varpi_{\mathfrak d})_v (B_{\mathfrak d}^{-1})_v$ for $v\mid\mathfrak d$ 
gives the alternative formula
\begin{equation}\label{eq:involution:rewrite}
 W_{\mathfrak d} f(g)
 =-\chi(\varpi_{\mathfrak d})
  \overline{\chi_1}(\det g)
  f(gB_{\mathfrak d}^{-1}).
\end{equation}
Then~\eqref{eq:involution:rewrite} gives
\[
 W_{\mathfrak N}W_{\mathfrak d}f(g)
 =\zeta_\chi
  \overline{\chi_2}(\det g)f(gC).
\]
This function has central character
$\overline{\chi'}=\chi_1\overline{\chi_2}$.
Applying the same formula once more, with the compatible auxiliary
characters, gives
\[
 (W_{\mathfrak N}W_{\mathfrak d})^2f
 =  \overline{\chi_1}(\varpi_{\mathfrak N_0})f.
\]
This proves the assertion.
\end{proof}

\subsection{Analytic continuation of twisted partial \emph{L}-functions}\label{subsec:analytic:conti}
Let $\mathfrak c$ be an integral ideal of $F$.
Recall that $\mathcal R_{\mathfrak c}$ is the subset of $\prod_{v\mid\mathfrak c}F_v^\times$ 
introduced in \S\ref{twisted:cuspforms}.
For every $u\in\mathcal R_{\mathfrak c}$, there is a unique $u'\in\mathcal R_{\mathfrak c}$ such that
\begin{equation}\label{eq:inversion:representative}
 (\varpi_{\mathfrak c}u')(\varpi_{\mathfrak c}u)\equiv-1\pmod{\mathfrak c}. 
\end{equation}
The map $u\mapsto u'$ is a bijection of $\mathcal R_{\mathfrak c}$. 
Indeed, $(u')'=u$ by the definition.

Let $\mathfrak N$ be an integral ideal of $F$. Put
\[
 \mathfrak d:=\mathfrak N+\mathfrak c^2,
 \qquad
 \mathfrak M:=\mathfrak N\cap\mathfrak c^2,
 \qquad
 \mathfrak N_0:=\mathfrak N\mathfrak d^{-1}.
\]
Let $f\in S_2(\mathfrak N,\chi)$ be a cusp form whose central character $\chi$ 
satisfies~\eqref{eq:atkin-lehner-unit-condition}. 
Since $\mathfrak M\subset\mathfrak N$, we also have $f\in S_2(\mathfrak M,\chi)$.
Assume further that
\[
2\operatorname{ord}_v(\mathfrak c)\ge\operatorname{ord}_v(\mathfrak N)\qquad (v\mid\mathfrak d),
\]
or equivalently $\mathfrak d$ is a Hall divisor of $\mathfrak N$.
Let $u\in\mathcal R_{\mathfrak c}$.

\begin{lemma}\label{lem:finite-inversion-matrix-identity}
Assume further that
\begin{equation}\label{eq:assumption:conductor:exponent:large}
\operatorname{ord}_v(\mathfrak c)
\geq
\operatorname{ord}_v(\mathfrak N)
\qquad (v\mid\mathfrak d).
\end{equation}
Then there exists
$A_u\in U_0(\mathfrak N)$ such that
\begin{equation}\label{eq:atkin-lehner-twist-matrix-identity}
 B_{\mathfrak M}t(u)
 =\varpi_{\mathfrak c}t(u')B_{\mathfrak N}B_{\mathfrak d}^{-1}A_u
\end{equation}
and the lower-right entry of $(A_u)_v$ is
$(\varpi_{\mathfrak c})_vu_v$ for $v\mid\mathfrak c$ and $1$ otherwise.
\end{lemma}

\begin{proof}
For $v\mid\mathfrak c$, put
\[
  a_{u,v}:=-(\varpi_{\mathfrak c}u'u)_v-(\varpi_{\mathfrak c})_v^{-1}\in O_v
\]
and define
\[
 (A_u)_v=
 \begin{cases}
  \left(\begin{matrix}
    -(\varpi_{\mathfrak c}u')_v&a_{u,v}
    \\
    (\varpi_{\mathfrak c})_v&(\varpi_{\mathfrak c}u)_v
   \end{matrix}\right),
   &v\mid\mathfrak c,\\
   \left(\begin{smallmatrix}1&0\\ 0&1\end{smallmatrix}\right),
   &\text{otherwise}.
 \end{cases}
\]
By~\eqref{eq:inversion:representative} and the definition of $\mathcal R_{\mathfrak c}$, 
$(A_u)_v\in\mathrm{SL}_2(O_v)$ for each finite place $v$. 
If $v\mid\mathfrak d$, then the lower-left entry $(\varpi_{\mathfrak c})_v$ of $(A_u)_v$
belongs to $\mathfrak N O_v$.
If $v\mid\mathfrak N_0$, then $(A_u)_v=\left(\begin{smallmatrix}1&0\\ 0&1\end{smallmatrix}\right)$.
Thus $A_u\in U_0(\mathfrak N)$.

At every $v\mid\mathfrak c$, direct multiplication gives
\[
 (\varpi_{\mathfrak c})_vt(u')_{v}(A_u)_v
 =\begin{pmatrix}0&-1\\(\varpi_{\mathfrak c})_v^2&(\varpi_{\mathfrak c}^2u)_v\end{pmatrix}
 =\begin{pmatrix}0&-1\\(\varpi_{\mathfrak c})_v^2&0\end{pmatrix}t(u)_v.
\]
At $v\mid\mathfrak N_0$, one has $(\varpi_{\mathfrak c})_v=1$, $t(u)_v=t(u')_v=1$,
and $(B_{\mathfrak M})_v=(B_{\mathfrak N})_v$, while
$(B_{\mathfrak d})_v=1$.
At every remaining place all the matrices
are the identity. This proves
\eqref{eq:atkin-lehner-twist-matrix-identity} at every finite place.
\end{proof}

For the remainder of this subsection, retain the hypothesis of
Lemma~\ref{lem:finite-inversion-matrix-identity}.

\begin{lemma}\label{lem:translated-atkin-lehner-identity}
Let $u'$ be as in~\eqref{eq:inversion:representative}. 
Then, for every
$g\in\mathrm{GL}_2(\mathbb A_F)$,
\[
 \widetilde{W_{\mathfrak M}^{\chi}(f\vert_u)}(g)
 =-
 \frac{
  \chi(\varpi_{\mathfrak c})
  \chi_{\mathfrak d}(\varpi_{\mathfrak c}u)}{
  \overline{\chi_1}(\varpi_{\mathfrak N})
  \chi(\varpi_{\mathfrak d})}
 \overline{\chi_1}(\det g)\,
 \widetilde{
  \left(W_{\mathfrak N}W_{\mathfrak d}f\right)
  \vert_{u'}}(g).
\]
In particular, for every
$y\in\mathbb C^\times$,
\[
 \widetilde{\bigl(W_{\mathfrak M}^{\chi}(f\vert_u)\bigr)_a}(y)
 =-
 \frac{
  \chi(\varpi_{\mathfrak c})
  \chi_{\mathfrak d}(\varpi_{\mathfrak c}u)}{
  \overline{\chi_1}(\varpi_{\mathfrak N})
  \chi(\varpi_{\mathfrak d})}
 \overline{\chi_1}(a)\,
 \widetilde{
  \left(
   \left(W_{\mathfrak N}W_{\mathfrak d}f\right)
   \vert_{u'}
  \right)_a}(y).
\]
\end{lemma}

\begin{proof}
Lemma~\ref{lem:finite-inversion-matrix-identity} and the
transformation law of $f$ give
\[
 f(gB_{\mathfrak M}t(u))
 =\chi(\varpi_{\mathfrak c})
  \chi_{\mathfrak d}(\varpi_{\mathfrak c}u)
  f(gt(u')B_{\mathfrak N}B_{\mathfrak d}^{-1}).
\]
On the other hand, \eqref{eq:involution:rewrite} gives
\[
 W_{\mathfrak N}W_{\mathfrak d}f(g)
 =\overline{\chi_1}(\varpi_{\mathfrak N})
  \chi(\varpi_{\mathfrak d})
  \overline\chi_2(\det g)
  f(gB_{\mathfrak N}B_{\mathfrak d}^{-1}).
\]
Combining these identities with the definition of
$W_{\mathfrak M}^{\chi}$ proves the assertion.
\end{proof}

\begin{lemma}
\label{functional:equation:lemma}
For a finite idele $a$ of $F$, $y\in\mathbb{C}^\times$, and $s\in\mathbb{C}$,
\[
N(\mathfrak{M})^{\frac{s}{2}}\widetilde{(f\vert_u)_a}(y^{-1})|ay^{-1}|_{\mathbb A_F}^{s-1}
=
N(\mathfrak{M})^{1-\frac{s}{2}}
\widetilde{W_\mathfrak{M}^\chi(f\vert_u)_{\varpi_\mathfrak{M}/a}}(y)|
(\varpi_\mathfrak{M}/a)y|_{\mathbb A_F}^{1-s}.
\]
\end{lemma}

\begin{proof}
In this proof, write
\[
B:=\left(\begin{matrix} 0 & 1 \\ -1 & 0 \end{matrix}\right)
\left(\begin{matrix} \varpi_\mathfrak{M}a^{-1} & 0 \\ 0 & 1 \end{matrix}
\right)
\left(\begin{matrix} y & 0 \\ 0 & 1 \end{matrix}\right)B_\mathfrak{M}t(u),
\qquad m_v=\mathrm{ord}_v(\mathfrak M).
\]
Here,
$\left(\begin{smallmatrix} 0 & 1 \\ -1 & 0 \end{smallmatrix}\right)
\in\mathrm{GL}_2(F)$. 
If $v\mid \mathfrak M$, then
\[
B_v=\left(\begin{matrix} 0 & 1 \\ -1 & 0 \end{matrix}\right)_v
\left(\begin{matrix} \varpi_v^{m_v}a_v^{-1} & 0 \\ 0 & 1 \end{matrix}\right)
\left(\begin{matrix} 0 & -1 \\ \varpi_v^{m_v} & 0 \end{matrix}\right)t(u_v)
=\varpi_v^{m_v}a_v^{-1}\left(\begin{matrix} a_v & 0 \\ 0 & 1 \end{matrix}\right)t(u_v).
\]
Otherwise, 
\[
B_v=(a^{-1}y)_v\left(\begin{matrix} (ay^{-1})_v & 0 \\ 0 & 1 \end{matrix}\right)
\left(\begin{matrix} 0 & 1 \\ -1 & 0 \end{matrix}\right)_v.
\]
In sum,
\[
B
=
a^{-1}\varpi_{\mathfrak{M}}y
\left(\begin{matrix} a & 0 \\ 0 & 1 \end{matrix}\right)
\left(\begin{matrix} y^{-1} & 0 \\ 0 & 1 \end{matrix}\right)t(u)
\prod_{v\nmid \mathfrak{M}}
\left(\begin{matrix} 0 & 1 \\ -1 & 0 \end{matrix}\right)_v,
\]
Using the definition of $\widetilde{W_\mathfrak{M}^\chi(f\vert_u)_{\varpi_\mathfrak{M}/a}}$ 
and the above expression for $B$, we obtain
\[
\widetilde{W_\mathfrak{M}^\chi(f\vert_u)_{\varpi_\mathfrak{M}/a}}(y)=
-\widetilde f\left(\left(\begin{matrix} a & 0 \\ 0 & 1 \end{matrix}\right)
\left(\begin{matrix} y^{-1} & 0 \\ 0 & 1 \end{matrix}\right)t(u)
\left(\begin{matrix} 0 & 1 \\ -1 & 0 \end{matrix}\right)_\infty
\right).
\]
Together with the property of $f$ and the equation 
\[
N(\mathfrak M)^{1-\frac s2}
\left|(\varpi_{\mathfrak M}/a)y\right|_{\mathbb A_F}^{1-s}
=
N(\mathfrak M)^{\frac s2}
|ay^{-1}|_{\mathbb A_F}^{s-1},
\]
we are done.
\end{proof}
For a finite idele $a$, define the completed additively twisted partial $L$-function by
\[
 \Lambda_a(s,f;u)
 :=N(\mathfrak M)^{\frac{s}{2}}\Gamma_F(s)L_a(s,f;u).
\]
Put
\[
 \epsilon_a^\chi(u):=-
 \frac{
 \chi(\varpi_{\mathfrak c})
 \chi_{\mathfrak d}(\varpi_{\mathfrak c}u)
 \overline{\chi_1}(\varpi_{\mathfrak M}/a)}
 {\overline{\chi_1}(\varpi_{\mathfrak N})
 \chi(\varpi_{\mathfrak d})}.
\]
\begin{proposition}\label{Lftn:analytic:conti} 
The function $\Lambda_a(s,f;u)$ extends to an entire function of
$s\in\mathbb C$ and satisfies
\[
 \Lambda_a(s,f;u)
 =\epsilon_a^\chi(u)
 \Lambda_{\varpi_{\mathfrak M}/a}
 \left(2-s,
 W_{\mathfrak N}W_{\mathfrak d}f;u'\right).
\]
\end{proposition}

\begin{proof}
Put
\[
 E_1:=\{y\in E:|y|\geq1\},
 \qquad
 E_2:=\{y\in E:|y|\leq1\}.
\]
For $\operatorname{Re}(s)>\frac32+\theta$,
Proposition~\ref{integral:representation:lfunction} gives
\[
 \Lambda_a(s,f;u)
 =N(\mathfrak M)^{\frac{s}{2}}
 \left(\int_{E_1}+\int_{E_2}\right)\widetilde{(f\vert_u)_a}(y)
 |ay|_{\mathbb A_F}^{s-1}\,d^\times y.
\]
In the second integral, make the measure-preserving change of variable
$y\mapsto y^{-1}$. Lemma~\ref{functional:equation:lemma}, followed by
Lemma~\ref{lem:translated-atkin-lehner-identity}, gives
\begin{align}\begin{split}\label{eq:lfunction:analytic:continuation}
 \Lambda_a(s,f;u)
 ={}&N(\mathfrak M)^{\frac{s}{2}}
 \int_{E_1}\widetilde{(f\vert_u)_a}(y)
 |ay|_{\mathbb A_F}^{s-1}\,d^\times y\\
 +\epsilon_a^\chi(u)&N(\mathfrak M)^{1-\frac{s}{2}}
 \int_{E_1}
 \widetilde{
 \left(
  \left(W_{\mathfrak N}W_{\mathfrak d}f\right)
  \vert_{u'}
 \right)_{\varpi_{\mathfrak M}/a}}(y)
 \left|(\varpi_{\mathfrak M}/a)y\right|_{\mathbb A_F}^{1-s}
 \,d^\times y.
\end{split}\end{align}
Both integrands are rapidly decreasing on $E_1$. 
Hence the integrals define an entire function of $s$ and give the analytic continuation of $\Lambda_a(s,f;u)$.

The definitions and the compatible choices give
\[
 \epsilon_a^\chi(u)\epsilon_{\varpi_{\mathfrak M}/a}^{\overline{\chi'}}(u')=\chi_1(\varpi_{\mathfrak N_0}).
\]
Apply the substitution
\[
 (s,a,f,u)\longmapsto
 \left(2-s,\varpi_{\mathfrak M}/a,
 W_{\mathfrak N}W_{\mathfrak d}f,u'\right)
\]
to~\eqref{eq:lfunction:analytic:continuation}. Since $u\mapsto u'$ is an involution and
$(W_{\mathfrak N}W_{\mathfrak d})^2=\overline{\chi_1}(\varpi_{\mathfrak N_0})$ 
by Lemma~\ref{lem:composite-atkin-lehner-involution},
multiplication by $\epsilon_a^\chi(u)$ interchanges the two terms in that formula.
This proves the functional equation.
\end{proof}

\begin{remark}\label{rem:mellin-transform-extension}
The Mellin transform on the left-hand side of
Proposition~\ref{integral:representation:lfunction} is well defined and
holomorphic for $\operatorname{Re}(s)>0$. By Proposition~\ref{Lftn:analytic:conti},
$L_a(s,f;u)$ has an analytic continuation. Since the two sides agree for
$\operatorname{Re}(s)>\frac32+\theta$, the
equality in Proposition~\ref{integral:representation:lfunction} can be extended to the
half-plane $\operatorname{Re}(s)>0$.
\end{remark}

\subsection{Approximate functional equations}

We record the approximate functional equation. Choose a smooth compactly supported function $\Psi$ on
$\mathbb R_{>0}$ and put
\[
 \Phi(t)=\int_0^\infty\Psi(x)x^t\frac{dx}{x}.
\]
Normalize $\Psi$ so that $\Phi(0)=1$.
For $j=1,2$, define
\[
 V_j(x)
 :=\frac{1}{2\pi i\,\Gamma_F(1)}
 \int_{2+i\mathbb R}
 \Phi((-1)^{j-1}t)\Gamma_F(1+t)x^{-t}\frac{dt}{t}.
\]
The functions $V_j$ are bounded and rapidly decreasing as
$x\rightarrow\infty$.

Choose $\beta\in F^\times$ such that the finite idele
$a':=\beta\varpi_{\mathfrak M}/a$
is prime to $\mathfrak c$.

\begin{proposition}\label{prop:additive-approximate-functional-equation}
If the condition~\eqref{eq:assumption:conductor:exponent:large} is satisfied, then for every $y>0$,
\begin{align*}
 L_a(1,f;u)
 ={}&\frac{1}{\#E_F}\sum_{\alpha\in F^\times}
 \frac{a_f(a\alpha\mathfrak d_F)\mathbf e_F(\alpha a u)}
 {N(a\alpha\mathfrak d_F)}
 V_1\left(\frac{N(a\alpha\mathfrak d_F)}{y}\right)\\
 &+\frac{\epsilon_a^\chi(u)}{\#E_F}
 \sum_{\alpha\in F^\times}
 \frac{a_{W_{\mathfrak N}W_{\mathfrak d}f}
 (a'\alpha\mathfrak d_F)\mathbf e_F(\alpha a'u')}
 {N(a'\alpha\mathfrak d_F)}
 V_2\left(
 \frac{N(a'\alpha\mathfrak d_F)y}{N(\mathfrak M)}
 \right).
\end{align*}
\end{proposition}

\begin{proof}
Consider
\[
 \frac{1}{2\pi i}\int_{2+i\mathbb R}
 \Phi(t)\frac{\Gamma_F(1+t)}{\Gamma_F(1)}
 L_a(1+t,f;u)y^t\frac{dt}{t}.
\]
Expanding the $L$-function and integrating term by term gives the
first sum. Shift the contour to $-2+i\mathbb R$. The pole at $t=0$
has residue $L_a(1,f;u)$. Proposition~\ref{Lftn:analytic:conti}
gives
\[
 L_a(1+t,f;u)
 =\epsilon_a^\chi(u)N(\mathfrak M)^{-t}
 \frac{\Gamma_F(1-t)}{\Gamma_F(1+t)}
 L_{\varpi_{\mathfrak M}/a}\left(1-t,W_{\mathfrak N}W_{\mathfrak d}f;u'\right).
\]
Making the change of variable $\alpha\mapsto\beta\alpha$, we obtain
\[
 L_{\varpi_{\mathfrak M}/a}\left(1-t,W_{\mathfrak N}W_{\mathfrak d}f;u'\right)=
 L_{a'}\left(1-t,W_{\mathfrak N}W_{\mathfrak d}f;u'\right).
\]
Use this identity on the shifted contour and then replace $t$ by $-t$.
Expanding this gives the second sum.
\end{proof}

\section{Vertical limits of twisted {\it L}-values}
\label{sec:vertical-limits}
\label{section:vertical-limits-of-l-values}

This section proves a vertical limit formula for averages of additively twisted partial \(L\)-values.
The argument is a ramified-level variant of \cite[Section~3.3]{Kwon:2025}. 

Let $\ell$ be an odd prime that splits in $F$, and let
$\mathfrak{l}$ be a prime ideal of $F$ above $\ell$.
Set $\Gamma_m:=1+\mathfrak l^mO_{\mathfrak l}$ for $m>0$.
Let $n_0>0$ be an integer.

\subsection{Additive characters and their averages}\label{subsec:char:aver}
Let $c\in O_{\mathfrak l}^{\times}$ in this subsection.
\begin{lemma}
\label{lem:additive-character-multiplicative-coset-average}
Let $b\in O_{\mathfrak l}$. For any $n\geq n_0$,
\[
 \sum_{a\in c\Gamma_{n_0}/\Gamma_n}
 \mathbf e_F\left(\frac{ab}{\varpi_{\mathfrak l}^n}\right)
 =
 \begin{cases}
  \ell^{n-n_0}
  \mathbf e_F\left(\dfrac{bc}{\varpi_{\mathfrak l}^n}\right),
  &b\in\mathfrak l^{n-n_0}O_{\mathfrak l},\\
  0,& \text{otherwise}.
 \end{cases}
\]
\end{lemma}

\begin{proof}
Write
\[
 a=c(1+\varpi_{\mathfrak l}^{n_0}x),
 \qquad
 x\in O_{\mathfrak l}/\mathfrak l^{n-n_0}O_{\mathfrak l}.
\]
Then the sum can be written as
\[
 \mathbf e_F\left(\frac{bc}{\varpi_{\mathfrak l}^n}\right)
 \sum_{x\;\mathrm{mod}\;\mathfrak l^{n-n_0}}
 \mathbf e_F\left(
  \frac{bcx}{\varpi_{\mathfrak l}^{n-n_0}}
 \right).
\]
Applying the orthogonality of characters, then we are done.
\end{proof}

\begin{lemma}
\label{lem:twisted-local-kloosterman-bound}
Let $\rho$ be a finite-order character of $O_{\mathfrak l}^{\times}/\Gamma_e$ for $e>0$. 
Let $b,d\in O_{\mathfrak l}$, not both zero. Set
\[
 r:=\min\{\operatorname{ord}_{\mathfrak l}(b),
           \operatorname{ord}_{\mathfrak l}(d)\}.
\]
If $n-r\geq2\max\{n_0,e\}$, then
\[
 \frac{1}{\ell^{n-n_0}}
 \left|\sum_{a\in c\Gamma_{n_0}/\Gamma_n}
 \rho(a)\mathbf e_F\left(
 \frac{ba+da^{-1}}{\varpi_{\mathfrak l}^{n}}\right)\right|
 \leq 2\ell^{n_0-\lfloor(n-r)/2\rfloor}.
\]
\end{lemma}

\begin{proof}
Put $h=n-r$, $b_0=\varpi_{\mathfrak l}^{-r}b$,
and $d_0=\varpi_{\mathfrak l}^{-r}d$.
Since $e\leq h$, the summand is invariant under $\Gamma_h$.
Reduction modulo $\mathfrak l^h$ therefore gives
\[
 \frac{1}{\ell^{n-n_0}}
 \sum_{a\in c\Gamma_{n_0}/\Gamma_n}
 \rho(a)\mathbf e_F\left(
 \frac{ba+da^{-1}}{\varpi_{\mathfrak l}^{n}}\right)
 =\frac{1}{\ell^{h-n_0}}
 \sum_{a\in c\Gamma_{n_0}/\Gamma_h}
 \rho(a)\mathbf e_F\left(
 \frac{b_0a+d_0a^{-1}}{\varpi_{\mathfrak l}^{h}}\right).
\]
Note that $h\geq n_0$. Set $m=\lceil h/2\rceil$ and write
\[
 a=t(1+\varpi_{\mathfrak l}^{m}x),\qquad
 t\in c\Gamma_{n_0}/\Gamma_m,\qquad
 x\in O_{\mathfrak l}/\mathfrak l^{h-m}O_{\mathfrak l}.
\]
Since $m\geq e$, we have $\rho(a)=\rho(t)$.
Since $2m\geq h$, we also have
\[
 b_0a+d_0a^{-1}
 \equiv b_0t+d_0t^{-1}
 +\varpi_{\mathfrak l}^{m}x(b_0t-d_0t^{-1})
 \pmod{\mathfrak l^h}.
\]
Additive orthogonality rewrites the unnormalized sum modulo
$\mathfrak l^h$ as
\[
 \ell^{h-m}
 \sum_{\substack{t\in c\Gamma_{n_0}/\Gamma_m\\
 b_0t^2\equiv d_0\; (\mathrm{mod}\;\mathfrak l^{h-m})}}
 \rho(t)\mathbf e_F\left(
 \frac{b_0t+d_0t^{-1}}{\varpi_{\mathfrak l}^{h}}\right).
\]
At least one of $b_0,d_0$ is a unit, and $h-m\geq1$.
If exactly one is a unit, the congruence has no solution.
Otherwise, since $\ell$ is odd, it has at most two solutions
modulo $\mathfrak l^{h-m}$.
These lift to at most $2\ell^{2m-h}$ classes modulo $\mathfrak l^m$.
The unnormalized sum is thus bounded by $2\ell^m$.
Dividing by $\ell^{h-n_0}$ proves the assertion.
\end{proof}

\subsection{Vertical limit of \emph{L}-values}
Let $f\in S_2(\mathfrak N,\chi)$ be a cusp form and $e:=\mathrm{ord}_{\mathfrak l}(\mathfrak N)$.
Assume that \(2n_0\ge e\), $e>0$, and the conductor of $\chi$ divides $\mathfrak l^e$. 
For an integer \(n\ge2n_0\), set
\[
 \mathfrak c=\mathfrak l^n,
 \qquad
 \mathfrak d=\mathfrak N+\mathfrak c^2=\mathfrak l^e,
 \qquad
 \mathfrak N_0=\mathfrak N\mathfrak d^{-1},
 \qquad
 \mathfrak M=\mathfrak N\cap \mathfrak c^2=\mathfrak N_0\mathfrak l^{2n}.
\]
Then $\mathfrak d$ is a Hall divisor of $\mathfrak N$, and
condition~\eqref{eq:assumption:conductor:exponent:large} holds,
so Proposition~\ref{prop:additive-approximate-functional-equation} is applicable.
Since $\chi$ is a Hecke character unramified outside $\mathfrak l$, 
the local restriction condition~\eqref{eq:atkin-lehner-unit-condition} is satisfied automatically,
and we take $\chi_1=\chi$ and $\chi_2=1$ in the Atkin--Lehner construction (\S\ref{subsec:atkin-lehner-func-eq}).

Recall the additively twisted partial $L$-function from
\S\ref{int:rep:sec}.  For $a\in O_{\mathfrak l}^{\times}$, we
regard $a/\varpi_{\mathfrak l}^n$ as the finite adele whose components
away from $\mathfrak l$ are zero.
For $c\in O_{\mathfrak l}^{\times}$ and $\alpha\in\mathfrak d_F^{-1}\backslash\{0\}$, define
\[
 L_n^{\mathrm{av}}(f;c,\alpha)
 :=\frac{1}{\ell^{n-n_0}}
 \sum_{a\in c\Gamma_{n_0}/\Gamma_n}
 \mathbf e_F\left(-\frac{\alpha a}{\varpi_{\mathfrak l}^n}\right)
 L_1\left(1,f;\frac{a}{\varpi_{\mathfrak l}^n}\right).
\]
By~\eqref{eq:inversion:representative}, 
we can choose $(a/\varpi_{\mathfrak l}^n)'=-a^{-1}/\varpi_{\mathfrak l}^n$. 
Choose $\beta_0\in F^\times$ such that
\[
  b_{\mathfrak M}:=\beta_0\varpi_{\mathfrak M}
\]
is prime to $\mathfrak l$. For
$\alpha\in\mathfrak d_F^{-1}$, $\beta\in(b_{\mathfrak M}\mathfrak d_F)^{-1}$ and $c\in O_{\mathfrak l}^\times$,
define the additive kernels:
\begin{align}\label{eq:additive-dual-kloosterman-kernel}
\mathcal A_n(c;\alpha)
 &:=\frac{1}{\ell^{n-n_0}}
 \sum_{a\in c\Gamma_{n_0}/\Gamma_n}
 \mathbf e_F\left(
  \frac{\alpha a}{\varpi_{\mathfrak l}^n}
 \right), \notag \\
\mathcal A_n^\chi(c;\alpha,\beta)
 &:=\frac{1}{\ell^{n-n_0}}
 \sum_{a\in c\Gamma_{n_0}/\Gamma_n}
 \mathbf e_F\left(
  \frac{-\alpha a-\beta b_{\mathfrak M}a^{-1}}
       {\varpi_{\mathfrak l}^n}
 \right)
 \epsilon_1^\chi\left(\frac{a}{\varpi_{\mathfrak l}^n}\right),
\end{align}
where $\epsilon_1^\chi(u)$ is the factor defined immediately before
Proposition~\ref{Lftn:analytic:conti}.
Applying Proposition~\ref{prop:additive-approximate-functional-equation}
to the partial $L$-function $L_1(s,f;a/\varpi_{\mathfrak l}^n)$ and
taking the weighted average over $a$, we obtain
\begin{equation}\label{eq:additive-average-two-pieces}
 L_n^{\mathrm{av}}(f;c,\alpha)
 =\mathcal L_1(y)
  +\mathcal L_2(y)\qquad (y>0),
\end{equation}
where
\begin{align}
 \mathcal L_1(y)
 :={}&\frac{1}{\#E_F}
 \sum_{\beta\in F^\times\cap\mathfrak d_F^{-1}}
 \frac{a_f(\beta\mathfrak d_F)}{N(\beta\mathfrak d_F)}
 \mathcal A_n(c;\beta-\alpha)
 V_1\left(\frac{N(\beta\mathfrak d_F)}{y}\right),
 \label{eq:additive-average-first-piece}\\
 \mathcal L_2(y)
 :={}&\frac{1}{\#E_F}
 \sum_{\substack{\beta\in F^\times \\ \beta b_{\mathfrak M}\mathfrak d_F\subset\mathcal O_F}}
 \frac{a_{W_{\mathfrak N}W_{\mathfrak d}f}(b_{\mathfrak M}\beta\mathfrak d_F)}
      {N(b_{\mathfrak M}\beta\mathfrak d_F)}
 \mathcal A_n^\chi(c;\alpha,\beta)
  V_2\left(
  \frac{N(b_{\mathfrak M}\beta\mathfrak d_F)y}
       {N(\mathfrak M)}
  \right).
 \label{eq:additive-average-second-piece}
\end{align}

\begin{proposition}\label{prop:additive-dual-kloosterman-bound}
For fixed $\alpha\in \mathfrak d_F^{-1}\setminus\{0\}$ and all sufficiently large $n$,
we have the bound
\[
 \left|\mathcal A_n^\chi(c;\alpha,\beta)\right|
 \ll_{\ell,n_0,\alpha}\ell^{-n/2}
\]
which holds uniformly in $c\in O_{\mathfrak l}^{\times}$ and $\beta\in (b_{\mathfrak M}\mathfrak d_F)^{-1}$.
\end{proposition}

\begin{proof}
We have
\[
 \epsilon_1^\chi\left(\frac{a}{\varpi_{\mathfrak l}^n}\right)
 =-\overline\chi(\varpi_{\mathfrak c})\chi(a).
\]
Since $\ell$ splits in $F$ and $b_{\mathfrak M}$ is a local unit,
both $\alpha$ and $\beta b_{\mathfrak M}$ are integral at $\mathfrak l$,
and $\operatorname{ord}_{\mathfrak l}(\beta b_{\mathfrak M})
=\operatorname{ord}_{\mathfrak l}(\beta)$.
Note that the restriction of $\chi$ to $\Gamma_e$ is trivial, 
so we apply Lemma~\ref{lem:twisted-local-kloosterman-bound} with
$b=-\alpha$ and $d=-\beta b_{\mathfrak M}$.
If $n\geq\operatorname{ord}_{\mathfrak l}(\alpha)+2\max\{n_0,e\}$, then
the bound holds uniformly in $\beta$.
\end{proof}

The unramified version of the following limit formula, in which
$\mathfrak l$ is prime to the level, is proved in
\cite[Proposition~6.2]{Kwon:2023}.  We record the corresponding
ramified-level statement.

\begin{proposition}
\label{prop:additive-average-vertical-limit}
Let $c\in O_{\mathfrak l}^{\times}$ and
$\alpha\in\mathfrak d_F^{-1}\setminus\{0\}$.  Then
\[
 \lim_{n\to\infty}L_n^{\mathrm{av}}(f;c,\alpha)
 =\frac{1}{\#E_F}
  \frac{a_f(\alpha\mathfrak d_F)}
       {N(\alpha\mathfrak d_F)}.
\]
\end{proposition}

\begin{proof}
The proof closely follows the proof of
\cite[Proposition~6.2]{Kwon:2023}.  We briefly recall the main steps.
Apply \eqref{eq:additive-average-two-pieces} and choose the auxiliary
parameter $y=y_n$ as in the proof cited above.  In
\eqref{eq:additive-average-first-piece},
Lemma~\ref{lem:additive-character-multiplicative-coset-average} shows
that only the terms satisfying
\[
 \beta\in (\alpha+\mathfrak l^{n-n_0}O_{\mathfrak l})\cap\mathfrak d_F^{-1}
\]
can contribute.  Among them, the term 
\[
 \frac{1}{\#E_F}
 \frac{a_f(\alpha\mathfrak d_F)}
      {N(\alpha\mathfrak d_F)}
 V_1\left(\frac{N(\alpha\mathfrak d_F)}{y_n}\right)
\]
at $\beta=\alpha$ is the main term.
This converges to the asserted main term as $n\to\infty$ since
$V_1(x)\to1$ as $x\to0^+$.
After removing this term, the remaining part of
\eqref{eq:additive-average-first-piece} and the whole of
\eqref{eq:additive-average-second-piece} tend to zero.  Indeed, one
uses Lemma~\ref{lem:additive-character-multiplicative-coset-average},
the triangle-inequality argument of
\cite[Proposition~6.2]{Kwon:2023}, the bound of Blomer--Brumley
\cite{BlomerBrumley:2011} for the coefficients of $f$ and
$W_{\mathfrak N}W_{\mathfrak d}f$, and Abel summation.  For the dual
sum, one additionally applies
Proposition~\ref{prop:additive-dual-kloosterman-bound} to
$\mathcal A_n^\chi(c;\alpha,\beta)$, then the result follows.
\end{proof}

\section{Cohomology and Homology of Manifolds}\label{section:conj:on:mod:symb}

\subsection{Universal Coefficient Theorems}
In this subsection, we recall the universal coefficient theorems for ordinary homology and cohomology.
Let \(R\) be a commutative ring with identity, \(C_\bullet\) a chain
complex of free \(R\)-modules, \(A\) an \(R\)-algebra, and \(j\geq0\) an integer. 
We have the following well-defined
$R$-bilinear map
\[
	H_j(C_\bullet)\times A\;\longrightarrow\;
	H_j(C_\bullet\otimes_R A),\qquad ([c],a)\mapsto[c\otimes a].
\]
This induces an \(A\)-linear map
\[
	\beta_A:H_j(C_\bullet)\otimes_R A\;\longrightarrow\;
	H_j(C_\bullet\otimes_R A).
\]
\begin{lemma}\label{universal:coefficient:theorem:homology}
Assume that $R$ is a principal ideal domain. Then we have the following
short exact sequence:
\begin{equation*}
\begin{tikzcd}
0 \arrow{r} &
H_j(C_\bullet)\otimes_R A \arrow{r}{\beta_A} &
H_j(C_\bullet\otimes_R A) \arrow{r} &
\operatorname{Tor}_1^R\!\bigl(H_{j-1}(C_\bullet),A\bigr)
\arrow{r} & 0
\end{tikzcd}.
\end{equation*}
We adopt the convention that \(H_{-1}(C_\bullet)=0\).
The sequence is natural in \(C_\bullet\) and \(A\).
\end{lemma}

Let \(\partial_j:C_j\to C_{j-1}\) denote the boundary maps. Set
\[
C_A^j:=\mathrm{Hom}_R(C_j,A),\qquad C_A^\bullet:=\mathrm{Hom}_R(C_\bullet,A).
\]
We regard $C_A^\bullet$ as a cochain complex with differential
\[
d_A^j:C_A^j\longrightarrow C_A^{j+1},\qquad d_A^j(\varphi):=\varphi\circ\partial_{j+1}.
\]
For a cycle $c\in C_j$ and a cocycle $\delta\in C_A^j$, the value $\delta(c)$ depends only on the classes
$[c]\in H_j(C_\bullet)$ and $[\delta]\in H^j(C_A^\bullet)$. Indeed, if $c$ is replaced by
$c+\partial_{j+1}b$, then $\delta(c+\partial_{j+1}b)=\delta(c)$ because $d_A^j\delta=0$.
Likewise, replacing $\delta$ by $\delta+d_A^{j-1}\eta$ 
does not change $\delta(c)$, since $\partial_jc=0$.
Hence we have a well-defined $R$-bilinear pairing
\[
\cap:H_j(C_\bullet)\times H^j(C_A^\bullet)\;\longrightarrow\; A,\qquad ([c],[\delta])\longmapsto\delta(c).
\]

For an $R$-module $M$, an $A$-module $H$, and 
an $R$-bilinear pairing $\cap:M\times H\to A$ which is $A$-linear in the second variable, let us denote by
\[
\mathcal{D}_\cap: H\;\longrightarrow\; \mathrm{Hom}_R\left(M,A\right),\qquad h\mapsto (m\mapsto m\cap h)
\]
the $A$-linear map induced by this pairing.

\begin{lemma}\label{universal:coefficient:theorem:cohomology}
Assume that $R$ is a principal ideal domain. Then there is a short exact sequence:
\begin{equation*}
\begin{tikzcd}
0 \arrow{r} 
& 
\mathrm{Ext}^1_R\left(H_{j-1}(C_\bullet),A\right) \arrow{r} 
& 
H^j(C_A^\bullet) \arrow{r}{\mathcal{D}_\cap} 
& 
\mathrm{Hom}_R\left(H_j(C_\bullet),A\right) \arrow{r} 
&
0
\end{tikzcd}.
\end{equation*}
We adopt the convention that $H_{-1}(C_\bullet)=0$. The sequence is natural in \(C_\bullet\) and \(A\).
\end{lemma}

\subsection{Singular Homology and Cohomology}
We begin with a brief review of singular homology and cohomology theory.

Let $Y$ be a topological space and $R$ a commutative ring with unity.  
Let $j\geq 0$ be an integer.
A continuous map from the standard $j$-simplex to $Y$ is called a \emph{singular $j$-simplex} of $Y$.  
Let $C_j(Y)$ be the free abelian group generated by the singular $j$-simplices of $Y$, and set
\[
  C_j(Y,R):=C_j(Y)\otimes_{\mathbb{Z}}R.
\]
The usual boundary maps on singular simplices induce maps $\partial_j: C_{j}(Y,R) \to C_{j-1}(Y,R)$, 
making \(C_\bullet(Y,R)\) a chain complex.
The $j$-th {\it (singular) homology} group of $Y$ with coefficients in $R$ is then
\[
  H_j(Y,R) := H_j\left(C_\bullet(Y,R)\right).
\]
The dual modules
\[
  C^j(Y,R) := \mathrm{Hom}_{\mathbb{Z}}(C_j(Y),R)
\]
form a cochain complex $C^\bullet(Y,R)$, with coboundary maps
induced by the boundary maps on \(C_\bullet(Y)\). 
Its cohomology groups
\[
  H^j(Y,R) := H^j\left(C^\bullet(Y,R)\right)
\]
are called the {\it (singular) cohomology} groups of $Y$ with coefficients in $R$.
The canonical homomorphism \(\mathbb Z\to R\) induces coefficient-change maps
\[
H_j(Y,\mathbb{Z}) \;\longrightarrow\; H_j(Y,R),\quad H^j(Y,\mathbb{Z}) \;\longrightarrow\; H^j(Y,R),
\]
which are induced by tensoring chains with $R$ and by the canonical map $C^j(Y,\mathbb{Z})\to C^j(Y,R)$, respectively.

There is a well-defined $R$-bilinear pairing
\[
  \cap : H_j(Y,R) \times H^j(Y,R) \;\longrightarrow\; R, 
		\quad ([c],[\delta]) \longmapsto \delta(c).
\]
This is called the \emph{evaluation pairing}.
The evaluation pairing is natural with respect to change of coefficients.

\subsection{De Rham cohomology}
Let \(Y\) be a smooth manifold, and let
\(\Omega^j(Y)\) denote the real vector space of smooth
differential \(j\)-forms on \(Y\). These spaces form the de Rham
complex
\[
  \Omega^0(Y)\xrightarrow{\;d\;}\Omega^1(Y)
  \xrightarrow{\;d\;}\Omega^2(Y)\xrightarrow{\;d\;}\cdots.
\]
Let \(R\) be either \(\mathbb R\) or \(\mathbb C\).
The \emph{de Rham cohomology group} of $Y$ (with coefficients in $R$) is defined by
\[
  H^j_{\mathrm{dR}}(Y,R) := H^j(\Omega^\bullet(Y)\otimes_\mathbb{R}R,d\otimes 1).
\]
By the de Rham theorem, there is a canonical isomorphism
\[
  H^j(Y,R) \xrightarrow{\;\simeq\;} H^j_{\mathrm{dR}}(Y,R).
\]
Let \(C_\bullet^\infty(Y)\) denote the smooth singular chain complex.
Its inclusion into \(C_\bullet(Y)\) induces an isomorphism on homology in every degree.
Using this identification, we define the integration pairing by
\[
  \cap:H_j(Y,\mathbb Z)\times H^j_{\mathrm{dR}}(Y,R)
\longrightarrow R,
\qquad
([c],[\omega])\longmapsto\int_c\omega.
\]
By Stokes' theorem, this pairing is well-defined.
The pairing $\cap$ is \(\mathbb Z\)-linear in the first variable
and \(R\)-linear in the second.
Under the de Rham isomorphism, this pairing agrees with the evaluation pairing in 
singular homology and cohomology.

From now on, when no confusion arises, we use $H^j(Y,R)$ to denote the de Rham cohomology group,
by a slight abuse of notation.

\subsection{Borel--Moore Homology and Cohomology with compact support}

Let $Y$ be a locally compact space and $R$ a commutative ring with unity.
Let \(C_j^{\rm BM}(Y,R)\) be the \(R\)-module of locally finite
formal sums
\[
\sum_\sigma z_\sigma\sigma,
\qquad z_\sigma\in R,
\]
where \(\sigma\) ranges over the singular \(j\)-simplices of \(Y\).
Local finiteness means that, for every compact subset \(K\subset Y\),
only finitely many simplices with nonzero coefficient meet \(K\).
The usual boundary operators make \(C_\bullet^{\rm BM}(Y,R)\)
a chain complex. The \(j\)-th \emph{Borel--Moore homology group}
of \(Y\) with coefficients in \(R\) is defined by
\[
  H_j^{\rm BM}(Y,R) := H_j\left(C_\bullet^{\rm BM}(Y,R)\right).
\]
 
Let $C_{\rm c}^\bullet(Y,R)\subset C^\bullet(Y,R)$
be the subcomplex consisting of compactly supported cochains.
Thus, \(\delta\in C_{\rm c}^j(Y,R)\) if there exists a compact
subset \(K\subset Y\) such that \(\delta\) vanishes on every
singular \(j\)-chain supported in \(Y\setminus K\).
The \emph{$j$-th cohomology group with compact support} is defined by
\[
  H^j_{\mathrm{c}}(Y,R) := H^j\left(C^\bullet_{\mathrm{c}}(Y,R)\right).
\]
A proper map
\(f:Y_1\to Y_2\) induces the pullback
\[
  f^\ast : H^j_{\mathrm{c}}(Y_2,R) \;\longrightarrow\; H^j_{\mathrm{c}}(Y_1,R).
\]

Suppose now that $Y$ is a connected, oriented, $d$-dimensional manifold without boundary.
Then there is a canonical trace map
\[
\mathrm{Tr}_Y : H^d_{\mathrm{c}}(Y,R) \;\longrightarrow\; R.
\]
This map is given by evaluation on the fundamental class of \(Y\). 
If \(A\) is an \(R\)-algebra, the trace maps are
compatible with the coefficient-change maps, as expressed by the
following commutative diagram.
\begin{equation}\label{cpt.supp.cohom.:identification:scalar:extension:compatible}
\begin{tikzcd}
H^d_{\mathrm{c}}(Y,R) \arrow{r}{\mathrm{Tr}_Y} \arrow{d} & R \arrow{d}  \\
H^d_{\mathrm{c}}(Y,A) \arrow{r}{\mathrm{Tr}_Y} & A,
\end{tikzcd}
\end{equation}

When $R=\mathbb{R}$ or $\mathbb{C}$, compactly supported cohomology 
is naturally isomorphic to the corresponding compactly supported de Rham cohomology. 
Under this identification, the trace map is given by integration:
\[
 \mathrm{Tr}_Y: H^d_{\mathrm{c}}(Y,R) \;\longrightarrow\; R,
  \qquad \omega \mapsto \int_Y \omega.
\]

Let $z=\sum_\sigma z_\sigma\sigma\in C_j^{\rm BM}(Y,R)$ and
$\delta\in C_{\rm c}^j(Y,R)$. Since $z$ is locally finite and
$\delta$ has compact support, the sum
\[
z\cap\delta:=\sum_\sigma z_\sigma\delta(\sigma)
\]
has only finitely many nonzero terms. Moreover, for
$w\in C_{j+1}^{\rm BM}(Y,R)$ and $\delta\in C_{\rm c}^j(Y,R)$,
\[
\partial w\cap\delta=w\cap d\delta.
\]
This follows directly from the definition
$(d\delta)(w)=\delta(\partial w)$ of the coboundary operator.
Hence evaluation induces a well-defined $R$-bilinear pairing
\[
\cap:H_j^{\rm BM}(Y,R)\times H_{\rm c}^j(Y,R)
\longrightarrow R,
\qquad ([z],[\delta])\longmapsto z\cap\delta.
\]
\begin{remark}\label{remark:borel-moore-relative-chain-model}
Suppose that $Y$ admits a Borel--Serre compactification
$\overline Y^{\rm BS}$. Then there are canonical identifications:
\[
H_j^{\rm BM}(Y,R)
\simeq
H_j(\overline Y^{\rm BS},\partial\overline Y^{\rm BS};R),
\qquad
H_{\rm c}^j(Y,R)
\simeq
H^j(\overline Y^{\rm BS},\partial\overline Y^{\rm BS};R).
\]
Consequently, Borel--Moore homology is computed by the free relative
chain complex
\[
C_\bullet(\overline Y^{\rm BS},
\partial\overline Y^{\rm BS};R),
\]
and compactly supported cohomology is computed by its dual. Thus the
universal coefficient theorems apply to these theories through this
relative model.
\end{remark}

There is a natural inclusion of chain complexes
\[
C_\bullet(Y,R)\longrightarrow C^{\rm BM}_\bullet(Y,R).
\]
The Borel--Moore and ordinary evaluation
pairings are compatible with these maps.

\subsection{Parabolic (Co)Homology and Duality Pairing}\label{subsec:par:cohom:dual}
Let $Y$ be a locally compact space and $R$ a commutative ring with unity.
The inclusion $C^\bullet_{\mathrm{c}}(Y,R)\subset C^\bullet(Y,R)$ induces a natural map
\[
  H^j_{\mathrm{c}}(Y,R) \;\longrightarrow\; H^j(Y,R).
\]
The \(j\)-th parabolic cohomology group is defined by
\[
H_!^j(Y,R):=
\operatorname{im}\bigl(H_{\rm c}^j(Y,R)\to H^j(Y,R)\bigr).
\]
Similarly, define the \(j\)-th parabolic homology group by
\[
H_j^!(Y,R):=
\operatorname{im}\bigl(H_j(Y,R)\to H_j^{\rm BM}(Y,R)\bigr).
\]

Assume that \(Y\) admits a Borel--Serre compactification.
Whenever we apply the universal coefficient theorems to
Borel--Moore homology or compactly supported cohomology, we use the
relative chain model described in
Remark~\ref{remark:borel-moore-relative-chain-model}.
For classes $c\in H_j^!(Y,R)$ and $\delta\in H^j_!(Y,R)$, choose a lift
$c'\in H_j(Y,R)$ of $c$ under the map $H_j(Y,R)\to H_j^!(Y,R)$ and define
\[
c\cap\delta:=c'\cap\delta,
\]
where the pairing on the right is the evaluation pairing between $H_j(Y,R)$ and $H^j(Y,R)$.
This gives a pairing
\[
\cap:H_j^!(Y,R)\times H^j_!(Y,R)\longrightarrow R.
\]
Let us explain why this pairing is well-defined.
Let $\overline{Y}^{\rm BS}$ 
be the Borel--Serre compactification of $Y$. The map
$H_j(Y,R)\to H_j^{\rm BM}(Y,R)$ fits into the localization exact sequence
\begin{equation}\label{boundary:exact:sequence}
\begin{tikzcd}
H_j(\partial\overline{Y}^{\rm BS},R) \arrow{r} &
H_j(Y,R) \arrow{r} &
H_j^{\rm BM}(Y,R) \arrow{r} &
H_{j-1}(\partial\overline{Y}^{\rm BS},R).
\end{tikzcd}
\end{equation}
Two lifts of the same class in \(H_j^!(Y,R)\) differ by a class
in the image of
\[
H_j(\partial\overline Y^{\rm BS},R)\longrightarrow H_j(Y,R).
\]
On the other hand, every parabolic cohomology class restricts to
zero on \(\partial\overline Y^{\rm BS}\). By naturality of the
evaluation pairing, it therefore pairs trivially with every
homology class coming from the boundary. Hence the definition is
independent of the choice of the lifting \(c'\).

\begin{lemma}\label{parabolic:duality:pairing:lemma}
Let $A$ be a Dedekind domain. Suppose that
\[
\operatorname{Tor}_1^{\mathbb{Z}}
\bigl(H_{j-1}(Y,\mathbb{Z}),A\bigr)
=
\operatorname{Tor}_1^{\mathbb{Z}}
\bigl(H_{j-1}^{\rm BM}(Y,\mathbb{Z}),A\bigr)=0,
\]
\[
\operatorname{Ext}_{\mathbb{Z}}^1
\bigl(H_{j-1}(Y,\mathbb{Z}),A\bigr)
=
\operatorname{Ext}_{\mathbb{Z}}^1
\bigl(H_{j-1}^{\rm BM}(Y,\mathbb{Z}),A\bigr)=0,
\]
and
\[
\operatorname{Ext}_A^1\left(
\operatorname{coker}\left(H_j(Y,A)\longrightarrow
H_j^{\rm BM}(Y,A)\right),A\right)=0.
\]
Then the parabolic pairing induces an isomorphism
\[
\mathcal{D}_\cap:H^j_!(Y,A)\;\longrightarrow\;
\operatorname{Hom}_A\bigl(H_j^!(Y,A),A\bigr),
\qquad \delta\longmapsto(c\mapsto c\cap\delta).
\]
\end{lemma}

\begin{proof}
Write $V^\vee:=\operatorname{Hom}_A(V,A)$
for an $A$-module $V$. Apply
Lemma~\ref{universal:coefficient:theorem:homology} over
$\mathbb{Z}$ with coefficients in $A$. The first vanishing
assumption implies that the following scalar-extension maps
are isomorphisms
\[
\beta_A:H_j(Y,\mathbb{Z})\otimes_{\mathbb{Z}}A
\xrightarrow{\ \simeq\ }H_j(Y,A), \qquad
\beta_A^{\rm BM}:H_j^{\rm BM}(Y,\mathbb{Z})
\otimes_{\mathbb{Z}}A
\xrightarrow{\ \simeq\ }H_j^{\rm BM}(Y,A).  
\]
Tensor--Hom adjunction, 
evaluation maps and the preceding scalar-extension isomorphisms give the commutative diagrams
\[
\begin{tikzcd}
H^j(Y,A) \arrow{r}{\mathcal D_\cap}
\arrow{d}{\mathcal D_\cap^{\mathbb Z}} &
H_j(Y,A)^\vee \arrow{d}{\beta_A^\vee} \\
\operatorname{Hom}_{\mathbb Z}(H_j(Y,\mathbb Z),A)
\arrow{r}{\operatorname{adj}} &
\left(H_j(Y,\mathbb Z)\otimes_{\mathbb Z}A\right)^\vee
\end{tikzcd}
\]
and
\[
\begin{tikzcd}
H_{\rm c}^j(Y,A) \arrow{r}{\mathcal D_\cap}
\arrow{d}{\mathcal D_\cap^{\mathbb Z}} &
H_j^{\rm BM}(Y,A)^\vee \arrow{d}{(\beta_A^{\rm BM})^\vee} \\
\operatorname{Hom}_{\mathbb Z}(H_j^{\rm BM}(Y,\mathbb Z),A)
\arrow{r}{\operatorname{adj}} &
\left(H_j^{\rm BM}(Y,\mathbb Z)\otimes_{\mathbb Z}A\right)^\vee
\end{tikzcd}.
\]
The vertical maps $\mathcal D_\cap^{\mathbb Z}$ are isomorphisms
for both the ordinary and the Borel--Moore settings
by Lemma~\ref{universal:coefficient:theorem:cohomology} with the second vanishing assumption.
In the ordinary case, for
$[\delta]\in H^j(Y,A)$, $[c]\in H_j(Y,\mathbb Z)$, and $a\in A$,
commutativity gives
\begin{equation}\label{eq:eval:map:compatible}
\mathcal D_\cap([\delta])\left(\beta_A([c]\otimes a)\right)
=a\,\mathcal D_\cap^{\mathbb Z}([\delta])([c])
=a\,\delta(c).
\end{equation}
The analogous identity in the Borel--Moore setting holds for
$[c]\in H_j^{\rm BM}(Y,\mathbb Z)$ and
$[\delta]\in H_{\rm c}^j(Y,A)$, with $\beta_A$ replaced by
$\beta_A^{\rm BM}$.
Let
\[
f_A:H_j(Y,A)\longrightarrow H_j^{\rm BM}(Y,A),
\qquad
\iota_A:H_{\rm c}^j(Y,A)\longrightarrow H^j(Y,A)
\]
be the natural maps.
The compatibility~\eqref{eq:eval:map:compatible} gives
\[
\mathcal D_\cap(\iota_A(\delta))(x)
=f_A^\vee\bigl(\mathcal D_\cap(\delta)\bigr)(x)
\]
for every $x\in H_j(Y,A)$ and $\delta\in H_{\rm c}^j(Y,A)$.
Hence the following diagram commutes, and its horizontal maps are
isomorphisms:
\begin{equation}\label{eq:comm:diag:parabolic}
\begin{tikzcd}
H_{\rm c}^j(Y,A) \arrow{r}{\mathcal D_\cap}
\arrow{d}{\iota_A} &
H_j^{\rm BM}(Y,A)^\vee \arrow{d}{f_A^\vee} \\
H^j(Y,A) \arrow{r}{\mathcal D_\cap} & H_j(Y,A)^\vee
\end{tikzcd}.  
\end{equation}
Put $I:=\operatorname{im}(f_A)=H_j^!(Y,A)$ and
$Q:=\operatorname{coker}(f_A)$. Applying $\vee$ to the exact sequence
\begin{equation}\label{eq:exact:seq:bm}
  0\longrightarrow I\rightarrow H_j^{\rm BM}(Y,A)
\longrightarrow Q\longrightarrow0
\end{equation}
and using the last vanishing assumption that $\operatorname{Ext}_A^1(Q,A)=0$, 
we find that the following restriction map is surjective:
\[
H_j^{\rm BM}(Y,A)^\vee\longrightarrow I^\vee.
\]
Moreover, the surjection $H_j(Y,A)\to I$ induces an
injection $I^\vee\hookrightarrow H_j(Y,A)^\vee$, and
\[
f_A^\vee=
\left(I^\vee\lhook\joinrel\longrightarrow H_j(Y,A)^\vee\right)
\circ
\left(H_j^{\rm BM}(Y,A)^\vee\longrightarrow I^\vee\right).
\]
Consequently, the image of $f_A^\vee$ is canonically isomorphic to
$I^\vee$. The commutative diagram~\eqref{eq:comm:diag:parabolic} therefore
induces an isomorphism
\[
H_!^j(Y,A)=\operatorname{im}(\iota_A)
\xrightarrow{\;\simeq\;}
\operatorname{im}(f_A^\vee)
\xrightarrow{\;\simeq\;}
I^\vee
=
\operatorname{Hom}_A\bigl(H_j^!(Y,A),A\bigr).
\]
This isomorphism sends \(\delta\in H_!^j(Y,A)\) to the functional
$c\longmapsto c\cap\delta$.
It is therefore the parabolic pairing map \(\mathcal D_\cap\).
\end{proof}

\begin{lemma}\label{universal:coefficient:parabolic:dedekind}
Let $A$ be a Dedekind domain and let $k$ be a field equipped with an
$A$-algebra structure. Assume that, for each $B\in\{A,k\}$,
\[
\operatorname{Tor}_1^{\mathbb{Z}}\left(H_{j-1}(Y,\mathbb{Z}),B\right)=
\operatorname{Tor}_1^{\mathbb{Z}}\left(H_{j-1}^{\rm BM}(Y,\mathbb{Z}),B\right)=0,
\]
and that
\[
\operatorname{Tor}_1^A\left(
\operatorname{coker}\left(H_j(Y,A)\longrightarrow
H_j^{\rm BM}(Y,A)\right),k\right)=0.
\]
Then the scalar-extension map
\[
H_j^!(Y,A)\otimes_Ak
\longrightarrow H_j^!(Y,k)
\]
is a well-defined isomorphism.
\end{lemma}

\begin{proof}
By Lemma~\ref{universal:coefficient:theorem:homology} and the first vanishing assumption, 
the scalar-extension maps 
\[\beta_A^{\mathbb Z}: H_j(Y,\mathbb Z)\otimes_{\mathbb Z}A\longrightarrow H_j(Y,A),
\qquad
\beta_k^{\mathbb Z}:H_j(Y,\mathbb Z)\otimes_{\mathbb Z}k\longrightarrow H_j(Y,k)
\]
are isomorphisms.
Therefore, the following diagram is commutative, and its horizontal maps are isomorphisms:
\[
\begin{tikzcd}
\left(H_j(Y,\mathbb{Z})\otimes_{\mathbb{Z}}A\right)\otimes_A k
\arrow{r}{\beta_A^{\mathbb Z}\otimes 1}\arrow{d}{\operatorname{assoc}} &
H_j(Y,A)\otimes_A k \arrow{d}{\beta_k} \\
H_j(Y,\mathbb{Z})\otimes_{\mathbb{Z}}k
\arrow{r}{\beta_k^{\mathbb Z}} & H_j(Y,k)
\end{tikzcd}.
\]
Here, $\operatorname{assoc}$ is the canonical associativity
isomorphism given by
\[
\operatorname{assoc}\left(([c]\otimes a)\otimes\lambda\right)
=[c]\otimes(a\lambda),
\]
and $\beta_k$ is the scalar-extension map.
$\beta_k$ is therefore an isomorphism. 
The same argument for Borel--Moore homology shows that 
the scalar-extension map
\[
\beta_k^{\rm BM}:H_j^{\rm BM}(Y,A)\otimes_A k
\longrightarrow H_j^{\rm BM}(Y,k)
\]
is an isomorphism. Let
\[
f_A:H_j(Y,A)\longrightarrow H_j^{\rm BM}(Y,A),
\qquad
f_k:H_j(Y,k)\longrightarrow H_j^{\rm BM}(Y,k)
\]
be the natural maps. By naturality, the scalar-extension maps fit into the commutative
diagram
\[
\begin{tikzcd}
H_j(Y,A)\otimes_A k \arrow{r}{\beta_k}\arrow{d}{f_A\otimes 1} &
H_j(Y,k) \arrow{d}{f_k} \\
H_j^{\rm BM}(Y,A)\otimes_A k
\arrow{r}{\beta_k^{\rm BM}} &
H_j^{\rm BM}(Y,k)
\end{tikzcd}.
\]
Tensoring the exact sequence~\eqref{eq:exact:seq:bm}
with $k$ and using the last vanishing assumption $\operatorname{Tor}_1^A(\mathrm{Coker}(f_A),k)=0$, 
we obtain an injection
\[
H_j^!(Y,A)\otimes_A k\lhook\joinrel\longrightarrow
H_j^{\rm BM}(Y,A)\otimes_A k.
\]
Its image is precisely $\operatorname{im}(f_A\otimes 1)$. By the second
commutative diagram and its two horizontal isomorphisms, the latter
image is carried isomorphically onto
$\operatorname{im}(f_k)=H_j^!(Y,k)$.
Thus, the induced map in the statement is a well-defined isomorphism.
\end{proof}

\section{Arithmetic orbifolds for \texorpdfstring{$\mathrm{GL}(2)$}{GL(2)} and its (co)homology}
\label{sec:arithmetic-orbifolds}
\label{alg:co:hom:mod:sym} 
Let $\mathscr H_3=\mathbb C\times\mathbb R_{>0}$
be the upper half 3-space.
The group $\mathrm{GL}_2(\mathbb C)$ acts on $\mathscr H_3$ by
\begin{equation*}
\left(\begin{matrix} a & b \\ c & d \end{matrix}\right)\cdot (x,y)
:=
\bigg(\frac{(ax+b)\overline{(cx+d)}+a\overline{c}y^2}{|cx+d|^2+|c|^2y^2},\frac{|ad-bc|y}{|cx+d|^2+|c|^2y^2}\bigg).
\end{equation*}
This gives a \(\mathrm{GL}_2(\mathbb C)\)-equivariant diffeomorphism
\begin{equation}\label{eq:diff:liegp:hyperbolic}
\mathrm{GL}_2(\mathbb{C})/\mathbb{C}^\times\mathrm{U}_2(\mathbb{C})
\longrightarrow
\mathscr{H}_3,
\qquad g\longmapsto g\cdot (0,1).  
\end{equation}

\subsection{Arithmetic orbifolds of GL(2)}\label{subsec:arith:orbi}
Let $U$ be an open compact subgroup of $\mathrm{GL}_2(\mathbb{A}_F^{(\infty)})$.
Define
\[
Y(U):=\mathrm{GL}_2(F)\backslash 
\mathrm{GL}_2(\mathbb{A}_F)/\mathbb{C}^\times\mathrm{U}_2(\mathbb{C})U.
\]
Here, \(\mathrm{GL}_2(F)\) is embedded diagonally in
\(\mathrm{GL}_2(\mathbb A_F)\).
We call \(Y(U)\) the arithmetic orbifold of level \(U\).
Set
\[
\mathrm{Cl}_F(U):=
F^\times\backslash\mathbb A_F^\times/
\mathbb C^\times\det(U).
\]
The determinant induces a map
\[
\det:Y(U)\longrightarrow\mathrm{Cl}_F(U).
\]
Since $\det(U)$ is an open subgroup of $\widehat{\mathcal O}_F^\times$, there is a
nonzero ideal $\mathfrak m$ of $\mathcal O_F$ such that
$1+\widehat{\mathfrak m}\subseteq\det(U)$. Consequently,
$\mathrm{Cl}_F(U)$ is finite. 
Choose a set of representatives
$\mathcal C_F(U)\subset\mathrm{GL}_2(\mathbb A_F^{(\infty)})$
for $\mathrm{Cl}_F(U)$ via the determinant.
Strong approximation for $\mathrm{GL}_2$ then
gives
\begin{equation}\label{eq:strong:approx:GL2:adele}
\mathrm{GL}_2(\mathbb{A}_F)
=
\bigsqcup_{t\in\mathcal C_F(U)}\mathrm{GL}_2(F)t\mathrm{GL}_2(\mathbb C) U.
\end{equation}

For a congruence subgroup \(\Gamma\subset\mathrm{GL}_2(F)\), write
\[
    Y_\Gamma:=\Gamma\backslash 
    \mathrm{GL}_2(\mathbb{C})/\mathbb{C}^\times\mathrm{U}_2(\mathbb{C})
    \simeq\Gamma\backslash\mathscr H_3.
\]
This is called the {\it arithmetic orbifold} attached to $\Gamma$.
Here, the identification is induced by the diffeomorphism~\eqref{eq:diff:liegp:hyperbolic}.
For $t\in\mathcal C_F(U)$, set 
$
\Gamma^t(U):=\mathrm{GL}_2(F)\cap tUt^{-1},
$
which is a congruence subgroup of $\mathrm{GL}_2(F)$. 
Put $Y^t(U):=Y_{\Gamma^t(U)}$.
Then the decomposition~\eqref{eq:strong:approx:GL2:adele} induces a homeomorphism
\begin{equation}\label{eq:arith:mfd:disjoint:union}
\bigsqcup_{t\in\mathcal C_F(U)} Y^t(U)\xrightarrow{\;\;\simeq\;\;} Y(U),
\end{equation}
where the map on the \(t\)-component sends \(g_\infty\) to
\(tg_\infty\).
Under this homeomorphism, $Y^t(U)$ is the fiber over $\det(t)$.
The determinant map is continuous, and $\mathrm{Cl}_F(U)$ is discrete.
Hence every connected component of $Y(U)$ lies in a single fiber.
Since $\mathscr H_3$ is connected, each $Y^t(U)$ is connected. 
Thus the $Y^t(U)$ are precisely the connected components of $Y(U)$, 
and determinant induces a bijection
\[
\pi_0(Y(U))\xrightarrow{\;\sim\;}\mathrm{Cl}_F(U).
\]

\subsection{Hecke-equivariant section of the natural map}\label{subsec:hecke:equiv:section}
Let us set $U$ and $\mathcal C_F(U)$ as the preceding subsection. 
Fix $t\in\mathcal C_F(U)$.
Assume that $U$ is sufficiently small, 
so that every component $Y^t(U)$ is a manifold.
In this case, $Y(U)$ and $Y^t(U)$ are called \emph{arithmetic manifolds}.
By Hida \cite[\S2]{Hida:1999}, there is a Hecke-equivariant section
\[
\iota_t:H^1_!(Y^t(U),\mathbb{C})
\;\longrightarrow\; 
H^1_{\mathrm{c}}(Y^t(U),\mathbb{C})
\]
of the natural map $H^1_{\mathrm{c}}\to H^1_!$.
We recall the construction, since it will be used later.
\begin{proof}[Sketch of construction]
For each cusp \(s\) of
\(Y^t(U)\), choose a cusp neighborhood
\[
  U_s\simeq
  (y_0,\infty)\times\Gamma_s\backslash\mathbb R^{d-1},
  \qquad d=\dim Y^t(U).
\]
Here, \(y\) is the height coordinate at \(s\); thus \(y\to\infty\)
means approaching the cusp.
Choose a closed differential form \(\omega\) representing a
class in \(H_!^1(Y^t(U),\mathbb C)\). 
Then \(\omega\) is rapidly decreasing in the height coordinate at every cusp.
There is therefore a canonical smooth function \(\theta_s\) on \(U_s\) satisfying
\[
  d\theta_s=\omega|_{U_s},
  \qquad
  \lim_{y\to\infty}\theta_s(y,x)=0.
\]
Choose a smooth function \(\phi_s\) supported in \(U_s\) and equal to
\(1\) for sufficiently large \(y\). Then
\[
  \omega':=\omega-d\left(\sum_s\phi_s\theta_s\right)
\]
is compactly supported, and we define
\[
  \iota_t([\omega])
  :=
  \left[
    \omega'
  \right]
  \in H_{\mathrm c}^1(Y^t(U),\mathbb C).
\]
The resulting class is independent of the cusp neighborhoods and
cutoff functions. $\iota_t$ is compatible with the Hecke action; 
for details, see \cite[\S3]{Hida:1993} and \cite[\S2]{Hida:1999}.
\end{proof}

\subsection{Arithmetic orbifolds at level $U_0$ and $U_1$}\label{subsec:u0:u1}
Let us set
\[
U_1(\mathfrak{N})
:=
\left\{\left(\begin{matrix} a & b \\ c & d \end{matrix}\right)
\in 
U_0(\mathfrak{N})
\;:\;
d\in 1+\widehat{\mathfrak{N}}\right\}.
\]
This is an open compact subgroup of $\mathrm{GL}_2(\mathbb A_F^{(\infty)})$.
For both $U=U_0(\mathfrak N)$ and $U=U_1(\mathfrak N)$, we have
\[
\det(U)=\widehat{\mathcal O}_F^\times.
\]
Therefore, in the notation of \S\ref{subsec:arith:orbi},
\[
\mathrm{Cl}_F(U_0(\mathfrak N))
=\mathrm{Cl}_F(U_1(\mathfrak N))
=\mathrm{Cl}_F.
\]
Choose a complete set $\{a_i\}_{i=1}^{h_F}$ of representatives of $\mathrm{Cl}_F$.
Set $t_i:=\left(\begin{smallmatrix} a_i & 0 \\ 0 & 1 \end{smallmatrix}\right)$.
We then take
\[
\mathcal C_F(U_0(\mathfrak N))
=\mathcal C_F(U_1(\mathfrak N))
:=\{t_1,\ldots,t_{h_F}\}.
\]
Let $\Gamma^t(U)$ be the group defined in \S\ref{subsec:arith:orbi}. Put
\[
\Gamma_0^i(\mathfrak N):=\Gamma^{t_i}(U_0(\mathfrak{N})),
\qquad
\Gamma_1^i(\mathfrak N):=\Gamma^{t_i}(U_1(\mathfrak{N})).
\]
For any subgroup $\Gamma\subset\mathrm{GL}_2(F)$, put
\[
\overline{\Gamma}:=\Gamma\cap\mathrm{SL}_2(F).
\]
Since $\left(\begin{smallmatrix} \varepsilon & 0 \\ 0 & 1 \end{smallmatrix}\right)
\in\Gamma_1^i(\mathfrak N)$, the determinant map gives an exact sequence
\begin{equation}\label{eq:gamma1-bargamma1-unit-sequence}
\begin{tikzcd}
1 \arrow{r} &
\overline{\Gamma_1^i(\mathfrak N)} \arrow{r} &
\Gamma_1^i(\mathfrak N) \arrow{r}{\det} &
E_F \arrow{r} & 1.
\end{tikzcd}
\end{equation}
Let $Y_\Gamma$ be the orbifold defined in~\ref{subsec:arith:orbi}. Set
\[
Y_0^i(\mathfrak N):=Y_{\Gamma_0^i(\mathfrak N)},
\qquad
Y_1^i(\mathfrak N):=Y_{\Gamma_1^i(\mathfrak N)},
\qquad
Y^i(\mathfrak N):=Y_{\overline{\Gamma_1^i(\mathfrak N)}}.
\]
The map $E_F\rightarrow\Gamma_1^i(\mathfrak N)$, 
$\varepsilon\mapsto\left(\begin{smallmatrix}\varepsilon&0\\0&1\end{smallmatrix}\right)$
is a section of the determinant map in
\eqref{eq:gamma1-bargamma1-unit-sequence}. Via this splitting, $E_F$
acts on $Y^i(\mathfrak N)$ by
\[
\varepsilon\cdot[z]
:=
\left[
\left(\begin{smallmatrix}\varepsilon&0\\0&1\end{smallmatrix}\right)z
\right].
\]
Thus $Y^i(\mathfrak N)\to Y_1^i(\mathfrak N)$ is an orbifold covering
with orbifold deck group $E_F$. Put
\[
E_F(\mathfrak N):=\{\varepsilon\in E_F:
\varepsilon\equiv1\pmod{\mathfrak N}\}.
\]
For $\varepsilon\in E_F(\mathfrak N)$, 
$\left(\begin{smallmatrix}\varepsilon&0\\0&\varepsilon^{-1}\end{smallmatrix}\right)$
is an element of $\overline{\Gamma_1^i(\mathfrak N)}$. Since 
$\left(\begin{smallmatrix}\varepsilon&0\\0&1\end{smallmatrix}\right)^2
=\varepsilon\left(\begin{smallmatrix}\varepsilon&0\\0&\varepsilon^{-1}\end{smallmatrix}\right)$
and $\mathbb C^\times$ acts trivially on $\mathscr H_3$, the subgroup $E_F(\mathfrak N)^2$ 
acts trivially on $Y^i(\mathfrak N)$. 

Reduction of the lower-right entry gives an exact sequence
\begin{equation}\label{eq:gamma0-gamma1-diamond-sequence}
\begin{tikzcd}
1 \arrow{r} &
\Gamma_1^i(\mathfrak N) \arrow{r} &
\Gamma_0^i(\mathfrak N) \arrow{r} &
(\mathcal O_F/\mathfrak N)^\times \arrow{r} & 1.
\end{tikzcd}
\end{equation}
Consequently, $Y_1^i(\mathfrak N)\to Y_0^i(\mathfrak N)$ is an
orbifold covering with deck group isomorphic to $(\mathcal O_F/\mathfrak N)^\times$.
For an element $d$ in $(\mathcal O_F/\mathfrak N)^\times$, choose a lift
$\gamma_{d,i}\in\Gamma_0^i(\mathfrak N)$.  Since
$\Gamma_1^i(\mathfrak N)$ is normal in $\Gamma_0^i(\mathfrak N)$, the
map
\[
R_{d,i}:Y_1^i(\mathfrak N)\longrightarrow Y_1^i(\mathfrak N),
\qquad
[z]\longmapsto[\gamma_{d,i}z]
\]
is independent of the lift.
The componentwise diamond operator associated with \(d\) is defined by the pullback
$R_{d,i}^*$ on the cohomology.
Under~\eqref{eq:gamma0-gamma1-diamond-sequence}, 
the diagonal matrix 
$\left(\begin{smallmatrix}\varepsilon&0\\0&\varepsilon\end{smallmatrix}\right)$
maps to $\varepsilon$ mod $\mathfrak N$.
Since scalar matrices act trivially on $\mathscr H_3$,
the image of $E_F$ in $(\mathcal O_F/\mathfrak N)^\times$
acts trivially on $Y_1^i(\mathfrak N)$.

\subsection{Generalized Eichler-Shimura isomorphism}\label{subsec:eichler:shimura}
For $\left(\begin{smallmatrix} a & b \\ c & d \end{smallmatrix}\right)
\in\mathrm{SL}_2(\mathbb{C})$ and $(x,y)\in\mathscr{H}_3$,
let us define the automorphic factor by
\[
\mathrm{j}
\left(\left(\begin{matrix} a & b \\ c & d \end{matrix}\right),
(x,y)\right):=
\left(\begin{matrix} cx+d & -cy \\ \overline{c}y & \overline{cx+d} \end{matrix}\right).
\]

Let $f\in S_2(\mathfrak{N},\chi)$ be a cusp form.
For \(1\le i\le h_F\), define a function $f_i:\mathrm{SL}_2(\mathbb{C})\to\mathbf{L}_2$ 
by $f_i(g,\mathbf{x}):=f(t_i g,\mathbf{x})$.
For $0\leq j\leq 2$, let $f_{i,j}$ be the $X^{2-j}Y^{j}$-component of $f_i$.
Define the function 
$f_i^\prime:\mathrm{SL}_2(\mathbb{C})\to\mathbf{L}_2$ by 
\begin{align*}
    f_i^\prime(g,\mathbf{a})
    :=
    \left(f_{i,j}(g)\right)_{j=0}^2
    \times
    \Psi\left(\mathrm{j}(g,(0,1))^{\top}\mathbf{a}\right),
\end{align*}
where $\mathbf{a}=\left(\begin{smallmatrix} A \\ B \end{smallmatrix}\right)$, 
$A$, $B$ are indeterminates, 
$\Psi(\mathbf{a})=(A^2,AB,B^2)$ 
is a vector of monomials, and
$\times$ is the standard inner product of vectors.
Let us define a differential $1$-form 
$\delta_i(f)$ on $\mathscr{H}_3$ by 
\[
\delta_i(f)\left((x,y)\right)
:=
f_i^\prime
\left(\left(\begin{matrix} \sqrt{y} & x/\sqrt{y} \\ 0 & 1/\sqrt{y} \end{matrix}\right),
\mathbf{a}\right)\bigg|_{(A^2,AB,B^2)=(-dx,dy,d\overline{x})},
\]
where $x,y,\overline{x}$ are the local coordinates of $\mathscr{H}_3$.

Let $\rho$ be a finite-order character of $\widehat{\mathcal O}_F^\times$
that is trivial on $1+\widehat{\mathfrak N}$ and on the image
of $E_F$. 
Let $R$ be a commutative ring containing the values of $\rho$.
For $\gamma\in\Gamma_0^i(\mathfrak N)$, 
there exists an element $u_\gamma\in U_0(\mathfrak N)$ such that
the image of $\gamma$ in $\mathrm{GL}_2(\mathbb A_F^{(\infty)})$
is given by $t_iu_\gamma t_i^{-1}$. Let us set $\rho_i(\gamma):=\rho(d)$
for the lower-right entry $d$ of $u_\gamma$.
This $\rho_i$ is a character of $\Gamma_0^i(\mathfrak N)$, which is well-defined by the assumption on $\rho$. 
Let us define an action of $\Gamma_0^i(\mathfrak N)$ on $\mathscr H_3\times R$ as
$\gamma\cdot(z,r):=\left(\gamma z,\rho_i^{-1}(\gamma)r\right)$.
Let $\mathcal L_i(\rho,R)$ be the sheaf of locally constant sections of the projection
\[
  \Gamma_0^i(\mathfrak N)\backslash\left(\mathscr H_3\times R\right)
  \longrightarrow
  Y_0^i(\mathfrak N).
\]
Let $\mathcal L(\rho,R)$ denote the local system on $Y_0(\mathfrak N)$
whose restriction to the $i$-th connected component is
$\mathcal L_i(\rho,R)$.
Denote $\chi_0:=\chi\vert_{\widehat{\mathcal O}_F^\times}$.
If $\chi_0=\rho$, the property of $f$ shows that
$\delta_i(f)$ descends to \(Y_0^i(\mathfrak N)\) and represents a class in
\[
 H_!^1\left(Y_0^i(\mathfrak N),\mathcal L_i(\rho,\mathbb C)\right).
\]

\begin{proposition}[{Hida \cite[Proposition 3.1]{Hida:1994}}]
\label{eichler:shimura:harder}
The direct sum of the maps $\delta_i$ induces a Hecke-equivariant isomorphism
\[
\bigoplus_{\chi_0=\rho}
S_2(\mathfrak N,\chi)
\xrightarrow{\ \simeq\ }
\bigoplus_{i=1}^{h_F}H_!^1\left(Y_0^i(\mathfrak N),\mathcal L_i(\rho,\mathbb C)\right)
\xrightarrow{\ \simeq\ }
H_!^1\left(Y_0(\mathfrak N),\mathcal L(\rho,\mathbb C)\right).
\]
Here, the leftmost direct sum is over the finite-order Hecke characters \(\chi\) of modulus $\mathfrak N$
satisfying \(\chi_0=\rho\).
On the \(\chi\)-summand, the first map sends \(f\) to \((\delta_i(f))_{i=1}^{h_F}\). 
The second map is induced by~\eqref{eq:arith:mfd:disjoint:union}.
\end{proposition}

\begin{remark}
By Proposition~\ref{prop:hecke-character-local-extension}, the assumption on $\rho$ 
allows $\rho$ to be extended to a finite-order Hecke character of modulus $\mathfrak N$.
Hence the indexing set of the direct sum $\bigoplus_{\chi_0=\rho}$ is nonempty.
\end{remark}

We now relate this generalized Eichler--Shimura isomorphism 
to the \(U_1\)-level formulation, component by component.
Under the decomposition~\eqref{eq:arith:mfd:disjoint:union},
the componentwise diamond actions, which are defined in \S\ref{subsec:u0:u1}, 
induce the action on $Y_1(\mathfrak N):=Y(U_1(\mathfrak N))$.
Pullback along the orbifold covering gives the identification
\[
H_!^1\left(Y_0^i(\mathfrak N),
\mathcal L_i(\rho,\mathbb C)\right)
\xrightarrow{\ \simeq\ }
H_!^1(Y_1^i(\mathfrak N),\mathbb C)[\rho_i],
\]
where the right-hand side denotes the \(\rho_i\)-isotypic subspace for the diamond action.
Taking the direct sum over $i$ yields
\begin{equation}\label{eq:localsys:eigensp}
H_!^1(Y_0(\mathfrak N),\mathcal L(\rho,\mathbb C))
\xrightarrow{\ \simeq\ }
H_!^1(Y_1(\mathfrak N),\mathbb C)[\rho].  
\end{equation}

\begin{proposition}
\label{prop:eichler-shimura-u1-componentwise}
The direct sum of the maps $\delta_i$ induces a Hecke-equivariant isomorphism
\[
 \bigoplus_{\chi}S_2(\mathfrak N,\chi)
 \xrightarrow{\ \simeq\ }
 \bigoplus_{i=1}^{h_F}H_!^1(Y_1^i(\mathfrak N),\mathbb C)
 \xrightarrow{\ \simeq\ }
 H_!^1(Y_1(\mathfrak N),\mathbb C),
 \qquad
 f\longmapsto(\delta_i(f))_{i=1}^{h_F},
\]
where the sum is over the finite-order Hecke characters $\chi$ of
modulus $\mathfrak N$.
\end{proposition}

\begin{proof}
For each character $\rho$ of $\widehat{\mathcal O}_F^\times$ trivial on $1+\widehat{\mathfrak N}$ and $E_F$,
Proposition~\ref{eichler:shimura:harder} and~\eqref{eq:localsys:eigensp} 
give the identification between the sum of cusp form spaces and the $\rho$-eigenspace 
$H_!^1(Y_1(\mathfrak N),\mathbb C)[\rho]$. Taking the sum over $\rho$ gives the asserted isomorphism.
\end{proof}

\begin{remark}
As in Subsection~\ref{subsec:atkin-lehner-func-eq}, cusp forms with
distinct central characters are linearly independent. Hence the direct
sum in Proposition~\ref{prop:eichler-shimura-u1-componentwise} may be
identified with the ordinary sum
\[
\bigoplus_{\chi} S_2(\mathfrak N,\chi)
=
\sum_{\chi} S_2(\mathfrak N,\chi)
\]
inside the space of functions
$\mathrm{GL}_2(\mathbb A_F)\to\mathbf L_2$.  
\end{remark}

\subsection{Arithmetic manifolds at level \(U_1^1\)}\label{subsec:SL2:covering}
Recall that $Y^i(\mathfrak N):=Y_{\overline{\Gamma_1^i(\mathfrak N)}}$ and 
$t_i=\left(\begin{smallmatrix} a_i & 0 \\ 0 & 1 \end{smallmatrix}\right)$,
where $\{a_i\}_{i=1}^{h_F}$ is a complete set of representatives of $\mathrm{Cl}_F$.
Assume that the positive generator of $\mathfrak N\cap\mathbb Z$ is
greater than $3$. Then $\overline{\Gamma_1^i(\mathfrak N)}$ is
torsion-free by Urban~\cite[Lemme~2.3.1]{Urban:1995}. Consequently,
$Y^i(\mathfrak N)$ is a manifold.
Set
\[
U_1^1(\mathfrak N):=
\left\{
\begin{pmatrix}a&b\\c&d\end{pmatrix}\in U_1(\mathfrak N):
a-1\in\widehat{\mathfrak N}
\right\}.
\]
The assumption on $\mathfrak N\cap\mathbb Z$ implies that
$E_F(\mathfrak N)=\{1\}$.
Under this assumption, we additionally obtain
$Y^{t_i}(U_1^1(\mathfrak N))=Y^i(\mathfrak N)$.

\begin{proposition}\label{prop:hoch:serre}
Let $A$ be a coefficient ring in which $\#E_F$ is invertible. Then the pullback of the orbifold covering 
$Y^i(\mathfrak N)\longrightarrow Y_1^i(\mathfrak N)$ induces an isomorphism
\[
H^j(Y_1^i(\mathfrak N),A)
\xrightarrow{\;\simeq\;}
H^j(Y^i(\mathfrak N),A)^{E_F}
\]
\end{proposition}

\begin{proof}
The assertion follows from the Hochschild--Serre spectral sequence of the $E_F$-covering 
(see~\eqref{eq:gamma1-bargamma1-unit-sequence}).
Indeed, as $\#E_F$ is invertible in $A$, the $E_F$-invariant functor is exact, 
so the sequence degenerates at $E_2$.
\end{proof}

\section{Integral automorphic symbols and {\it L}-values}
\label{sec:integral-symbols}

Throughout this section, we work at $U_1(\mathfrak N)$, $U_1^1(\mathfrak N)$-level and use
the notation $Y_1^i(\mathfrak N)$ and
$Y^i(\mathfrak N)$ introduced in the preceding section.

\subsection{Integral parabolic duality}\label{subsec:int:par:dual}
Let $A$ be a Dedekind domain, regarded as a trivial $\mathrm{GL}_2(F)$-module.

\begin{lemma}\label{lem:degree-one-bm-cokernel}
Let $r$ be the number of cusps of $Y_1^i(\mathfrak N)$.
Then there is a canonical identification
\[
\operatorname{coker}\left(H_1(Y_1^i(\mathfrak N),A)\longrightarrow H_1^{\rm BM}(Y_1^i(\mathfrak N),A)\right)
\cong
\ker\left(A^r\xrightarrow{\ (a_s)_s\mapsto\sum_s a_s\ }A\right).
\]
In particular, this cokernel is a free $A$-module of rank $r-1$.
\end{lemma}

\begin{proof}
Put $Y=Y_1^i(\mathfrak N)$.
The boundary exact sequence~\eqref{boundary:exact:sequence} gives
\[
H_1(Y,A)\longrightarrow H_1^{\rm BM}(Y,A)
\longrightarrow H_0(\partial\overline Y^{\rm BS},A)
\longrightarrow H_0(Y,A).
\]
Since $Y$ is connected and its boundary has $r$ connected components,
the last map is the summation map $A^r\to A$.
Exactness gives the identification. Since $r\geq 1$, its kernel
is isomorphic to $A^{r-1}$.
\end{proof}

\begin{lemma}
\label{lem:parabolic-homology-base-change}
Let $k$ be a field equipped with an $A$-algebra structure.
Scalar extension induces an isomorphism
\[
H_1^!(Y_1^i(\mathfrak N),A)\otimes_A k
\xrightarrow{\ \sim\ }
H_1^!(Y_1^i(\mathfrak N),k).
\]
\end{lemma}

\begin{proof}
By Lemma~\ref{lem:degree-one-bm-cokernel}
and the connectedness and noncompactness of $Y_1^i(\mathfrak N)$, all the hypotheses of
Lemma~\ref{universal:coefficient:parabolic:dedekind} are satisfied for
$j=1$. The result follows from that lemma.
\end{proof}

In this subsection, for simplicity,
put
\[
H:=H_!^1(Y_1^i(\mathfrak N),A),
\qquad
M:=H_1^!(Y_1^i(\mathfrak N),A),
\]
and write $(-)^\vee:=\operatorname{Hom}_A(-,A)$.  

\begin{lemma}\label{dualmap:isom}
For the parabolic pairing for homology and cohomology with coefficient $A$,
the map 
\[\mathcal{D}_\cap:
H \;\longrightarrow\; M^\vee
\]
is an isomorphism. Furthermore, if $A$ is a field, 
then the parabolic pairing is perfect.
\end{lemma}
\begin{proof}
Put $Y=Y_1^i(\mathfrak N)$. By Lemma~\ref{lem:degree-one-bm-cokernel}
and the connectedness and noncompactness of $Y$, all the hypotheses of
Lemma~\ref{parabolic:duality:pairing:lemma} are satisfied for $j=1$.
That lemma gives the asserted isomorphism. If $A$ is a field, 
then $M$ is finite-dimensional and hence canonically
isomorphic to its double dual. Therefore the parabolic pairing is perfect.
\end{proof}

By Lemma~\ref{dualmap:isom}, 
the parabolic pairing induces an isomorphism
\[
\mathcal D_\cap:H\xrightarrow{\ \sim\ }M^\vee,
\qquad
\eta\longmapsto\left(m\mapsto m\cap\eta\right).
\]
Dualizing it gives an isomorphism
\[
\mathcal D_\cap^\vee:M^{\vee\vee}
\xrightarrow{\ \sim\ }H^\vee,
\qquad
x\longmapsto x\circ \mathcal D_\cap.
\]
On the other hand, the parabolic pairing also induces a map
\[
\mathcal P_\cap:M\longrightarrow H^\vee,
\qquad
m\longmapsto\left(\eta\mapsto m\cap\eta\right).
\]
Let
\[
{(-)}^{\vee\vee}:M\longrightarrow M^{\vee\vee},
\qquad
m\longmapsto\left(L\mapsto L(m)\right)
\]
be the evaluation map.  Directly from the definitions, we have the
factorization
\begin{equation}\label{eq:parabolic-pairing-double-dual-factorization}
\mathcal P_\cap=\mathcal D_\cap^\vee\circ(-)^{\vee\vee}.
\end{equation}
Indeed, for $m\in M$ and $\eta\in H$,
\[
\left(\mathcal D_\cap^\vee(m^{\vee\vee})\right)(\eta)
=m^{\vee\vee}\left(\mathcal D_\cap(\eta)\right)
=\mathcal D_\cap(\eta)(m)
=m\cap\eta
=\mathcal{P}_\cap(m)(\eta).
\]

The $A$-module $M$ is finitely generated over the Dedekind domain $A$.
By the structure theorem, we have the following short exact sequence:
\begin{equation*}
\begin{tikzcd}
0 \arrow{r} &
M_{\rm tors} \arrow{r} &
M \arrow{r}{\;(-)^{\vee\vee}\;} &
M^{\vee\vee} \arrow{r} &
0,  
\end{tikzcd}  
\end{equation*}
where $M_{\rm tors}$ is the torsion submodule of $M$.
As $M^{\vee\vee}$ is projective by the structure theorem again, 
the surjection $(-)^{\vee\vee}:M\twoheadrightarrow M^{\vee\vee}$
admits a section.  Fix such a section
\begin{equation}\label{eq:double-dual-section}
s_M:M^{\vee\vee}\longrightarrow M.
\end{equation}
Let $\lambda$ be an element of $H^\vee$. Then, for $\eta\in H$,
\begin{equation}\label{eq:duality:section:pair:well-def}
s_M\left((\mathcal D_\cap^{\vee})^{-1}(\lambda)\right)\cap \eta
=\mathcal{P}_\cap
\left(s_M\left((\mathcal D_\cap^{\vee})^{-1}(\lambda)\right)\right)(\eta)
=\lambda(\eta)
\end{equation}
by equation~\eqref{eq:double-dual-section} and the definition of $s_M$.

\subsection{Annihilators of boundary cohomology groups}
Choose a complete set of representatives $\{a_i\}_i$ of $\mathrm{Cl}_F$.
Put $\Gamma:=\Gamma_1^i(\mathfrak{N})$ and $\mathfrak a_i:=a_i\mathcal O_F$. We have
\[
\Gamma=
\left\{
\begin{pmatrix}a&b\\c&d\end{pmatrix}\in\operatorname{GL}_2(F)
:\begin{array}{c}
a,d\in\mathcal{O}_F,\quad b\in \mathfrak a_i,\quad 
c\in \mathfrak a_i^{-1}\mathfrak{N}, \\
ad-bc\in\mathcal{O}_F^\times,\quad d\equiv1\pmod{\mathfrak N}
\end{array}
\right\}.
\]
Let $q\in 1+\mathfrak N$ be such that $q\mathcal{O}_F$ is a prime ideal.
For a cusp $s\in\Gamma\backslash\mathbb{P}^1(F)$, write $\Gamma_s$ for its
stabilizer in $\Gamma$.  Choose a representative $[x_s:z_s]$ of $s$ and set
$\mathfrak{a}_s:=x_s\mathcal{O}_F+z_s\mathfrak a_i$.  Since
$x_s\mathfrak{a}_s^{-1}+z_s\mathfrak a_i\mathfrak{a}_s^{-1}
=\mathcal{O}_F$, we may choose
$y_s\in \mathfrak a_i\mathfrak{a}_s^{-1}$ and
$w_s\in\mathfrak{a}_s^{-1}$ such that $x_sw_s-y_sz_s=1$.  Put
\[
g_s:=\begin{pmatrix}x_s&y_s\\z_s&w_s\end{pmatrix}\in\operatorname{GL}_2(F).
\]

\begin{proposition}\label{double:coset:operator:representatives}
Every double coset in the quotient space
\[
\Gamma\backslash
\Gamma\begin{pmatrix}1&0\\0&q\end{pmatrix}\Gamma
/\Gamma_s
\]
has a representative in the set
\[
\left\{
g_s\begin{pmatrix}1&0\\0&q\end{pmatrix}g_s^{-1},\ 
g_s\begin{pmatrix}q&0\\0&1\end{pmatrix}g_s^{-1}
\right\}.
\]
Both elements fix \(s\).
\end{proposition}

\begin{proof}
For this proof, put
\[
D:=\begin{pmatrix}1&0\\0&q\end{pmatrix},
\qquad
D^\vee:=\begin{pmatrix}q&0\\0&1\end{pmatrix},
\qquad
\mathfrak b=\mathfrak a_i.
\]
Every class is represented by $D\gamma$ for some
$\gamma=\left(\begin{smallmatrix}a&b\\c&d\end{smallmatrix}\right)\in\Gamma$.
Write
\[
\gamma g_s=
\begin{pmatrix}a_s&b_s\\ c_s&d_s\end{pmatrix}.
\]
Then
\begin{equation}\label{eq:membership:entry}
a_s\in\mathfrak{a}_s,\qquad
b_s\in\mathfrak{b}\mathfrak{a}_s^{-1},\qquad
c_s\in\mathfrak{b}^{-1}\mathfrak{a}_s,\qquad
d_s\in\mathfrak{a}_s^{-1},  
\end{equation}
and
$a_sd_s-b_sc_s=u$, where
$u:=\det(\gamma)\in\mathcal{O}_F^\times$. It follows that the integral ideals
$a_s\mathfrak{a}_s^{-1}$ and $c_s\mathfrak{b}\mathfrak{a}_s^{-1}$ are coprime.
We now distinguish two cases.

Suppose first that $a_s\mathfrak{a}_s^{-1}$ is coprime to
$q\mathcal{O}_F$.
Then
\[
a_s\mathfrak{a}_s^{-1}
+q c_s\mathfrak{b}\mathfrak{a}_s^{-1}=\mathcal{O}_F,
\]
so there exist $e\in\mathfrak{a}_s^{-1}$ and
$f\in\mathfrak{b}\mathfrak{a}_s^{-1}$ such that
$a_s e-qc_s f=1$.  Set
\[
t:=(1-q)b_s e
\in\mathfrak{b}\mathfrak{a}_s^{-2}\mathfrak{N},
\qquad
u_{s,t}:=g_s\begin{pmatrix}1&t\\0&1\end{pmatrix}g_s^{-1}.
\]
Explicitly,
\[
u_{s,t}=
\begin{pmatrix}
1-tx_sz_s&tx_s^2\\
-tz_s^2&1+tx_sz_s
\end{pmatrix}.
\]
Its diagonal entries lie in $\mathcal{O}_F$, its upper-right entry lies in
$\mathfrak{b}$, and its lower-left entry lies in
$\mathfrak{b}^{-1}\mathfrak{N}$.  Since $\det(u_{s,t})=1$ and
$u_{s,t}$ fixes $s$, we have $u_{s,t}\in\Gamma_s$.  Let
\[
H_{s,1}:=g_sDg_s^{-1},
\qquad
\gamma_1:=H_{s,1}u_{s,t}\gamma^{-1}D^{-1}.
\]
Since
\[
\gamma^{-1}
=u^{-1}g_s
\begin{pmatrix}d_s&-b_s\\-c_s&a_s\end{pmatrix},
\]
we first obtain
\[
\gamma_1
=u^{-1}g_s
\begin{pmatrix}
d_s-tc_s&q^{-1}(-b_s+ta_s)\\
-qc_s&a_s
\end{pmatrix}.
\]
Substituting $t=(1-q)b_s e$ and using
$a_s e-qc_s f=1$ gives
\[
\gamma_1=
u^{-1}
\begin{pmatrix}
d+c_s(1-q)(y_s-b_s e x_s)
&b_sc_s f x_s+a_s(y_s-b_s e x_s)\\
-c+c_s(1-q)(w_s-b_s e z_s)
&b_sc_s f z_s+a_s(w_s-b_s e z_s)
\end{pmatrix}.
\]
Here
\[
y_s-b_s e x_s\in\mathfrak{b}\mathfrak{a}_s^{-1},
\qquad
w_s-b_s e z_s\in\mathfrak{a}_s^{-1},
\qquad
f\in\mathfrak{b}\mathfrak{a}_s^{-1}.
\]
Using also the memberships~\eqref{eq:membership:entry},
$x_s\in\mathfrak a_s$ and $z_s\in\mathfrak b^{-1}\mathfrak a_s$,
the displayed formula for $\gamma_1$ and the fact that $u$ is a unit show that
\[
\gamma_1\in
\begin{pmatrix}
\mathcal O_F&\mathfrak b\\
\mathfrak b^{-1}\mathfrak N&\mathcal O_F
\end{pmatrix}.
\]
Moreover, $\det(\gamma_1)=u^{-1}\in\mathcal O_F^\times$.
It remains to check the congruence condition on the lower-right entry.
The upper-left entry is congruent to $\det(\gamma_1)$ modulo $\mathfrak N$.
Consequently, the lower-right entry belongs to $1+\mathfrak N$.
Hence $\gamma_1\in\Gamma$.
The identity
$\gamma_1D\gamma=H_{s,1}u_{s,t}$ therefore yields
\[
\Gamma D\gamma\Gamma_s=\Gamma H_{s,1}\Gamma_s.
\]
Suppose now that $a_s\mathfrak{a}_s^{-1}\subseteq q\mathcal{O}_F$.  Write
$a_s=qa_s'$ with $a_s'\in\mathfrak{a}_s$, and let
\[
H_{s,2}:=g_sD^\vee g_s^{-1},
\qquad
\gamma_2:=H_{s,2}\gamma^{-1}D^{-1}.
\]
A direct calculation gives
\[
\gamma_2=
u^{-1}
\begin{pmatrix}
d+(q-1)d_s x_s&(1-q)a_s'y_s-b\\
-c+(q-1)d_s z_s&a+(1-q)a_s'w_s
\end{pmatrix}\in
\begin{pmatrix}
\mathcal O_F&\mathfrak b\\
\mathfrak b^{-1}\mathfrak N& 1+\mathfrak{N}
\end{pmatrix}
\]
and $\det(\gamma_2)=u^{-1}$.
This shows that $\gamma_2\in\Gamma$.  Hence
\[
\Gamma D\gamma\Gamma_s=\Gamma H_{s,2}\Gamma_s.
\]
The two cases are exhaustive because $q\mathcal{O}_F$ is prime, proving the
claim.
\end{proof}

We can now determine the Hecke action in degree zero.

\begin{lemma}\label{hecke:action:boundary:cohom}
Let $A$ be a commutative ring with unity equipped with the trivial $\operatorname{GL}_2(F)$-action.
If $\mathfrak{a}_i$ is coprime to $q\mathcal O_F$, then
\[
T_{q\mathcal{O}_F}=N(q\mathcal O_F)+1
\]
on $H^0(\partial \overline{Y_1^i(\mathfrak{N})}^{\mathrm{BS}},A)$.
\end{lemma}

\begin{proof}
Put $\mathfrak{b}=\mathfrak a_i$ and
\(
D=\left(\begin{smallmatrix}1&0\\0&q\end{smallmatrix}\right).
\)
Under the canonical decomposition
\[
H^0(\partial \overline{Y_1^i(\mathfrak{N})}^{\mathrm{BS}},A)
\cong
\bigoplus_{s\,:\,\textrm{cusps}}H^0(\Gamma_s,A),
\]
write $c=(c_s)_s$.  Fix a cusp $s$, and choose decompositions
\[
\Gamma D\Gamma
=\bigsqcup_i\Gamma\beta_i\Gamma_s,
\qquad
\Gamma\beta_i\Gamma_s
=\bigsqcup_{j=1}^{m_i}\Gamma\beta_i\pi_{ij},
\quad \pi_{ij}\in\Gamma_s.
\]
Hida's boundary formula \cite[(3.1a)--(3.1c)]{Hida:1993} gives
\[
(T_{q\mathcal O_F}c)_s
=\sum_i
c_{\beta_i(s)}\bigm|
[\Gamma_{\beta_i(s)}\beta_i\Gamma_s].
\]
By Proposition~\ref{double:coset:operator:representatives}, the $\beta_i$ may
be chosen so that $\beta_i(s)=s$.  Hida's result
\cite[Lemma~3.3]{Hida:1993} identifies 
\(\Gamma_{\beta_i(s)}\beta_i\Gamma_s
=\bigsqcup_{j=1}^{m_i}
\Gamma_{\beta_i(s)}\beta_i\pi_{ij}.\)
Since $A$ has
trivial action, each of these cosets acts as the identity on $H^0(\Gamma_s,A)$.
Therefore
\[
(T_{q\mathcal O_F}c)_s
=\sum_i m_i c_s
=\left|\Gamma\backslash\Gamma D\Gamma\right|c_s.
\]
It remains to compute the degree of the Hecke correspondence. 
Since $q\mathcal O_F$ is coprime to $\mathfrak b$,
\[
\Gamma D\Gamma
=\Gamma\begin{pmatrix}q&0\\0&1\end{pmatrix}
\sqcup
\bigsqcup_{\beta\in\mathfrak b/q\mathfrak b}
\Gamma\begin{pmatrix}1&\beta\\0&q\end{pmatrix}.
\]
Hence
\(
\left|\Gamma\backslash\Gamma D\Gamma\right|
=1+\left|\mathfrak b/q\mathfrak b\right|
=N(q\mathcal O_F)+1.
\)
Since \(s\) was arbitrary, 
\(T_{q\mathcal O_F}\) acts as multiplication by \(N(q\mathcal O_F)+1\).
\end{proof}

Let $A$ be a commutative ring with unity, equipped with the trivial \(\mathrm{GL}_2(F)\)-action.
Let us define an operator acting on both ordinary and compactly supported cohomology by
\[
\mathcal E_q:=T_{q\mathcal O_F}-N(q\mathcal O_F)-1.
\]
Assume that $\mathfrak a_i$ is coprime to $q\mathcal O_F$.
By Lemma~\ref{hecke:action:boundary:cohom}, 
$\mathcal E_q$ annihilates the boundary cohomology of coefficients in $A$.

\begin{definition}\label{defn:compact:class:lifting}

For a class $\eta$ in $H^1_!(Y_1^i(\mathfrak N),A)$, 
let $\eta'\in H_{\mathrm c}^1(Y_1^i(\mathfrak N),A)$ be a lift of $\eta$, and set
\[
\eta_{\mathrm c}:=\mathcal E_q\eta'\in H_{\mathrm c}^1(Y_1^i(\mathfrak N),A).
\]
\end{definition}

\begin{proposition}\label{prop:canonical-compact-lift}
The map
\[
(-)_{\operatorname{c}}:H_!^1(Y_1^i(\mathfrak N),A)
\longrightarrow
H_{\mathrm c}^1(Y_1^i(\mathfrak N),A),
\qquad
\eta\longmapsto\eta_{\mathrm c},
\]
is a well-defined \(A\)-linear map. Furthermore,
\[
(H^1_{\operatorname{c}}\to H^1_!)\circ (-)_{\operatorname{c}}=\mathcal E_q.
\]
\end{proposition}

\begin{proof}
Put $Y=Y_1^i(\mathfrak N)$.
If \(\eta'\) and \(\eta''\) map to \(\eta\), then
\(\eta'-\eta''\) lies in the image of the Hecke-equivariant map
\[
H^0(\partial\overline{Y}^{\mathrm{BS}},A)
\longrightarrow H_{\mathrm c}^1(Y,A).
\]
Lemma~\ref{hecke:action:boundary:cohom} shows that
\(\mathcal E_q\) annihilates this image. Hence
\[
\mathcal E_q\eta'=\mathcal E_q\eta''.
\]
The final assertion follows from the Hecke equivariance of
\(H_{\mathrm c}^1\to H_!^1\).
\end{proof}

Note that if $A=\mathbb{C}$, then Proposition~\ref{prop:canonical-compact-lift} implies that
\begin{equation}\label{eq:hecke:section:compact:lifting}
\eta_{\operatorname{c}}=\iota_i(\mathcal{E}_q\eta)
\end{equation}
in the cohomology of coefficients in $\mathbb C$ as $\iota_i$ is Hecke equivariant 
(see~\S\ref{subsec:hecke:equiv:section}).

\subsection{Integral automorphic symbols}\label{subsection:integral-automorphic-symbols}
Let $1\leq i\leq h_F$ be an integer.
Recall that $E:=\mathbb C_1^\times\backslash\mathbb C^\times$
and define the proper map
\[
\Delta_i:E\longrightarrow Y^i(\mathfrak N),
\qquad y\longmapsto(0,|y|).
\]
We use the same notation for the corresponding map at any level when no confusion arises.

Let $\mathfrak c$ be an integral ideal of $F$
such that \(\operatorname{ord}_v(\mathfrak c)\ge\operatorname{ord}_v(\mathfrak N)\) 
for any \(v\mid\mathfrak c\), which is the condition~\eqref{eq:assumption:conductor:exponent:large}.
By changing $a_i$ if necessary, we may assume that $a_i$ is prime to $\mathfrak c$.
Choose an element $u$ of $\mathcal{R}_\mathfrak{c}$.
Recall that $t(u)=\left(\begin{smallmatrix} 1 & u \\ 0 & 1 \end{smallmatrix}\right)$.
Abusing the notation, let us denote by $t(u)$ the right translation map on
$\mathrm{GL}_2(\mathbb{A}_F)$ defined by $g\mapsto gt(u)$.
Put $\mathfrak M=\mathfrak N\cap\mathfrak c^2$.
By the strong approximation theorem (see Hida \cite[p.490]{Hida:1994}), 
we can find an element
\[
\alpha_{u,i}\in\mathrm{SL}_2(F)\cap 
t_it(u)U_1^1(\mathfrak M)t_i^{-1}.
\]
Assume that the positive generator of $\mathfrak N\cap\mathbb Z$ is greater than $3$. 
Then $Y^i(\mathfrak N)$ coincides with $Y^{t_i}(U_1^1(\mathfrak N))$ and is a manifold.
Abusing notation, let us write the resulting level-changing map
\[
\alpha_{u,i}:Y^i(\mathfrak M)
\longrightarrow Y_1^i(\mathfrak N)
\] 
by $g\mapsto \alpha_{u,i}g$. It is a finite covering, hence proper.
This map is well-defined because the normalization
$t(u)U_1^1(\mathfrak M)t(-u)\subset U_1(\mathfrak N)$
gives 
\begin{equation}\label{eq:semi:normalization}
\alpha_{u,i}\overline{\Gamma_1^i(\mathfrak M)}
\alpha_{u,i}^{-1}\subset\overline{\Gamma_1^i(\mathfrak N)}
\subset\Gamma_1^i(\mathfrak N).
\end{equation}

Let $q$ be a prime element in $1+\mathfrak{N}$. 
Note that the density of principal prime ideals generated by such $q$ is positive 
due to Chebotarev's density theorem. By using this prime, we construct the lifting 
$(-)_{\mathrm c}$ in Proposition~\ref{prop:canonical-compact-lift}.

From now on, whenever necessary, we regard a cohomology class of level
$U_1(\mathfrak N)$ as a class of level $U_1^1(\mathfrak N)$ via pullback
along the finite covering $Y^i(\mathfrak N)\to Y_1^i(\mathfrak N)$.
Let $A$ be a Dedekind domain.
The linear map 
\begin{equation}\label{eq:symbol:defining:functional}
H_!^1(Y_1^i(\mathfrak N),A)\longrightarrow A,\qquad
\eta\longmapsto \mathrm{Tr}_E\Delta_i^*\alpha_{u,i}^*
\eta_{\operatorname{c}}
\end{equation}
is well-defined. Here, the embedding $\Delta_i$ in the above equation
is defined for level $\mathfrak M$.
Retain the notations $\mathcal D_\cap$ and $s_M$ introduced in~\S\ref{subsec:int:par:dual}.
\begin{definition}[Automorphic symbol]\label{def:auto:sym}
Let $\lambda_{u,i}$
be the well-defined functional~\eqref{eq:symbol:defining:functional}.
Define the \emph{automorphic symbol} associated with $u\in\mathcal R_{\mathfrak c}$ by
\[
\Upsilon_i(u)
:=s_M\!\left((\mathcal D_\cap^\vee)^{-1}
(\lambda_{u,i})\right)
\in H_1^!(Y_1^i(\mathfrak N),A).
\]
\end{definition}
By equation~\eqref{eq:duality:section:pair:well-def},
we have
\[
\Upsilon_i(u)\cap\eta=\mathrm{Tr}_E\Delta_i^*\alpha_{u,i}^*
\eta_{\operatorname{c}}
\]
Thus the pairing defined by \(\Upsilon_i(u)\) is independent of the choice of \(s_M\).
From now on, we fix such section $s_M$.

\begin{remark}\label{remark:automorphic-symbol-choice-independency}
We can check that the automorphic symbol does not depend
on the choice of $\alpha_{u,i}$: Let us choose another element
$\beta_{u,i}$. 
By equation~\eqref{eq:semi:normalization}, 
the two induced maps differ by the action of an element of
$\Gamma_1^i(\mathfrak N)$.  This action is trivial on
$Y_1^i(\mathfrak N)$, and hence the functional
\eqref{eq:symbol:defining:functional} is independent of the choice.
\end{remark}

Let us give the properties of the automorphic symbols:

\begin{lemma}
\label{lem:unit-action-automorphic-symbol}
Let $\varepsilon\in E_F$ and $u\in\mathcal R_{\mathfrak c}$.
Let $\Upsilon_i(\varepsilon u)$ be the symbol for the parameter $u_\varepsilon$,
where $u_\varepsilon$ is the element of $\mathcal R_{\mathfrak c}$ 
congruent to $\varepsilon u$ mod $O_{\mathfrak c}$.
Then
\[
  \Upsilon_i(\varepsilon u)=\Upsilon_i(u)
  \quad\text{in}\quad H_1^!(Y_1^i(\mathfrak N),A).
\]
\end{lemma}

\begin{proof}
Put
$d_\varepsilon=
  \left(\begin{smallmatrix}\varepsilon&0\\0&1\end{smallmatrix}\right)$ and
$b=\varepsilon u-u_\varepsilon\in O_{\mathfrak c}$.
The diagonal matrices $d_\varepsilon$ and $t_i$ commute. Moreover,
\[
  d_\varepsilon t(u)=t(u_\varepsilon)
  \left(\begin{smallmatrix}1&b\\0&1\end{smallmatrix}\right)d_\varepsilon
\]
and $d_\varepsilon$ normalizes $U_1^1(\mathfrak M)$.
It follows that
$d_\varepsilon\alpha_{u,i}d_\varepsilon^{-1}\in
t_it(u_\varepsilon)U_1^1(\mathfrak M)t_i^{-1}$
so we may choose
$\alpha_{\varepsilon u,i}
=d_\varepsilon\alpha_{u,i}d_\varepsilon^{-1}$ by 
remark~\ref{remark:automorphic-symbol-choice-independency}.
The definition of $\Delta_i$ gives a commutative
diagram
\[
\begin{tikzcd}
E \arrow{r}{=} \arrow{d}{\Delta_i}
  & E \arrow{d}{\Delta_i} \\
Y^i(\mathfrak M)
\arrow{r}{d_\varepsilon^{\pm 1}} &
Y^i(\mathfrak M) .
\end{tikzcd}
\]
By the commutative diagram, it follows that
\begin{equation}\label{eq:tr:unit:equiv}
\operatorname{Tr}_E\Delta_i^*(d_\varepsilon^*)^{\pm1}
=\operatorname{Tr}_E\Delta_i^*.  
\end{equation}

Now let $\eta\in H_!^1(Y_1^i(\mathfrak N),A)$. As recalled in \S\ref{subsec:u0:u1}, the pullback to
$Y^i(\mathfrak N)$ of a cohomology class on $Y_1^i(\mathfrak N)$ is invariant under
the deck transformation $d_\varepsilon^*$. 
Since the matrices $d_\varepsilon$ and $\left(\begin{smallmatrix}1&0\\0&q\end{smallmatrix}\right)$ commute,
we have \(d_\varepsilon^*\eta_{\mathrm c}=d_\varepsilon^*\mathcal E_q\eta'
=\mathcal E_qd_\varepsilon^*\eta'=\eta_{\mathrm c}.\)
Consequently, by~\eqref{eq:tr:unit:equiv},
\[
\begin{aligned}
  \operatorname{Tr}_E\Delta_i^*\alpha_{\varepsilon u,i}^*
  \eta_{\operatorname c}
  &=
  \operatorname{Tr}_E\Delta_i^*(d_\varepsilon^*)^{-1}
  \alpha_{u,i}^*d_\varepsilon^*\eta_{\operatorname c}\\
  &=
  \operatorname{Tr}_E\Delta_i^*\alpha_{u,i}^*
  \eta_{\operatorname c}.
\end{aligned}
\]
Therefore, both symbols are constructed from the same functional using the fixed section \(s_M\).
So the assertion is true.
\end{proof}

\begin{lemma}\label{lem:parabolic-cohomology-base-change}
Assume that \(A\subset\mathbb C\) is a Dedekind domain. 
Then the coefficient-change map induces an isomorphism
\[
H_!^1(Y_1^i(\mathfrak N),A)\otimes_A\mathbb C
\xrightarrow{\ \sim\ }
H_!^1(Y_1^i(\mathfrak N),\mathbb C).
\]
\end{lemma}

\begin{proof}
Put $Y=Y_1^i(\mathfrak N)$, $H=H_!^1(Y,A)$, $M=H_1^!(Y,A)$, 
$H_{\mathbb C}=H_!^1(Y,\mathbb C)$ and $M_{\mathbb C}=H_1^!(Y,\mathbb C)$.
Set
\[
M^\vee=\operatorname{Hom}_A(M,A),\qquad
M_{\mathbb C}^\vee
=\operatorname{Hom}_{\mathbb C}(M_{\mathbb C},\mathbb C).
\]
The coefficient-change map $H\longrightarrow H_{\mathbb C}$
induces the $\mathbb C$-linear map
\[
\beta_H:H\otimes_A\mathbb C\longrightarrow H_{\mathbb C}.
\]
By
Lemma~\ref{lem:parabolic-homology-base-change}, scalar extension
induces an isomorphism
\[
\beta_M:M\otimes_A\mathbb C
\xrightarrow{\;\simeq\;}M_{\mathbb C}.
\]
Write $T=M_{\mathrm{tors}}$ and $P=M/T$.
Since $M$ is finitely generated over the Dedekind domain $A$,
the module $P$ is finitely generated projective. Moreover,
\[
M^\vee=P^\vee,
\qquad
M\otimes_A\mathbb C=P\otimes_A\mathbb C.
\]
It follows that the natural map
$\Theta:M^\vee\otimes_A\mathbb C\longrightarrow M_{\mathbb C}^\vee$
defined by
\[
\Theta(\lambda\otimes z)
\left(\beta_M(m\otimes w)\right)
=zw\lambda(m)
\]
is an isomorphism.
We claim that the following diagram commutes:
\[
\begin{tikzcd}
H\otimes_A\mathbb C
 \arrow{r}{\beta_H}
 \arrow{d}{\mathcal D_{\cap}\otimes 1}
&
H_{\mathbb C}
 \arrow{d}{\mathcal D_{\cap,\mathbb C}}
\\
M^\vee\otimes_A\mathbb C
 \arrow{r}{\Theta}
&
M_{\mathbb C}^\vee.
\end{tikzcd}
\]
Indeed, for $\eta\in H$, $m\in M$, and $z,w\in\mathbb C$,
the naturality of the parabolic pairing gives
\[
\begin{aligned}
\mathcal D_{\cap,\mathbb C}
\left(\beta_H(\eta\otimes z)\right)
\left(\beta_M(m\otimes w)\right)
&=
\beta_M(m\otimes w)\cap\beta_H(\eta\otimes z)\\
&=zw(m\cap\eta)\\
&=
\Theta\left(
\mathcal D_{\cap}(\eta)\otimes z
\right)
\left(\beta_M(m\otimes w)\right).
\end{aligned}
\]
This proves the claim. By Lemma~\ref{dualmap:isom}, both vertical
maps are isomorphisms. Since $\Theta$ is also an isomorphism, the
commutative diagram shows that $\beta_H$ is an isomorphism.
\end{proof}

Assume further that $A\subset\mathbb C$.  Abusing the notation, denote 
$\Upsilon_i(u)\in H_1^!(Y_1^i(\mathfrak N),\mathbb C)$ 
by the image of $\Upsilon_i(u)$ under the scalar extension of parabolic homology.

\begin{proposition}
\label{prop:additive-lvalue-automorphic-symbol}
For every
$f\in S_2(\mathfrak N,\chi)$ and $u\in\mathcal R_{\mathfrak c}$, we have
\[
\Upsilon_i(u)\cap\delta_i(f)
=\Gamma_F(1)L_{a_i}(1,\mathcal E_qf;u).
\]
\end{proposition}

\begin{proof}
By the definition of the automorphic symbol,
the Hecke equivariance of the generalized Eichler--Shimura map 
(Proposition~\ref{prop:eichler-shimura-u1-componentwise}),
the naturality of the parabolic pairing,
Lemma~\ref{lem:parabolic-cohomology-base-change},
and~\eqref{eq:hecke:section:compact:lifting},
we obtain
\[
\Upsilon_i(u)\cap\delta_i(f)
=\operatorname{Tr}_E\Delta_i^*\alpha_{u,i}^*
\iota_i\left(\delta_i(\mathcal E_qf)\right).
\]
Choose a differential form $\omega$ representing
$\delta_i(\mathcal E_qf)$. By the construction of section $\iota_i$ (see \S\ref{subsec:hecke:equiv:section}),
there is a function $\theta$, vanishing at the cusps,
such that
\[
\iota_i\left(\delta_i(\mathcal E_qf)\right)=[\omega-d\theta].
\]
From the defining property of $\alpha_{u,i}$ and the right invariance at
level $U_1^1(\mathfrak M)$, we have
\[
\begin{aligned}
\operatorname{Tr}_E\Delta_i^*\alpha_{u,i}^*
\iota_i\left(\delta_i(\mathcal E_qf)\right)
&=
\int_E\Delta_i^*
\left(\alpha_{u,i}^*\omega-\alpha_{u,i}^*d\theta\right)=
\int_E\Delta_i^*\alpha_{u,i}^*\omega\\
&=
\int_E
\widetilde{((\mathcal E_qf)\vert_u)_{a_i}}(y)
\frac{dy}{y}.
\end{aligned}
\]
by Stokes' theorem.
Remark~\ref{rem:mellin-transform-extension} allows us to apply 
Proposition~\ref{integral:representation:lfunction} for $s=1$. 
\end{proof}

\subsection{Representation of \(L\)-values by automorphic symbols}

Let $\xi$ be a finite order Hecke character of conductor $\mathfrak{c}$. 
By (6.9) and the proof of Theorem 8.1 in \cite{Hida:1994}, we obtain
\begin{equation}\label{eq:twisting:class:character}
(f\otimes\xi)_i=\frac{\xi(a_i)}{G(\xi)}
\sum_{u}
\xi(\varpi_{\mathfrak c}u)f_i\vert\alpha_{u,i},\quad
\delta_i(f\otimes\xi)=\frac{\xi(a_i)}{G(\xi)}
\sum_{u}
\xi(\varpi_\mathfrak{c}u)\alpha_{u,i}^*\delta_i(f),
\end{equation}
where $u$ runs through $\mathcal R_{\mathfrak c}$.
Consider the $A$-linear twisting map
\begin{align*}
  \mathcal{R}_i(\xi): 
  H^1_!(Y_1^i(\mathfrak{N}),A)&\longrightarrow 
  H^1_!(Y_1^i(\mathfrak M),A[\xi]),\qquad 
  \eta\longmapsto
  \xi(a_i)\sum_{u\in\mathcal R_{\mathfrak c}}
  \xi(\varpi_{\mathfrak c}u)\alpha_{u,i}^*\eta.
\end{align*}
When \(A=\mathbb C\), the definition of \(\mathcal R_i(\xi)\) gives
    \begin{equation}\label{twisting:operators:compare}
        G(\xi)\delta_i(f\otimes\xi)=\mathcal{R}_i(\xi)\delta_i(f).
    \end{equation}

\begin{proposition}\label{integral:lvalue:automorphic:symbol}
Suppose that \(A\subset\mathbb C\). Then, for every $f\in S_{2}(\mathfrak N,\chi)$, we have
\[
\Gamma_F(1)G(\xi)L_{a_i}(1,\mathcal E_qf\otimes\xi)
=
\xi(a_i)\sum_{u\in\mathcal{R}_{\mathfrak{c}}}
\xi(\varpi_\mathfrak{c}u)\Upsilon_i(u)\cap\delta_i(f).
\]
\end{proposition}
\begin{proof}
By Proposition~\ref{prop:additive-lvalue-automorphic-symbol}, we obtain
\[
\begin{aligned}
\xi(a_i)\sum_u
\xi(\varpi_{\mathfrak c}u)
\Upsilon_i(u)\cap\delta_i(f)
&=
\Gamma_F(1)\xi(a_i)
\sum_u
\xi(\varpi_{\mathfrak c}u)
L_{a_i}(1,\mathcal E_qf;u).
\end{aligned}
\]
Then the assertion is true by Proposition~\ref{integral:representation:lfunction} 
and~\eqref{eq:twisting:class:character}.
\end{proof}

\section{Generation of parabolic homology by automorphic symbols}\label{sec:full-rank}

Let $\mathfrak N$ be an integral ideal.
Choose a boundary-killing prime \(q\in 1+\mathfrak N\).
Let $A$ be a Dedekind domain. Then the section $s_M$ exists (see~\eqref{eq:double-dual-section}).
By using $q$ and $s_M$, we can define the automorphic symbols of coefficients in $A$.
Assume that the positive generator of $\mathfrak N\cap\mathbb Z$ is
greater than $3$ so that $Y^i(\mathfrak N)$ is a manifold and 
the Hecke-equivariant sections $\iota_i$ are well-defined.
Assume further that $A\subset\mathbb C$. This allows us to extend the
automorphic symbols from $A$ to $\mathbb C$. The section $\iota_i$
then expresses their pairings with complex parabolic cohomology
classes in terms of compactly supported classes. Applied to
the Eichler--Shimura class $\delta_i(f)$, this description yields the
$L$-value formula of
Proposition~\ref{prop:additive-lvalue-automorphic-symbol}.

Let $\mathfrak l$ be a split prime ideal of $F$ and assume that $N(\mathfrak l)>3$
so that $Y^i(\mathfrak l^e)$ is a manifold for $e>0$.

\subsection{Additively twisted symbols}
For $a\in O_{\mathfrak l}^{\times}$, set
\[
\Upsilon_i\left(\frac{a}{\varpi_{\mathfrak l}^n}\right)
:=\Upsilon_i(u_a)\in H_1^!(Y_1^i(\mathfrak N),A),
\]
where $u_a$ is the unique element of $\mathcal R_{\mathfrak l^n}$ such that
$\varpi_{\mathfrak l}^n u_a\equiv a
\pmod{\mathfrak l^n}$.
Thus this symbol depends only on $a$ mod $\mathfrak l^n$.
Fix an integer $n_0>0$. 

\begin{definition}
\label{def:additive-average-symbols}
For an integer $n\geq n_0$ and elements $\alpha\in\mathfrak d_F^{-1}\setminus\{0\}$, 
$c\in O_\mathfrak{l}^\times$, define
\[
\Upsilon_c^n(\alpha):=\frac{\#E_F}{\ell^{n-n_0}}
\sum_{a\in c\Gamma_{n_0}/\Gamma_n}
\mathbf{e}_F\left(
-\frac{\alpha a}{\varpi_{\mathfrak l}^n}
\right)
\Upsilon_1\left(
\frac{a}{\varpi_{\mathfrak l}^n}
\right)\in H_1^!(Y_1^1(\mathfrak N),\mathbb C).
\]
\end{definition}

\begin{proposition}\label{prop:automorphic-symbol-pairing-limit}
Let $0<e\leq 2n_0$ be an integer and set $\mathfrak{N}=\mathfrak l^e$.
For \(f\in \bigoplus_{\chi} S_2(\mathfrak l^e,\chi)\),
\[
\lim_{n\to\infty}
\Upsilon_c^n(\alpha)\cap\delta_1(f)
=
\Gamma_F(1)\cdot\frac{a_{\mathcal E_qf}(\alpha \mathfrak d_F)}{N(\alpha \mathfrak d_F)}.
\]
\end{proposition}

\begin{proof}
Since the maps $f\mapsto a_{\mathcal E_qf}$ and $\delta_1$ are linear,
we may assume that $f\in S_2(\mathfrak l^e,\chi)$.
By Proposition~\ref{prop:additive-lvalue-automorphic-symbol}, we obtain
\[
\Upsilon_c^n(\alpha)\cap\delta_1(f)
=\# E_F \Gamma_F(1) L_n^{\mathrm{av}}(f;c,\alpha).
\]
Then Proposition~\ref{prop:additive-average-vertical-limit} gives the asserted limit.
\end{proof}

\subsection{Full-rankness}\label{subsec:full:rank}

Recall that $\Gamma_n:=1+\mathfrak{l}^n O_{\mathfrak l}$.
\begin{definition}
For an integral level $\mathfrak N$ and
$c\in O_{\mathfrak l}^{\times}$, define the $A$-submodule
\[
M_n^i(c):=
\operatorname{Span}_A
\left\{
  \Upsilon_i\left(\frac{a}{\varpi_\mathfrak{l}^n}\right):
  a\in c\Gamma_{n_0}
\right\}
\subset H_1^!(Y_1^i(\mathfrak N),A).
\]
\end{definition}

\begin{lemma}\label{lem:hecke-boundary-killing-operator-invertible}
$\mathcal E_q$ is invertible on the complex parabolic cohomology 
$H_!^1(Y_1(\mathfrak N),\mathbb{C})$.
\end{lemma}
\begin{proof}
By the Hecke equivariance of the generalized Eichler--Shimura
isomorphism, it suffices to prove that $\mathcal E_q$ is invertible
on $S_2(\mathfrak N,\chi)$ for every $\chi$.

By~\cite{JacquetLanglands:1970}, 
the space $S_2(\mathfrak N,\chi)$ decomposes into its irreducible subspaces of cuspidal
automorphic forms. Since $\mathfrak q=q\mathcal O_F$ is prime
to $\mathfrak N$, 
$T_{\mathfrak q}$ acts on each summand by the multiplication by the Hecke eigenvalue $\lambda_f(\mathfrak q)$.
Put $Q=N(\mathfrak q)$. The bound of Blomer--Brumley~\cite{BlomerBrumley:2011} gives
\[
\left|\lambda_\pi(\mathfrak q)\right|
\leq 2Q^{1/2+\theta},
\qquad
\theta=\frac{7}{64}.
\]
Since $Q>1$ and $\theta<1/2$, we have
$\lambda_f(\mathfrak q)\neq Q+1$.
Thus $\mathcal E_q=T_{\mathfrak q}-Q-1$ acts on every automorphic
summand by a nonzero scalar. It is consequently invertible.
\end{proof}

By changing $\{a_i\}_i$, we may assume that $a_1=1$. 
From now on, we study $M_n^i(c)$ at $i=1$.

\begin{theorem}[Full-rankness]\label{vertical:direction}
Let $0<e\leq 2n_0$ be an integer and set $\mathfrak N=\mathfrak l^e$.
Fix $c\in O_{\mathfrak l}^{\times}$.
For all sufficiently large \(n\),
$M_n^1(c)$ is of full rank in
$H_1^!(Y_1^1(\mathfrak l^e),A)$.
\end{theorem}

\begin{proof}
Assume the contrary.
Then there is a sequence $n_j\to\infty$ for which
\[
M_{n_j}^1(c)\otimes_A\mathbb C
\subsetneq H_1^!(Y_1^1(\mathfrak l^e),A)\otimes_A\mathbb C
\xrightarrow{\;\simeq\;}H_1^!(Y_1^1(\mathfrak l^e),\mathbb C),
\]
where the isomorphism is given by 
Lemma~\ref{lem:parabolic-homology-base-change}.
After passage to a subsequence,
\cite[Lemma~6.3]{Kwon:2023} gives a nonzero class
$\omega\in H_!^1(Y_1^1(\mathfrak l^e),\mathbb C)$ such that
\[
m_j\cap\omega\longrightarrow0
\]
for every bounded sequence
$m_j\in M_{n_j}^1(c)\otimes_A\mathbb C$. Write
$\omega=\delta_1(f)$ by
Proposition~\ref{prop:eichler-shimura-u1-componentwise}.
Let $\alpha\in \mathfrak d_F^{-1}\backslash\{0\}$.
By the definition,
\[
\Upsilon_c^{n_j}(\alpha)
\in M_{n_j}^1(c)\otimes_A\mathbb C.
\]
Choose cusp forms $\{g_r\}_r$ such that
$\{\delta_1(g_r)\}_r$ is a basis of $H^1_!(Y_1^1(\mathfrak l^e),\mathbb C)$. 
By the perfect pairing of Lemma~\ref{dualmap:isom} for $A=\mathbb C$, this fixed basis
defines a norm on $H_1^!(Y_1^1(\mathfrak N),\mathbb C)$.
Proposition~\ref{prop:automorphic-symbol-pairing-limit}
shows that, for every $r$, the sequence
$\{\Upsilon_c^{n_j}(\alpha)\cap\delta_1(g_r)\}_j$ converges.
This implies that $\{\Upsilon_c^{n_j}(\alpha)\}_j$ is bounded. 
The defining property of $\omega$ and Proposition~\ref{prop:automorphic-symbol-pairing-limit},
now applied to $f$, give
\begin{equation}\label{eq:pairing:limit:vertical}
  0=
\lim_{j\to\infty}
\Upsilon_c^{n_j}(\alpha)\cap\delta_1(f)
 =\Gamma_F(1)\cdot
\frac{a_{\mathcal E_qf}(\alpha  \mathfrak d_F)}{N(\alpha \mathfrak d_F)}.
\end{equation}
Let us define the principal-class projector:
\[
P_1:S_2(\mathfrak l^e,\chi)\longrightarrow \sum_{\psi\in H_F}S_2(\mathfrak l^e,\chi\psi^2),\qquad
h\longmapsto\frac{1}{h_F}\sum_{\psi\in H_F}h\otimes\psi.
\]
Here, $H_F$ is the group of Hilbert characters.
Extend $P_1$ linearly to $\bigoplus_\chi S_2(\mathfrak l^e,\chi)$.
Since $\alpha$ was arbitrary, $P_1(\mathcal E_qf)=0$ by~\eqref{eq:pairing:limit:vertical}.
Therefore, we obtain 
\[
\delta_1(\mathcal E_qf)=\delta_1(P_1(\mathcal E_qf))=0.
\]
As $\delta$ is Hecke equivariant and 
$\mathcal E_q$ preserves the connected component,
Lemma~\ref{lem:hecke-boundary-killing-operator-invertible} now forces
$\omega=\delta_1(f)=0$, contradicting the choice of $\omega$.
\end{proof}

\subsection{Regularized Bianchi modular symbols}
We define regularized Bianchi modular symbols and compare them
with the automorphic symbols $\Upsilon_1(u)$.
Let $\Gamma$ be a congruence subgroup of $\mathrm{GL}_2(\mathcal O_F)$.
Set
\[
X_\Gamma:=\Gamma\backslash\mathscr H_3^{*},
\qquad
\mathscr H_3^{*}:=\mathscr H_3\cup\mathbb P^1(F).
\]
Cremona~\cite{Cremona:1984} studied integrals along oriented geodesic
paths between cusps.  Via Poincar\'e duality and the Eichler--Shimura
isomorphism, such an integration functional can be considered as an
element of $H_1(X_\Gamma,\mathbb C)$.
We now give a regularized version adapted to the parabolic homology used in this paper.

For $\eta\in H_!^1(Y_1^i(\mathfrak N),A)$, choose a lift
$\eta'\in H_{\mathrm c}^1(Y_1^i(\mathfrak N),A)$ and recall that
$\eta_{\mathrm c}:=\mathcal E_q\eta'$ is independent of the choice of
$\eta'$ by Proposition~\ref{prop:canonical-compact-lift}.
For distinct $r,s\in\mathbb P^1(F)$, let
\[
\Delta_i^{r\to s}:E\longrightarrow Y^i(\mathfrak N)
\]
be a parametrization of the oriented geodesic from $r$ to $s$, where
$E\simeq\mathbb R_{>0}$ is oriented by $dy/y$.
Define an $A$-linear
functional
\begin{equation}\label{eq:regularized-bianchi-functional}
\lambda_{r,s}:H_!^1(Y_1^i(\mathfrak N),A)\to A,\qquad\eta\longmapsto
\operatorname{Tr}_E\left(\Delta_i^{r\to s}\right)^*\eta_{\mathrm c}
\end{equation}
and define the regularized Bianchi modular symbol by
\[
\{r,s\}_i
:=s_M\!\left((\mathcal D_\cap^\vee)^{-1}(\lambda_{r,s})\right)
\in H_1^!(Y_1^i(\mathfrak N),A).
\]
For $r=s$, put $\{r,r\}_i=0$.
The definition is independent of the
parametrization of the oriented geodesic.

For $r,s,t\in\mathbb P^1(F)$ and
$\gamma\in\Gamma_1^i(\mathfrak N)$, the regularized symbols satisfy the following relations:
\begin{equation}\label{eq:regularized-bianchi-relations}
\{r,s\}_i=-\{s,r\}_i,
\qquad
\{r,s\}_i+\{s,t\}_i=\{r,t\}_i,
\qquad
\{\gamma r,\gamma s\}_i=\{r,s\}_i.
\end{equation}

Indeed, after pairing with
$\eta\in H_!^1(Y_1^i(\mathfrak N),A)$, the first identity is orientation reversal.  
For the second identity, choose a compactly supported singular
$1$-cocycle $z$ representing $\eta_{\mathrm c}$.  
Choose cusp neighborhoods disjoint from the support of \(z\).
Replace the portions of the three geodesics contained in these
neighborhoods by paths along the boundaries of neighborhoods.  The resulting closed
chain is the boundary of a singular \(2\)-chain \(C\) in $Y_1^i(\mathfrak N)$.  Since \(z\) vanishes on the
added boundary paths, we have 
\[
 0=z(\partial C)
  =\lambda_{r,s}(\eta)+\lambda_{s,t}(\eta)
   -\lambda_{r,t}(\eta).
\]
This proves the asserted additivity.
The final identity follows because the two paths have the same image in the quotient $Y_1^i(\mathfrak N)$.

As $A\subset\mathbb C$, we have the equality
$\eta_{\mathrm c}=\iota_i(\mathcal E_q\eta)$ under the scalar extension, and hence
$\{r,s\}_i=\mathcal E_q^\vee\{r,s\}_i^{\mathrm{Cr}}$,
where $\mathcal E_q^\vee$ is adjoint to $\mathcal E_q$ under the
complex parabolic pairing and $\{r,s\}_i^{\mathrm{Cr}}$ is Cremona's modular symbol.

For $u\in\mathcal R_{\mathfrak c}$, let $r_u\in F^\times$ be an
element whose image under the diagonal map
\[
 F^\times\longrightarrow\prod_{v\mid\mathfrak c}F_v^\times
\]
is $u$.  
Such an element exists by the definition of $\mathcal R_{\mathfrak c}$.
\begin{theorem}
\label{thm:automorphic-bianchi-comparison}
Let $u\in\mathcal R_{\mathfrak c}$. Then, in
$H_1^!(Y_1^1(\mathfrak N),A)$, one has
\begin{equation*}
\Upsilon_1(u)=\{r_u,\infty\}_1.
\end{equation*}
\end{theorem}

\begin{proof}
Set $\alpha_{u,1}:=\left(\begin{smallmatrix}1&r_u\\0&1\end{smallmatrix}\right)$.
By the choice of \(r_u\),
\[
\alpha_{u,1}\in
\mathrm{SL}_2(F)\cap
t(u)U_1^1(\mathfrak N\cap\mathfrak c^2).
\]
Thus $\alpha_{u,1}$ is an admissible choice in the definition of $\Upsilon_1(u)$.  
By Remark~\ref{remark:automorphic-symbol-choice-independency},
the symbol is independent of the choice of $\alpha_{u,1}$. Note that
\[
\alpha_{u,1}\circ\Delta_1
=\Delta_1^{r_u\to\infty}.
\]Thus \(\alpha_{u,1}\) translates the geodesic $\Delta_1(E)$ by \(r_u\).
Here $E\simeq\mathbb R_{>0}$ is oriented by $dy/y$, so the resulting
path runs from $r_u$ to $\infty$.
Therefore, for $\eta\in H_!^1(Y_1^1(\mathfrak N),A)$, we have
\[
\operatorname{Tr}_E\Delta_1^*
\alpha_{u,1}^*\eta_{\mathrm c}
=
\operatorname{Tr}_E
\bigl(\Delta_1^{r_u\to\infty}\bigr)^*\eta_{\mathrm c},
\]
which is the equality of the defining functionals of 
\(\Upsilon_1(u)\) and \(\{r_u,\infty\}_1\). This proves the assertion.
\end{proof}

\subsection{Full-rankness over finite fields}
Let $\mathfrak N=\mathfrak l^e$ for some integer $0< e\leq 2n_0$.
\begin{theorem}[Vertical persistence]
\label{full:rank:vertical}
There exist positive integers $n^\circ$ and $d^\circ$
such that, for every $c\in O_{\mathfrak l}^{\times}$, the module
$M_{n^\circ}^1(c)$ has full rank in
$H_1^!(Y_1^1(\mathfrak l^e),A)$ and
\[
M_{n^\circ}^1(c)\subset M_n^1(c)\quad\text{for every}\quad n\in n^\circ+d^\circ\mathbb Z_{\geq 0}.
\]
\end{theorem}

\begin{proof}
Let $L_0$ be the order of the group $(O_{\mathfrak l}/\mathfrak l^{n_0}O_{\mathfrak l})^\times$.
Since $M_n^1(c)$ depends only on $c$ mod $\mathfrak l^{n_0}$,
Theorem~\ref{vertical:direction} allows us to choose a sufficiently
large $r^\circ\geq e$, divisible by $L_0$, such that, for
$n^\circ=h_Fr^\circ$, every module $M_{n^\circ}^1(c)$ with
$c\in O_{\mathfrak l}^\times$ has full rank.

Choose a generator $\pi\in \mathcal O_F$ 
of the principal ideal $\mathfrak l^{h_F}$.
For $r\geq0$ and $\beta\in \mathcal O_F$ prime to $\mathfrak l$, let
$u_{\beta,r}\in\mathcal R_{\mathfrak l^{rh_F}}$ be the image of 
$\beta\pi^{-r}\in(\mathfrak l^{-rh_F}/\mathcal O_F)^\times$.
Then Theorem~\ref{thm:automorphic-bianchi-comparison} gives
\begin{equation}\label{eq:global-numerator-comparison}
\Upsilon_1(u_{\beta,r})
=\{\beta\pi^{-r},\infty\}_1.
\end{equation}
For each $c\in O_{\mathfrak l}^\times$, choose a set of representatives $B_c\subset \mathcal O_F$ of
\[
c\Gamma_{n_0}/\Gamma_{n^\circ}.
\]
Put $\mathcal B=\bigcup_{c\in(O_{\mathfrak l}/\mathfrak l^{n_0} O_{\mathfrak l})^\times}B_c$.
Let $D>0$ be a common multiple of $L_0$ and of the orders of $\pi$ in the finite groups
$(\mathcal O_F/\beta\mathcal O_F)^\times$
for all $\beta\in\mathcal B$. Fix $c\in O_{\mathfrak l}^\times$ and
choose $\beta\in B_c$, and $s\in D\mathbb Z_{\geq0}$. Set
\[
b=\beta\pi^{-r^\circ},
\qquad
g=\frac{\pi^{r^\circ}(\pi^s-1)}{\beta}.
\]
The factor $\pi^{r^\circ}$ is divisible by
$\mathfrak l^e$, whereas $\pi^s-1$ is divisible by
$\beta$.  Hence $\gamma_g:=\left(\begin{smallmatrix}1&0\\g&1\end{smallmatrix}\right)$
belongs to $\Gamma_1^1(\mathfrak l^e)$.  The relations
\eqref{eq:regularized-bianchi-relations} give
\begin{align*}
\{b,\infty\}_1
&=\gamma_g\{b,\infty\}_1
=\left\{\frac{b}{bg+1},\infty\right\}_1
  +\{\infty,0\}_1
  +\gamma_g\{0,\infty\}_1 \\
&=\left\{\beta\pi^{-(r^\circ+s)},\infty\right\}_1.
\end{align*}
Applying~\eqref{eq:global-numerator-comparison} to the above equation, we obtain
\[
\Upsilon_1(u_{\beta,r^\circ})
=
\Upsilon_1(u_{\beta,r^\circ+s})\quad\text{in}\quad H_1^!(Y_1^1(\mathfrak l^e),A).
\]
The element $u_{\beta,r}$ represents the same class as $\beta w^r/\varpi_{\mathfrak l}^{rh_F}$, where
$
w=\pi^{-1}
\varpi_{\mathfrak l}^{h_F}\in O_{\mathfrak l}^{\times}.
$
Since \(L_0\mid r^\circ\) and $s$, both \(w^{r^\circ}\) and
\(w^s\) belong to \(\Gamma_{n_0}\).
Consequently, 
\[
M_{n^\circ}^1(c)\subset M_{h_F(r^\circ+s)}^1(c)=M_{n^\circ+h_Fs}^1(c).
\]
Take $d^\circ=h_FD$, then this proves the assertion.
\end{proof}

For the fixed level $\mathfrak N$, put
$
\mathfrak V:=n^\circ+d^\circ\mathbb Z_{\geq0},
$
where the two integers are supplied by
Theorem~\ref{full:rank:vertical}. 
For $c\in O_{\mathfrak l}^\times$, let $\mathcal Q_c$ be the finitely generated torsion $A$-module
\[
\mathcal Q_c
=H_1^!(Y_1^1(\mathfrak N),A)/
 M_{n^\circ}^1(c)
\]
and consider its zeroth Fitting ideal $\operatorname{Fitt}_A^0(\mathcal Q_c)$. 

\begin{theorem}\label{prop:vertical-full-rank-finite-field}
Let $B\subset\mathbb{C}$ be a Dedekind domain containing $A$.
Let $\mathfrak q$ be a prime ideal of $B$ such that $\mathfrak q\cap A$ is prime to
$\operatorname{Fitt}_A^0(\mathcal Q_c)$. Then, for any
$n\in\mathfrak V$, the inclusion induces an isomorphism
\[
M_n^1(c)\otimes_A B/\mathfrak q
\xrightarrow{\ \simeq\ }
H_1^!(Y_1^1(\mathfrak N),B/\mathfrak q).
\]
\end{theorem}

\begin{proof}
Fix $c\in O_{\mathfrak l}^\times$ and put
\[
H_A:=H_1^!(Y_1^1(\mathfrak N),A),
\qquad
M^\circ:=M_{n^\circ}^1(c),
\qquad
\mathfrak{p}:=\mathfrak q\cap A,
\qquad
k=B/\mathfrak q.
\]
The assumption on $\mathcal Q_c$ implies
$(\mathcal Q_c)_{\mathfrak p}=0$.
Localizing the exact sequence
\[
0\longrightarrow M^\circ\longrightarrow H_A
\longrightarrow \mathcal Q_c
\longrightarrow0
\]
at $\mathfrak p$ gives the isomorphism
$
(M^\circ)_{\mathfrak p}\simeq(H_A)_{\mathfrak p}.
$
For every $n\in\mathfrak V$,
Theorem~\ref{full:rank:vertical} gives
\[
M^\circ\subset M_n^1(c)\subset H_A.
\]
After localization at \(\mathfrak p\), the two inclusions become
equalities.  Hence
$M_n^1(c)_{\mathfrak p}=(H_A)_{\mathfrak p}$.
Tensoring this equality over \(A_{\mathfrak p}\) with \(k\) gives
\[
M_n^1(c)\otimes_A k
\xrightarrow{\ \sim\ }
H_A\otimes_A k.
\]
Finally,
Lemma~\ref{lem:parabolic-homology-base-change} identifies the
right-hand side with $H_1^!(Y_1^1(\mathfrak N),k)$. This proves the proposition.
\end{proof}

\section{Non-vanishing mod {\it p} of integral {\it L}-values}
\label{sec:mod-p-nonvanishing}
In this section, $\ell$ is an odd rational prime.

\subsection{Averages of roots of unity and their reductions}\label{subsec:roots:averages}
For an integer $n$ and a number field $M$, write $M_n:=M(\mu_{\ell^n})$.
For integers $n,m>0$ with $n\geq m$ and an element $\zeta\in\mu_{\ell^n}$, 
we have the following formula for cyclotomic Galois averages:
\begin{equation}\label{eq:standard:cyclotomic:average}
  \frac{1}{[\mathbb{Q}_n:\mathbb{Q}_{m}]}
  \operatorname{Tr}_{\mathbb{Q}_n/\mathbb{Q}_{m}}(\zeta)
  =
  \begin{cases}
    \zeta, & \text{if }\zeta\in\mu_{\ell^{m}},\\
    0, & \text{otherwise}.
  \end{cases}
\end{equation}

Let us fix a number field $L$ in this subsection.
Let $p\neq \ell$ be an odd rational prime
and $\mathfrak{P}$ a prime of $\overline{\mathbb{Z}}$ lying above it.
Put $\mathfrak{p}_n:=\mathfrak{P}\cap\mathcal O_{L_n}$ for $n\geq 0$.
Then reduction modulo $\mathfrak{p}_n$ sends $\zeta\in\mu_{\ell^n}$
to $\overline{\zeta}\in\mathcal O_{L_{n}}/\mathfrak p_{n}$.
Set 
\[r_L^{\mathrm{cyc}}:=
\min\left\{r\ge0\;:\;L\cap\mathbb Q(\mu_{\ell^\infty})\subseteq\mathbb Q_r\right\}.
\]
Let $n_0\geq r_L^{\mathrm{cyc}}$ be an integer.
For a prime ideal $\mathfrak p$ of $L$, 
let $r_\ell(\mathfrak p)$ be the $\ell$-adic valuation of $N(\mathfrak p)^d-1$, where
$d$ is the multiplicative order of $N(\mathfrak p)$ in $\mathbb{F}_\ell^\times$.

\begin{lemma}\label{lem:galois-average-reduction}
Assume that $n_0\geq r_\ell(\mathfrak p_0)$.
Then for every $\zeta\in\mu_{\ell^n}$ with $n\geq n_0$,
\[
 \frac{1}{\#D_n}\sum_{\tau\in D_n}\zeta^\tau
 \equiv
 \frac{1}{[\mathbb Q_n:\mathbb Q_{n_0}]}\sum_{\sigma\in\mathrm{Gal}(\mathbb Q_n/\mathbb Q_{n_0})}\zeta^\sigma
 \pmod{\mathfrak p_n},
\]
where $D_n\subset \operatorname{Gal}(L_n/L_{n_0})$ is the decomposition group of $\mathfrak p_n$.

\end{lemma}

\begin{proof}
Let $\zeta\in\mu_{\ell^n}$.
Put $\lambda_n:=\mathcal O_{L_n}/\mathfrak p_n$. The action of $D_n$ on $\lambda_n$ induces a surjection
$
 D_n\twoheadrightarrow\operatorname{Gal}(\lambda_n/\lambda_{n_0}).
$
Consequently, compatibility of reduction with the Galois action gives
\[
 \frac{1}{\#D_n}\sum_{\tau\in D_n}\zeta^\tau
 \equiv
 \frac{1}{[\lambda_n:\lambda_{n_0}]}\operatorname{Tr}_{\lambda_n/\lambda_{n_0}}(\overline\zeta)
 \pmod{\mathfrak p_n}.
\]
Since $p\ne\ell$, reduction preserves the order of every
$\ell$-power root of unity. 
By the Lifting The Exponent lemma and our assumption on $n_0$, it follows that
\[
 \frac{1}{[\lambda_n:\lambda_{n_0}]}\operatorname{Tr}_{\lambda_n/\lambda_{n_0}}(\overline\zeta)
 =
 \begin{cases}
  \overline\zeta,&\zeta\in\mu_{\ell^{n_0}},\\
  0,&\text{otherwise}.
 \end{cases}
\]
By~\eqref{eq:standard:cyclotomic:average}, the assertion is true.
\end{proof}

\subsection{Coates--Wiles characters and their averages}\label{subsec:coates:wiles}
Assume that $\ell$ is unramified in $F$.
Let $\mathfrak{l}$ be a prime ideal of $F$ lying above $\ell$, and
$\Gamma_n=1+\mathfrak{l}^nO_\mathfrak{l}$. We have the following split short exact sequence
\[
  1
  \longrightarrow
  \Gamma_1
  \longrightarrow
  O_{\mathfrak l}^{\times}
  \longrightarrow
  (O_{\mathfrak l}/\mathfrak lO_{\mathfrak l})^\times
  \longrightarrow
  1
\]
induced by reduction modulo $\mathfrak l$.
The image of the section of the mod $\mathfrak l$-reduction map 
is denoted by $W_{\mathfrak l}$ and
is called the \emph{Teichm{\"u}ller group}. 
Thus the reduction restricts to an isomorphism
$W_{\mathfrak l}\xrightarrow{\;\simeq\;}
  (O_{\mathfrak l}/\mathfrak lO_{\mathfrak l})^\times.$
Let us set
\[
\mathrm{Cl}_F(\mathfrak{l}^\infty)
:=
\varprojlim_{m>0} \mathrm{Cl}_F(\mathfrak{l}^m).
\]
Applying the snake lemma to~\eqref{eq:ray-class-local-global-exact-sequence} for
\(\mathfrak c=\mathfrak l^n\) and \(\mathfrak c=\mathfrak l\), 
and then passing to the inverse limit over \(n\), we obtain the exact sequence
\begin{equation}\label{fund:exact:seq}
  \begin{tikzcd}
    1 \arrow{r} 
    & 
    \Gamma_1 \arrow{r} 
    & 
    \mathrm{Cl}_F(\mathfrak{l}^\infty) \arrow{r} 
    & 
    \mathrm{Cl}_F(\mathfrak{l}) \arrow{r} & 1. 
  \end{tikzcd}
\end{equation}
Here, the map from $\Gamma_1$ to $\mathrm{Cl}_F(\mathfrak{l}^\infty)$ is
induced from the injection 
$O_{\mathfrak l}^\times\hookrightarrow\mathbb{A}_F^\times$.
As $\ell$ is unramified in $F$,
we may therefore identify \(E_F\) with a subgroup of \(W_{\mathfrak l}\). 
Consequently, the following sequence is exact.
\begin{equation}\label{fund:exact:seq:2}
  \begin{tikzcd}
    1 \arrow{r} 
    & 
    W_\mathfrak{l}/E_F \arrow{r} 
    & 
    \mathrm{Cl}_F(\mathfrak{l}) \arrow{r} 
    & 
    \mathrm{Cl}_F \arrow{r} & 1. 
  \end{tikzcd}
\end{equation}
Let us assume that $\ell$ splits in $F$.
In this case, the $\mathbb{Z}_\ell$-rank of 
$\mathrm{Cl}_F(\mathfrak{l}^\infty)$ is $1$.

\begin{definition}\label{coates:wiles:extension}
For $n>0$, let us denote by $X_n$ the set of $\ell$-power order 
primitive Hecke characters of modulus $\mathfrak{l}^n$. 
\end{definition}

We impose the additional technical assumption that $\ell\nmid h_F$.
This makes the structure of $\mathrm{Cl}_F(\mathfrak{l}^\infty)$ easy.
We can split the ray class group 
as the product of its pro-$\ell$ part $\Gamma_1$ 
and its prime-to-$\ell$ part $\Delta$:
\[
  \mathrm{Cl}_F(\mathfrak{l}^\infty)=\Gamma_1\times\Delta.
\]
Clearly, $\Delta$ is equal to the split image of $\mathrm{Cl}_F(\mathfrak{l})$
in $\mathrm{Cl}_F(\mathfrak{l}^\infty)$ in the sequence~\eqref{fund:exact:seq}.
A character $\phi\in X_n$ restricts to a character of $\Gamma_1$
via the sequence~\eqref{fund:exact:seq}.

\begin{lemma}\label{lem:prim:char:isom}
Let $n>1$ and $\phi\in X_n$. 
Then the restriction $\phi\vert_{\Gamma_1}$ gives a bijection 
between $X_n$ and the set of isomorphisms
$\Gamma_1/\Gamma_n\cong\mu_{\ell^{n-1}}$.
\end{lemma}

\begin{proof}
The restriction has kernel $\Gamma_n$ and inducesan isomorphism
as $\phi$ is primitive of modulus $\mathfrak l^n$ and is trivial on $\Delta$.
\end{proof}

Let $n_0>0$ be an integer and assume that $n> n_0$.
For $\phi\in X_n$, let us denote by 
$G_\phi=\operatorname{Gal}\left(\mathbb Q(\phi)/\mathbb Q(\mu_{\ell^{n_0}})\right)$.

\begin{lemma}\label{lem:galois-average-residue-unit}
Let $\phi\in X_n$ and $a\in O_{\mathfrak l}^\times$.
Then,
\[
  \frac1{\#G_\phi}\sum_{\sigma\in G_\phi}\phi^\sigma(a)
  =
  \begin{cases}
    \displaystyle
    \phi(a)
    &
    \text{if }a\in W_\mathfrak{l}\Gamma_{n-n_0},
    \\
    0 & \text{otherwise}.
  \end{cases}
\]
\end{lemma}

\begin{proof}
The assertion follows from Lemma~\ref{lem:prim:char:isom} 
and the formula~\eqref{eq:standard:cyclotomic:average}.
\end{proof}

Since $\#\Delta$ is prime to $\ell$, the $\#\Delta$-power map
on $\Gamma_1$ is an automorphism. 
Recall that $\{a_i\}_{i=1}^{h_F}$ is a complete set of representatives for $\mathrm{Cl}_F$
such that $a_1$ and the $\mathfrak l$-components of $a_i$'s are trivial.

\begin{lemma}\label{lem:character-average}
For \(1\le i\le h_F\) and \(a\in O_{\mathfrak l}^{\times}\),
\[
  \frac{1}{\#\Delta}\sum_{\psi\in\widehat{\Delta}}
  \psi(aa_i)
  =
    \begin{cases}
    \displaystyle
    1, &
    \text{if }i=1\text{ and }a\in E_F\Gamma_1,\\
    0, & \text{otherwise},
  \end{cases}
\]
\end{lemma}
\begin{proof}
As $\psi$ is trivial on $\Gamma_1$, the average in the assertion is given by
$1$ if $\kappa a_i=1$ in $\Delta$ and $0$ otherwise,
where $\kappa\in W_{\mathfrak l}$ is the projection of $a$. 
$\kappa a_i=1$ if and only if $\kappa\in E_F$ and $i=1$ by~\eqref{fund:exact:seq:2}.
\end{proof}

Assume further that $n\geq 2n_0$.
Let $\xi$ be any finite-order Hecke character of modulus
$\mathfrak l^n$.  Since $\Gamma_1/\Gamma_n$ is an $\ell$-group, the
restriction of $\xi$ to this group has $\ell$-power order.  Hence
the map
$b\mapsto\xi(1+\varpi_{\mathfrak l}^{n-n_0}b)$ 
is an additive character of
$O_\mathfrak{l}$ whose kernel contains $\mathfrak{l}^{n_0}O_\mathfrak{l}$.  
There is an element $c_\xi\in O_\mathfrak{l}$ such that
\[
  \xi\left(1+\varpi_{\mathfrak l}^{n-n_0}b\right)
  =
  \mathbf{e}_F\left(
    \frac{b c_\xi}{\varpi_{\mathfrak l}^{n_0}}
  \right),
  \qquad b\in O_{\mathfrak l}.
\]
Here we regard $bc_\xi$ as an element of $\mathbb A_F$ 
via the embedding $F_\mathfrak l\to \mathbb A_F$.
The assignment $\xi\mapsto c_\xi$ gives a map
\[
 \widehat{\mathrm{Cl}_F(\mathfrak l^n)}
 \longrightarrow
 O_\mathfrak{l}/\mathfrak{l}^{n_0}O_\mathfrak{l}.
\]

\begin{proposition}\label{prop:average:lpower:l}
For \(1\le i\le h_F\), $\phi\in X_n$ and \(a\in O_{\mathfrak l}^{\times}\),
\[
\frac1{\# G_\phi}\sum_{\sigma\in G_\phi}\frac1{\#\Delta}\sum_{\psi\in\widehat\Delta}\phi^\sigma\psi(aa_i)=
\begin{cases}
\mathbf{e}_F\left(\frac{c_\phi(\kappa_a^{-1}a-1)}{\varpi_{\mathfrak l}^{n}}\right)
& \text{if }i=1\text{ and }a\in E_F\Gamma_{n-n_0}, \\  
0 & \text{otherwise},
\end{cases}
\]
where $\kappa_a$ is the image of $a$ under the projection 
$O_{\mathfrak l}^\times\rightarrow W_{\mathfrak l}$. 
\end{proposition}

\begin{proof}
The assertion follows from the definition of $c_\phi$ 
and Lemmas~\ref{lem:galois-average-residue-unit} and~\ref{lem:character-average}.
\end{proof}

\subsection{Hecke eigenspaces in parabolic cohomology}\label{subsec:parabolic:eigenspace}
A cusp form is called a \emph{Hecke eigenform} 
if it is a simultaneous eigenvector for all Hecke operators.
A Hecke eigenform satisfying $a_f(\mathcal O_F)=1$ is said to be 
{\it normalized}. In this case, the Fourier-Whittaker coefficients $a_f(\mathfrak{a})$ of $f$ 
coincide with the eigenvalues of $T_\mathfrak{a}$.

Let $f\in S_2(\mathfrak{N},\chi)$ be a normalized Hecke eigenform.
By Hida~\cite[p.~488]{Hida:1994}, there exists a number field
\(K\) such that the \((f,\chi_0)\)-eigenspace
\[
H_!^1(Y_1(\mathfrak N),\mathcal O_K)[f,\chi_0]
\]
is free of rank one over \(\mathcal O_K\) and $K$ contains the Hecke field of $f$ over $\mathbb{Q}$.
Let us set 
\[
A:=\mathcal O_K[\#E_F^{-1}].
\]
By flatness of localization, the corresponding eigenspace over \(A\)
is also free of rank one over $A$.
Fix a generator $\eta_f$ of this eigenspace.
We denote its image in
\(H_!^1(Y_1(\mathfrak N),\mathbb C)\) by the same symbol.
Hida's period \(\Omega_f\in\mathbb C^\times\) is characterized by
\[
    \frac{\delta(f)}{\Omega_f}=\eta_f
    \quad\text{in}\quad H^1_!(Y_1(\mathfrak{N}),\mathbb{C}).
\]
Thus, $\Omega_f$ is well-defined up to multiplication 
by an element of $A^\times$.
Under the decomposition~\eqref{eq:arith:mfd:disjoint:union}, let
\[
(-)_i:
H_!^1(Y_1(\mathfrak N),A)
\longrightarrow
H_!^1(Y_1^i(\mathfrak N),A),
\qquad \eta\longmapsto\eta_i,
\]
denote the projection onto the \(i\)-th component.
Choose a boundary killing prime $q\in 1+\mathfrak N$.
Recall from Definition~\ref{defn:compact:class:lifting} that 
for a lifting of $\eta_{f,i}$ under the natural map
\(H_c^1\to H_!^1\), $\eta_{f,i,\mathrm c}$ is the class $\mathcal E_q\eta'$,
which does not depend on the choice of the lifting $\eta'$. 
Let \(\varepsilon_q\in A\) be defined by
$\mathcal E_qf=\varepsilon_qf$.
Note that $\varepsilon_q$ is nonzero by Lemma~\ref{lem:hecke-boundary-killing-operator-invertible}.
\begin{lemma}\label{lemma:integrality:cohomology:class}
Under the coefficient-change map for $Y_1^i(\mathfrak N)$ induced by $A\hookrightarrow\mathbb C$,
the image of \(\eta_{f,i,\mathrm c}\) is
$\varepsilon_q\frac{\iota_i(\delta_i(f))}{\Omega_f}$.
\end{lemma}
 
\begin{proof}
Choose a lift $\eta^\prime$ of $\eta_{f,i}$.
After changing of coefficients to $\mathbb{C}$, naturality of the maps in the
localization sequence gives the commutative diagram
\[
\begin{tikzcd}
H_c^1(Y_1^i(\mathfrak{N}),A) \arrow[r] \arrow[d] & 
H^1(Y_1^i(\mathfrak{N}),A) \arrow[d] \\
H_c^1(Y_1^i(\mathfrak{N}),\mathbb{C}) \arrow[r] & 
H^1(Y_1^i(\mathfrak{N}),\mathbb{C})
\end{tikzcd},\quad
\begin{tikzcd}
\eta^\prime \arrow[r, mapsto] & \eta_{f,i} \arrow[d, mapsto] \\
\frac{\iota_i(\delta_i(f))}{\Omega_f} \arrow[r, mapsto] & \frac{\delta_i(f)}{\Omega_f}
\end{tikzcd}.
\]
Consequently, the difference between the image of \(\eta^\prime\)
under coefficient change and
\(\iota_i(\delta_i(f))/\Omega_f\) lies in the image of
\[
H^0\left(\partial \overline{Y_1^i(\mathfrak N)}^{\mathrm BS},\mathbb C\right)
\longrightarrow
H_c^1(Y_1^i(\mathfrak N),\mathbb C).
\]
Lemma~\ref{hecke:action:boundary:cohom} shows that
\(\mathcal E_q\) annihilates this image. Since all the maps involved
are Hecke-equivariant, we obtain
\[
    \eta_{f,i,\operatorname{c}} = \mathcal{E}_q\eta^\prime
    \longmapsto\mathcal{E}_q\frac{\iota_i(\delta_i(f))}{\Omega_f}
    =
    \varepsilon_q \frac{\iota_i(\delta_i(f))}{\Omega_f}.
\]
This proves the lemma.
\end{proof}

\subsection{Integrality of \emph{L}-values}\label{subsec:integral:Lvalue}

\begin{definition}
Define the $L$-function of $f$ by
\[
L(s,f):=\sum_{\mathfrak{a}}\frac{a_f(\mathfrak{a})}{N(\mathfrak{a})^s},
\qquad\mathrm{Re}(s)>\frac{3}{2}+\theta,
\]
where $\mathfrak{a}$ runs through the non-zero integral ideals of $F$.
\end{definition}
By the definition, $L(s,f)=\sum_{i=1}^{h_F}L_{a_i}(s,f)$, so that
$L(s,f)$ has an analytic continuation to $s=1$ by Proposition~\ref{Lftn:analytic:conti}.
For a Hecke character $\xi$ of modulus $\mathfrak c$, set
\[
L_f(\xi):=\Gamma_F(1)\frac{G(\xi)}{\Omega_f}L(1,f\otimes\xi)=\frac{D_F\# E_F}{16\pi^2}\frac{G(\xi)}{\Omega_f}L(1,f\otimes\xi).
\]
\begin{cor}\label{cor:eigenform-lvalue-automorphic-symbol}
If $f$ is a normalized Hecke eigenform, then
\[
\varepsilon_q L_f(\xi)=\sum_{i=1}^{h_F}
\xi(a_i)\sum_{u\in\mathcal R_{\mathfrak c}}
\xi(\varpi_{\mathfrak c}u)\Upsilon_i(u)\cap\eta_{f,i}.
\]
\end{cor}

\begin{proof}
Apply Proposition~\ref{integral:lvalue:automorphic:symbol} to
$f/\Omega_f$, noting that
$\frac{\delta_i(f)}{\Omega_f}=\eta_{f,i}$ and
$\mathcal E_qf=\varepsilon_qf$, then the assertion is true by the naturality of the pairing.
\end{proof}

The preceding corollary implies that
\begin{equation}\label{eq:integrality:lvalue}
\varepsilon_qL_f(\xi)\in A[\xi].
\end{equation}

\subsection{Fourier operators and pairing-content ideals}
Let $n_0>0$ be an integer and $\mathfrak N$ an integral ideal of $F$.
Denote $\mathfrak N'=\mathfrak N\cap\mathfrak l^{2n_0}$.
For $g\in O_{\mathfrak l}^\times$, define
\[
\mathcal F_{g,i} :=
\sum_{b\ \mathrm{mod}\ \mathfrak l^{n_0}}
\mathbf e_F\left(
\frac{bg}{\varpi_{\mathfrak l}^{n_0}}
\right)
\sum_{\varepsilon\in E_F}
\alpha_{b\varepsilon/\varpi_{\mathfrak l}^{n_0},i}^*.
\]
Here each \(\alpha_{u,i}^*\), with
\(u=b\varepsilon/\varpi_{\mathfrak l}^{n_0}\), is the level-changing
pullback defined in \S\ref{subsection:integral-automorphic-symbols}. 
Assume that $n_0\geq\mathrm{ord}_{\mathfrak l}(\mathfrak N)$. Then the pullbacks are well-defined.
The $\varepsilon\in E_F$-action is given by 
$d_\varepsilon^*$ and $d_\varepsilon^*\alpha_{u,i}^*=\alpha_{\varepsilon^{-1}u,i}^*d_\varepsilon^*$,
where $d_\varepsilon=\left(\begin{smallmatrix}\varepsilon&0\\0&1\end{smallmatrix}\right)$.
Therefore, the sum of $\alpha^*_{\varepsilon u,i}\eta$ in the summand of $\mathcal F_{g,i}\eta$
is \(E_F\)-invariant and hence defines a class on \(Y_1^i(\mathfrak N')\)
by Proposition~\ref{prop:hoch:serre}.
Consequently, the operator
\[
\mathcal F_{g,i}:
H_!^1(Y_1^i(\mathfrak N),A)
\longrightarrow
H_!^1\left(Y_1^i(\mathfrak N'),A[\xi,\mu_{\ell^{n_0}}]\right)
\]
is well-defined.

\begin{lemma}
\label{lem:gauss-inversion-fourier-operator}
Let $\xi$ be a primitive ray-class character of conductor
$\mathfrak l^{n_0}$. Then
\begin{equation}\label{eq:gauss-inversion-full-fourier-operator}
\xi(a_i)\sum_{[g]\in(O_{\mathfrak l}/\mathfrak l^{n_0}O_{\mathfrak l})^\times/E_F}
\overline\xi(g)\mathcal F_{g,i}\eta_{f,i}
=G(\overline\xi)\overline\xi(d_F)\xi(d_{F,\mathfrak l})\mathcal R_i(\xi)\eta_{f,i}.
\end{equation}
\end{lemma}

\begin{proof}
Unfolding the definition of $\mathcal F_{g,i}$,
the left-hand side of~\eqref{eq:gauss-inversion-full-fourier-operator} is given by
\[
\xi(a_i)\sum_{b\ \mathrm{mod}\ \mathfrak l^{n_0}}
\left(
\sum_{g\in(O_{\mathfrak l}/\mathfrak l^{n_0}O_{\mathfrak l})^\times}
\overline\xi(g)
\mathbf e_F\left(
\frac{bg}{\varpi_{\mathfrak l}^{n_0}}
\right)
\right)
\alpha_{b/\varpi_{\mathfrak l}^{n_0},i}^*\eta_{f,i}.
\]
Using inversion formula~\eqref{eq:gauss:inversion}, we are done.
\end{proof}

\subsection{Twisted averaged congruences and non-vanishing}\label{subsec:main:result}
Assume that $\ell\nmid h_F$ to use averaging lemmas in \S\ref{subsec:coates:wiles}.
Let $f$ be a normalized eigenform of $S_2(\mathfrak N,\chi)$ and $K$ a number field such that the eigenspace
$H_!^1(Y_1(\mathfrak N),\mathcal O_K)[f,\chi_0]$ is free of rank one over $\mathcal O_K$.

Set $L=K(\mu_{\#\Delta})$.  Retain the notations $\Delta$, $r_\ell(\mathfrak p)$ and
$r_L^{\rm cyc}$ from \S\ref{subsec:roots:averages} and \S\ref{subsec:coates:wiles}.
For an integer $h>0$, let us set
\[
\mathcal P_h:=\{\mathfrak p\nmid \ell\#E_F:\text{prime of }L\;:\;r_\ell(\mathfrak p)\leq h\}.
\]
By Chebotarev's density theorem, there exists $h_0>0$ such that 
$\mathcal  P_h$ has positive density for any $h>h_0$. 
Fix a cutoff $h>h_0$.
Choose an integer $n_0$ such that
\[
 n_0\geq
 \max\left\{\mathrm{ord}_{\mathfrak l}(\mathfrak N),r_L^{\rm cyc},h\right\}>0.
\]
This bound on $n_0$ permits the use of
Lemma~\ref{lem:galois-average-reduction}.
Put $\mathfrak N':=\mathfrak N\cap\mathfrak l^{2n_0}$.
Assume the condition on $\mathfrak N$ so that both $Y^1(\mathfrak N)$ and $Y^1(\mathfrak N')$ are manifolds,
so that the Hecke equivariant sections $\iota_1$ for both levels are well-defined.

Replace the boundary-killing prime \(q\) by one
satisfying \(q\in1+\mathfrak N'\). 
In what follows, the notations \(M_n^1(c)\) and \(\mathcal Q_c\)
refer to the fixed level \(\mathfrak N'\), rather than to the original
level \(\mathfrak N\).

Put $B:=\mathcal O_{L(\mu_{\ell^{n_0}})}[\#E_F^{-1}]$.
Then, $B$ is an $A$-algebra containing the Hecke eigenvalues of $f$ and the values
of every ray-class character of conductor $\mathfrak l^{n_0}$.
Define the pairing-content ideal
\[
\mathfrak J_{f,\xi}
:=
\left\{
m\cap\mathcal R_1(\xi)\eta_{f,1}:
m\in H_1^!(Y_1^1(\mathfrak N'),B)
\right\}
\subset B.
\]
\begin{lemma}
\label{lem:hecke-transport-twisted-class}
Let $\xi$ be a primitive finite-order Hecke character of conductor $\mathfrak l^{n_0}$. 
Then $\mathfrak J_{f,\xi}$ is a nonzero ideal of $B$.
\end{lemma}

\begin{proof}
Equation~\eqref{twisting:operators:compare} gives
\[
\mathcal R_1(\xi)\eta_{f,1}
=
\frac{G(\xi)}{\Omega_f}\delta_1(f\otimes\xi).
\]
Since $f$ is a normalized eigenform, $a_{f\otimes\xi}(\mathcal O_F)=1$, 
which implies that $\delta_1(f\otimes\xi)\neq 0$.
This proves the assertion together with Lemma~\ref{dualmap:isom}.
\end{proof}
Define the ideal $\mathfrak J_f$ of $B$ by
$\mathfrak J_f:=\sum_{\xi}\mathfrak J_{f,\xi}$,
where $\xi$ runs through the primitive finite-order Hecke characters of conductor $\mathfrak l^{n_0}$. 
Then $\mathfrak J_f$ is nonzero by Lemma~\ref{lem:hecke-transport-twisted-class}.

Now choose $\mathfrak p$ from the positive-density set $\mathcal P_h$.
Let $\mathfrak P$ be a prime of $\overline{\mathbb Z}$ lying above $\mathfrak p$, and set 
$\mathfrak q:=B\cap\mathfrak P$, $k':=B/\mathfrak q$.

\begin{proposition}
\label{prop:residual-nonvanishing-eta-f}
Suppose that $\mathfrak q$ is prime to $\mathfrak J_f$.
Then there exists
$c\in O_{\mathfrak l}^\times$ such that
\[
\mathcal F_{c,1}\eta_{f,1}\ne0
\quad\text{in}\quad H_!^1(Y_1^1(\mathfrak N'),k').
\]
\end{proposition}

\begin{proof}
By assumption on $\mathfrak q$, 
we can choose a primitive finite order character $\xi$ of modulus $\mathfrak l^{n_0}$ such that
$\mathfrak J_{f,\xi}$ is coprime to $\mathfrak q$.
The definition of $\mathfrak J_{f,\xi}$ gives a class $m$ in
$H_1^!(Y_1^1(\mathfrak N'),B)$
such that
\[
m\cap\mathcal R_1(\xi)\eta_{f,1}\not\equiv0\pmod{\mathfrak q}.
\]
Naturality of the pairing implies
\[
\mathcal R_1(\xi)\eta_{f,1}\not\equiv0\pmod{\mathfrak q}.
\]
Since $p\ne\ell$, the Gauss sum $G(\overline\xi)$ is a
$\mathfrak q$-adic unit. Thus the reduction of
\eqref{eq:gauss-inversion-full-fourier-operator} has a nonzero
right-hand side. This proves the assertion.
\end{proof}

As $\mathfrak q$ is prime to $\#E_F$, by the integrality~\eqref{eq:integrality:lvalue},
we can consider the mod $\mathfrak P$ reduction of $\varepsilon_q L_f(\phi\psi)$
for $\phi\in X_n$ and $\psi\in\widehat\Delta$.
For $\varepsilon$ in $E_F$, define $\sigma_\varepsilon$
the automorphism on $\mathbb{Q}(\mu_{\ell^\infty})$ given by
$\zeta\mapsto\zeta^{b(\varepsilon)}$.
Here, $b(\varepsilon)$ is the image of $\varepsilon$ under the isomorphism
$E_F\subset O_\mathfrak{l}^\times\xrightarrow{\;\simeq\;}\mathbb{Z}_\ell^\times$. 
For $\phi\in X_n$, put $\phi^\varepsilon:=\sigma_\varepsilon\circ\phi$,
which is an element of $X_n$. Clearly $\phi^1=\phi$.

\begin{proposition}\label{prop:first-averaged-congruence}
Let $\phi\in X_n$ and $g_0\in O_{\mathfrak l}^{\times}$.
Assume that $\mathfrak p$ is coprime to $\#\Delta$.
If $n\geq 2n_0$ and
\begin{equation}\label{eq:nonvanishing:lvalue}
\varepsilon_qL_f(\phi^\varepsilon\psi)\equiv0\pmod{\mathfrak P}
\end{equation}
for every $\psi\in\widehat\Delta$ and $\varepsilon \in E_F$, then
\begin{equation*}
\Upsilon_1
\left(\frac{g}{\varpi_{\mathfrak l}^n}\right)
\cap
\mathcal F_{c_\phi g_0^{-1},1}\eta_{f,1}
\equiv 0
\pmod{\mathfrak P}
\end{equation*}
for any $g\in g_0\Gamma_{n_0}$.
Here $\Upsilon_1$ is taken at level $\mathfrak N'$.
\end{proposition}

\begin{proof}
Let $g\in g_0\Gamma_{n_0}$. Put
\[
 G':=\operatorname{Gal}\left(
 L(\mu_{\ell^n})/
 L(\mu_{\ell^{n_0}})
 \right).
\]
Fix $\theta\in\{\phi^\varepsilon\}_{\varepsilon\in E_F}$.
The rationality of the algebraic $L$-values (Hida~\cite[\S8]{Hida:1994}) gives
\[
 L_f(\theta\psi)^\tau
 =L_f(\theta^\tau\psi)
 \qquad(\tau\in G').
\]
Apply Corollary~\ref{cor:eigenform-lvalue-automorphic-symbol} to~\eqref{eq:nonvanishing:lvalue}.
Take the average of the resulting equation twisted by $\theta^\tau\psi(g^{-1})$ over $\tau$ in the decomposition
group of $\mathfrak P$ in $G'$ and $\psi\in\widehat\Delta$. 
Lemma~\ref{lem:galois-average-reduction}
replaces the normalized decomposition-group average modulo
$\mathfrak P$ by the average for $G_\theta$. Thus, we obtain
\begin{equation}\label{eq:proof-prop106-character-average}
\begin{aligned}
 &\frac{1}{\#G_\theta}
 \sum_{\sigma\in G_\theta}
 \frac1{\#\Delta}
 \sum_{\psi\in\widehat\Delta}
 \sum_{i=1}^{h_F}\\
 &\qquad\qquad\cdot
 \sum_{a\in(O_{\mathfrak l}/\mathfrak l^nO_{\mathfrak l})^\times}
 \theta^\sigma\psi(aa_i g^{-1})
 \Upsilon_{i}\left(
   \frac{a}{\varpi_{\mathfrak l}^n}
 \right)\cap\eta_{f,i}
 \equiv0\pmod{\mathfrak P}.
\end{aligned}
\end{equation}
Here, the symbols are taken at level $\mathfrak N$.
Proposition~\ref{prop:average:lpower:l} shows that the common support of
the inner sums in \eqref{eq:proof-prop106-character-average} is
\[
 i=1,\qquad
 a=\varepsilon g \gamma\in E_Fg\Gamma_{n-n_0}/\Gamma_n.
\]
The character $\theta$ is trivial on $E_F$, and the automorphic symbols are $E_F$-invariant by
Lemma~\ref{lem:unit-action-automorphic-symbol}. Hence the terms indexed
by $\varepsilon g\gamma$ do not depend on $\varepsilon$.
Writing $\gamma=1+\varpi_{\mathfrak l}^{n-n_0}b$, the resulting
identity is
\begin{equation}\label{eq:proof-prop106-two-additive-sums}
 \#E_F\sum_{b\ {\rm mod}\ \mathfrak l^{n_0}}
 \mathbf e_F\left(
   \frac{bc_\theta}{\varpi_{\mathfrak l}^{n_0}}
 \right)
 \Upsilon_{1}\left(
   \frac{g(1+\varpi_{\mathfrak l}^{n-n_0}b)}
        {\varpi_{\mathfrak l}^{n}}
 \right)\cap\eta_{f,1}
 \equiv0\pmod{\mathfrak P}.
\end{equation}

Put $u:=g/\varpi_{\mathfrak l}^{n}$ and $v_b:=b/\varpi_{\mathfrak l}^{n_0}$.
Using the equality $c_{\phi^\varepsilon}\equiv \varepsilon c_\phi \pmod{\mathfrak l^{n_0}}$,
taking change of variable $b\mapsto (\varepsilon g_0)^{-1} b$,
and summing \eqref{eq:proof-prop106-two-additive-sums} over
\(\theta=\phi^\varepsilon\), indexed by \(\varepsilon\in E_F\), we obtain
\begin{equation}\label{eq:proof-prop106-level-composition}
 \#E_F\sum_{\varepsilon\in E_F}\sum_{b}
 \mathbf e_F\left(
   \frac{bc_\phi g_0^{-1}}{\varpi_{\mathfrak l}^{n_0}}
 \right)
 \mathrm{Tr}_E\Delta_1^*\alpha^*_{u+\varepsilon v_b,1}\eta_{f,1,\mathrm{c}}
 \equiv0\pmod{\mathfrak P}.
\end{equation}
Here, the underlying maps are
\[
 \alpha_{v_b,1}:Y^1(\mathfrak N')\longrightarrow
 Y_1^1(\mathfrak N),
 \qquad
 \alpha_{u,1}:Y^1(\mathfrak N\cap\mathfrak l^{2n})\longrightarrow
 Y_1^1(\mathfrak N').
\]
Note that we can choose matrices satisfying
$\alpha_{u+\varepsilon v_b,1}
=\alpha_{\varepsilon v_b,1}\alpha_{u,1}$
for every $b$ and $\varepsilon\in E_F$.

Let $\eta'$ be a compactly supported lift of $\eta_{f,1}$.
Since $\mathcal E_q$ commutes with the level-changing pullbacks at
$\mathfrak l$ and $\alpha_{v,1}^*\eta'=(\alpha_{v,1}^*\eta)'$
by the naturality, we have
\begin{align*}
\alpha_{u,1}^*(\mathcal F_{h,1}\eta_{f,1})_{\mathrm c}
&=\alpha_{u,1}^*\mathcal E_q
\left(
\sum_{b\ \mathrm{mod}\ \mathfrak l^{n_0}}
\mathbf e_F\left(\frac{bh}{\varpi_{\mathfrak l}^{n_0}}\right)
\sum_{\varepsilon\in E_F}
\alpha_{\varepsilon v_b,1}^*\eta'
\right)\\
&=\sum_{b}
\mathbf e_F\left(\frac{bh}{\varpi_{\mathfrak l}^{n_0}}\right)
\sum_{\varepsilon\in E_F}
\alpha_{u+\varepsilon v_b,1}^*\eta_{f,1,\mathrm c}
\end{align*}
for $h\in O_{\mathfrak l}^\times$.
Taking $h=c_\phi g_0^{-1}$, we conclude that
\eqref{eq:proof-prop106-level-composition} is equal to
\begin{equation*}
 \#E_F\operatorname{Tr}_E\Delta_1^*\alpha_{u,1}^*
 (\mathcal F_{c_\phi g_0^{-1},1}\eta_{f,1})_{\mathrm c}
 \equiv0\pmod{\mathfrak P}.
\end{equation*}
Then the assertion follows from the definition of automorphic symbol.
\end{proof}

Set $\mathfrak N=\mathfrak l^e$ for an integer $e>0$. Assume further that $2n_0\geq e$. 
Then we can apply Theorem~\ref{full:rank:vertical} at level \(\mathfrak N'\). 
Let \(n^\circ\) and \(d^\circ\) be the resulting integers. 
Put
\[
\mathfrak V:=n^\circ+d^\circ\mathbb Z_{\geq0},
\qquad
X^\circ:=\bigsqcup_{n\in\mathfrak V}X_n,
\qquad
\mathcal Q_c:=H_1^!(Y_1^1(\mathfrak N'),\mathcal O_K)/M_{n^\circ}^1(c).
\]
After increasing \(n^\circ\) if necessary, assume that $n^\circ\geq2n_0$. 
Define the obstruction ideal $\mathfrak J$ of $B$:
\[
\mathfrak J:=\varepsilon_q\#\Delta\bigcap_{[g_0]\in (O_{\mathfrak l}/\mathfrak l^{n_0}O_{\mathfrak l})^\times} 
\operatorname{Fitt}_{A}^0(\mathcal Q_{g_0})\cdot\mathfrak J_f.
\]
This ideal is independent of $\mathfrak p$ and nonzero 
since $\varepsilon_q$ and the Fitting ideals are nonzero, and $\mathfrak J_f$ is nonzero
by Lemma~\ref{lem:hecke-transport-twisted-class}.
Now we prove Theorem~\ref{thm:main:introduction}:

\begin{theorem}[Non-vanishing mod $p$ of $L$-value]
\label{thm:mod-p-nonvanishing}
Assume that $\mathfrak q$ is coprime to $\mathfrak J$.
Then, for every $\phi\in X^\circ$, there exist
$\varepsilon\in E_F$ and $\psi\in\widehat\Delta$ such that
\[
 L_f(\phi^\varepsilon\psi)
 \not\equiv0\pmod{\mathfrak P}.
\]
\end{theorem}

\begin{proof}
Fix $\phi\in X^\circ$. Then $\phi\in X_n$ for some $n\in\mathfrak V$.
Assume the contrary, i.e.,
\[
L_f(\phi^\varepsilon\psi)\equiv0\pmod{\mathfrak P}
\qquad(\varepsilon\in E_F,\ \psi\in\widehat\Delta).
\]
Fix $g_0\in O_{\mathfrak l}^\times$.
Proposition~\ref{prop:first-averaged-congruence} gives
\[
\Upsilon_1\left(\frac{g}{\varpi_{\mathfrak l}^n}\right)
\cap
\mathcal F_{c_\phi g_0^{-1},1}\eta_{f,1}\equiv 0\pmod{\mathfrak P}
\]
for any $g\in g_0\Gamma_{n_0}/\Gamma_n$. Thus the naturality of the pairing,
Lemma~\ref{dualmap:isom} for $A=k'$ and 
Theorem~\ref{prop:vertical-full-rank-finite-field} imply
\[
\mathcal F_{c_\phi g_0^{-1},1}\eta_{f,1}\equiv 0\pmod{\mathfrak P}
\]
for every $g_0\in O_{\mathfrak l}^\times$.
This contradicts Proposition~\ref{prop:residual-nonvanishing-eta-f}.
\end{proof}

\begin{remark}
Lemma~\ref{lem:galois-average-reduction}, which is essential to the
proof of Theorem~\ref{thm:mod-p-nonvanishing}, requires the $\mathfrak p$-dependent bound on $n_0$.
On the other hand, our main theorem needs the assumption $\mathfrak p\nmid\mathfrak J$,
where \(\mathfrak J\) depends on \(n_0\).
If \(n_0\) were chosen for each \(\mathfrak p\), then
\(\mathfrak J\) would also vary with \(\mathfrak p\), 
making the condition $\mathfrak p\nmid\mathfrak J$ in the statement of 
Theorem~\ref{thm:mod-p-nonvanishing} circular.
The set $P_h$ with cutoff \(h>h_0\) removes this circular dependence:
\(\mathcal P_h\) has positive density, and then fix \(n_0\geq h\). 
Thus Lemma~\ref{lem:galois-average-reduction} applies uniformly to every
\(\mathfrak p\in P_h\), while \(\mathfrak J\) is fixed independently
of \(\mathfrak p\). 
\end{remark}

\begin{remark}\label{rem:positive:density}
Let $\mathfrak p\in \mathcal P_h$.
Our main theorem is true for any $\mathfrak P$ lying above $\mathfrak p$
satisfying $B\cap\mathfrak P\nmid\mathfrak J$.
Note that the density of $\mathcal P_h$ can be arbitrarily close to $1$ as $h$ varies by Chebotarev's theorem.
\end{remark}

\end{document}